%% file: Dissertation_arxiv2.tex
\documentclass[headsepline,12pt,a4paper]{report}
\usepackage{amsmath,amssymb,amsfonts,amsthm,exscale,calc,hyperref}
\usepackage[latin1]{inputenc}
\usepackage{graphicx}
\usepackage{parskip}
\usepackage{floatflt}
\usepackage{picins}
\usepackage{mathrsfs}
\usepackage{vmargin}
\usepackage{dsfont}
\usepackage[ngerman,english]{babel}

\theoremstyle{theorem}
\newtheorem{Thm}{Theorem}
\newtheorem{thm}{Theorem}\numberwithin{thm}{chapter}
\newtheorem{lemma}[thm]{Lemma}
\newtheorem{prop}[thm]{Proposition}
\newtheorem{cor}[thm]{Corollary}
\newtheorem{question}{Question}

\theoremstyle{definition}

\newtheorem{Ex}[thm]{Example}

\newtheorem{Def}[thm]{Definition}

\newcommand{\Hm}[1]{\leavevmode{\marginpar{\tiny%
$\hbox to 0mm{\hspace*{-0.5mm}$\leftarrow$\hss}%
\vcenter{\vrule depth 0.1mm height 0.1mm width \the\marginparwidth}%
\hbox to 0mm{\hss$\rightarrow$\hspace*{-0.5mm}}$\\\relax\raggedright
#1}}}

\newcommand{\C}{{\mathbb C}}

\newcommand{\EE}{{\mathbb E}}
\newcommand{\h}{{\mathbb H}}

\newcommand{\N}{{\mathbb N}}
\newcommand{\PP}{{\mathbb P}}
\newcommand{\R}{{\mathbb R}}
\newcommand{\TT}{\mathbb{T}}
\newcommand{\Z}{{\mathbb Z}}

\newcommand{\A}{{\mathcal A}}
\newcommand{\E}{{\mathcal E}}

\newcommand{\T}{{\mathcal T}}
\newcommand{\V}{{\mathcal V}}
\newcommand{\W}{{\mathcal W}}

\newcommand{\Sp}{{\mathbb S}}

\newcommand{\dd}{{\partial}}

\newcommand{\Vis}{\mathrm{Vis}}
\newcommand{\Leb}{\mathrm{Leb}}
\newcommand{\Lip}{\mathrm{QC}}
\newcommand{\clos}{{\mathrm {clos}\,}}
\newcommand{\inn }{{\mathrm {int }}}

\newcommand{\qand}{{\quad\mathrm {and}\quad}}
\newcommand{\qqand}{{\qquad\mathrm {and}\qquad}}
\newcommand{\dist}{{\mathrm{dist}}}
\newcommand{\diam}{{\mathrm{diam}}}
\newcommand{\id}{{\mathrm{id}}}

\newcommand{\ka}{{\kappa}}
\newcommand{\al}{{\alpha}}
\newcommand{\be}{{\beta}}
\newcommand{\de}{{\delta}}
\newcommand{\De}{{\Delta}}
\newcommand{\eps}{{\varepsilon}}
\newcommand{\gm}{{\gamma}}
\newcommand{\Gm}{{\Gamma}}
\newcommand{\ph}{{\varphi}}
\newcommand{\lm}{{\lambda}}

\newcommand{\te}{{\theta}}
\newcommand{\om}{{\omega}}
\newcommand{\Om}{{\Omega}}
\newcommand{\si}{{\sigma}}

\newcommand{\ap}[1]{\left( #1\right)}
\newcommand{\mo}[1]{\left\vert #1\right\vert}
\newcommand{\na}[1]{\left\vert #1\right\vert_{\arg}}
\newcommand{\ma}[1]{\vert #1\vert_{\arg}}
\newcommand{\no}[1]{\left\Vert #1\right\Vert}

\newcommand{\set}[1]{\left\{ #1\right\}}
\newcommand{\ip}[1]{\langle #1\rangle}

\newcommand{\ab}[1]{\left( #1\right)}
\newcommand{\ac}[1]{\left\{ #1\right\}}
\newcommand{\as}[1]{\langle #1\rangle}
\newcommand{\av}[1]{\left\vert #1\right\vert}
\newcommand{\am}[1]{\vert #1\vert_{\arg}}
\newcommand{\aM}[1]{\left\vert #1\right\vert_{\arg}}
\newcommand{\aV}[1]{\left\Vert #1\right\Vert}

\newcommand{\ov}[1]{\overline{#1}}
\newcommand{\ow}[1]{\widetilde{#1}}
\newcommand{\oh}[1]{\widehat{#1}}

\let\Im\undefined
\let\Re\undefined
\DeclareMathOperator{\Im}{Im}
\DeclareMathOperator{\Re}{Re}
\DeclareMathAlphabet{\mathpzc}{OT1}{pzc}{m}{it}

\begin{document}
\input{title_arxiv.tex}

\input{text.tex}

\input{cv.tex}
\end{document}

%% file: title_arxiv.tex
\newcommand{\HRule}{\rule{\linewidth}{0.5mm}}

\begin{titlepage}
\begin{center}
\HRule \\[0.4cm]
{ \Huge  \bfseries On the spectral theory of  operators on trees}\\[0.4cm]

\HRule \\[2.5cm]
\includegraphics[width=10cm]{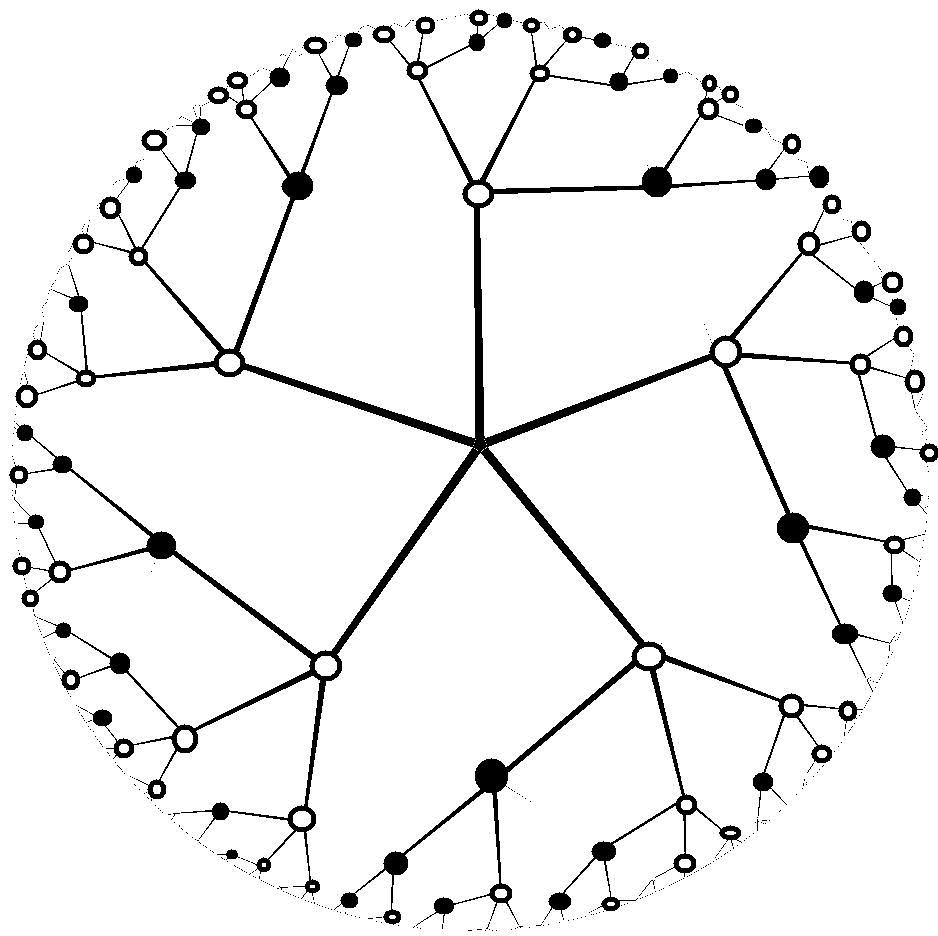}\\[2.5cm]
\bfseries{\huge Matthias Keller}\\[1cm]
\bfseries{\Large PhD Thesis}\\[.5cm]
\bfseries{\Large submitted September 3rd 2010}\\[.2cm]
\bfseries{\Large defended December 17th 2010}\\[1.5cm]
\vfill
\end{center}
\end{titlepage}

\abstract
We study a class of rooted trees with a substitution type structure. These trees are not necessarily regular, but exhibit a lot of symmetries. We consider nearest neighbor operators which reflect the symmetries of the trees. The spectrum of such operators is proven to be purely absolutely continuous and to consist of finitely many intervals.  We further investigate stability of the absolutely continuous spectrum under perturbations by sufficiently small potentials.
On the one hand, we look at a class of deterministic potentials which include radial symmetric ones. The absolutely continuous spectrum is stable under sufficiently small perturbations of this type if and only if the tree is not regular. On the other hand, we study random potentials. In this case, we prove stability of absolutely continuous spectrum for both regular and non regular trees provided the potentials are sufficiently small.

%% file: text.tex
\pagenumbering{roman}
\cleardoublepage
\tableofcontents
\chapter*{Acknowledgements}
The first thanks belong undoubtedly to Daniel Lenz. I am more than fortunate to have an advisor, teacher and friend like him. The uncountable hours of explanations and discussion of ideas have contributed a major deal to what I know about mathematics today. I am very thankful for all the encouragement, advice and inspiration during the last years.

The work of this thesis was done partially at the Technical University of Chemnitz,  Princeton University and Friedrich Schiller University Jena. The years  in Chemnitz are of great importance to me. I am grateful to the group of Peter Stollmann for the friendly and supportive atmosphere and the profound education I received there. I would like to thank Michael Aizenman and the Princeton Department of Mathematics for their generous support and hospitality. Especially,  I am indebted to Simone Warzel for the numerous invaluable discussions and explanations on various topics during my time in Princeton. I am also very thankful to the group in Jena. In particular,  completing this work would not have been half as much fun without my office mates and friends Sebastian Haeseler and Felix Pogorzelski.

During my research I enjoyed the hospitality of various institutions. In particular,  I want to thank Norbert Peyerimhoff (Durham University), J\'ozef Dodziuk (Graduate Center CUNY, New York), Florian Sobieczky (Technical University of Graz) and Rados{\l}aw Wojciechowski (Group of Mathematical Physics, Lisbon) for the invitations and the wonderful time that I spent at their places.

During the first year of my graduate studies, I was financially supported by the German Research Council (DFG) and, for the rest of the time, by the German Business Foundation (SDW). I appreciate this support as it gave me the freedom to concentrate on research.

I am very grateful to my parents who always supported me strongly. I thank them for their love, encouragement and their outstanding example that you can achieve more than you ever thought by working hard.

My son Elliott is an abundant source of joy to me and a wonderful reminder that there is more to life than mathematics. Finally,  I want to express my love and gratitude to my wife Yvonne. Without her support and love during all these years this work would not have been possible. It is wonderful being married to her and I cannot imagine a life without her.

\thispagestyle{empty}
\pagenumbering{arabic}
\setcounter{page}{0} 
\setcounter{chapter}{-1}

\chapter{Introduction}
\begin{quote}
\begin{flushright}
\scriptsize{And out of the ground made the Lord God to grow
every \emph{tree} that is pleasant to the sight, and good for food;
the \emph{tree} of life also in the midst of the garden, and the \emph{tree} of knowledge of good and evil.} Genesis 2:9\end{flushright}
\end{quote}

The spectral theory of graphs has a long tradition and to this day  it is a vibrant branch of mathematics. The motivations to study nearest neighbor operators on graphs, often referred to as combinatorial Laplacians, reach from random walks over discrete spectral geometry to mathematical physics.

There is a large amount of literature found on questions concerning the return probability of random walks. We refer to \cite{Woe2} and references therein. A great survey on many classical results of the spectral theory of certain operators on graphs is given in \cite{MW}. A question of particular importance concerns the bottom of the spectrum. There are global geometric invariants such as isoperimetric constants and volume growth which can be used to obtain estimates for the bottom of the spectrum, see \cite{Do,DKe,DKa,Fu1,Fu2}. Since these geometric invariants are not always explicitly computable, one is also interested in the implications obtained by local quantities. For instance, curvature is a powerful tool to study planar graphs, see  \cite{Fu2,Hi,Kel,KLPS,KP,Woe1,Woj}.
However,  to study spectral properties in greater detail, such as the type of the spectral measures, one needs to take a much closer look at the particular structure of the graph. At the moment there are two classes of graphs which seem to be accessible for such a detailed spectral analysis.

One class consists of abelian coverings of finite graphs where one can apply Fourier/ Bloch theory to prove that the corresponding operators have absolutely continuous spectrum. For a recent development see \cite{HN} and references therein. The most prominent example is $\Z^{^d}$ for which the spectral theory was known long before. Moreover, in \cite{MRT} ladder graphs are studied for which the analysis can be more or less reduced to the situation of $\Z$.

The other class of examples are trees. It is mathematical folklore that the nearest neighbor Laplacian on a regular tree has purely absolutely continuous spectrum. Moreover, the spectrum consists of one interval depending on the branching number of the graph. In \cite{Ao}, a characterization for the existence of eigenvalues is given for operators on trees which are invariant under a group action. On the other hand, there are classes of trees for which singular spectrum has been proven. For example,  the Laplacian on sparsely branching, radial symmetric trees exhibits singular continuous spectrum, see \cite{Br}. In  \cite{BF} it is shown that Laplacians on radial symmetric trees whose branching number does not become eventually periodic have pure singular spectrum. The proofs of these results depend strongly on reducing the operators to a one dimensional situation. Another model of a tree which has pure point spectrum is the canopy tree, studied in \cite{AW}. From a particular viewpoint, the canopy tree is the limit of a regular tree truncated at the $n$-th sphere as $n$ tends to infinity.

The class of trees which we study in this work have a substitution type structure. Although still rich in symmetry,  the trees in this class are, in general, distinguished by a loss of regularity. In general, they are neither regular nor radial symmetric. One motivation to study such trees is that they appear in the context of regular hyperbolic tessellations  as spanning trees.

In the first step, we study nearest neighbor operators which exhibit the substitution type structure of the underlying tree. We call them label invariant operators. The spectrum of these operators turns out to be purely absolutely continuous and to consist of finitely many intervals, see Theorem~\ref{main1}. In the second step, we analyze how small perturbations of these operators by potentials effect the spectral properties.

Before describing the perturbation results in more detail we want to explain the motivation which comes from mathematical physics. There such operators are used to model quantum mechanical phenomena of solid states. In particular, the structure of the graph describes the order of atoms in the solid state and the spectral properties of the operators stand in close correspondence to its conductivity properties. The first step is to study systems which exhibit a lot of symmetries. In the second step, one is interested in models where the symmetric structure is broken.

Particular attention has been attracted by models where the loss of symmetry is due to some randomness. In this way the model becomes totally disordered but the symmetries remain in a statistical sense. A question which arises naturally is how much of the spectral properties are preserved under such small random perturbations. Operators of this kind have been widely studied for $\Z^{d}$ (or $\R^{d}$ in the continuous case). See the monographs \cite{CFKS,CL,Sto} for details and further reference.

A meta theorem in this context is that a lot of symmetry corresponds to absolutely continuous spectrum of the operator and conductance while large disorder corresponds to point spectrum and the behavior of an insulator. For operators on $\Z$, it was proven that an arbitrary small perturbation by a random potential turns the absolutely continuous spectrum of the operator completely into pure point spectrum  \cite{CKM,KS}.  For higher dimensions, i.e., $\Z^{d}$, $d\geq2$, this behavior is only proven for large disorder or energy regimes close to the band edges of the spectrum.

It is expected by physicists that, for $d=2$, the absolutely continuous spectrum turns immediately into point spectrum as soon as disorder occurs. On the other hand, for $d\geq3$, one expects that parts of the absolutely continuous spectrum remain stable for small disorder. This phenomena is referred to as the extended states conjecture or stability of absolutely continuous spectrum. But there are no rigorous results known so far for this questions on $\Z^{d}$.

However,  in \cite{Kl1} exactly this behavior expected for $\Z^{d}$, $d\geq3$, was observed and proven rigorously for regular trees, see also \cite{Kl2,Kl3}. Later in \cite{ASW1,FHS2} similar results were obtained by different methods of proof, see also \cite{ASW2,ASW3, FHS1,FHS3}. In some sense, a tree can be considered as an infinite dimensional analogue of the euclidean lattice. But the absence of cycles in the graph makes the analysis of the spectrum much more accessible. However,  the analysis of \cite{ASW1,FHS2,Kl1} uses strongly the regularity of the tree. A first step away from regularity was taken in \cite{Hal1,Hal2} where additional vertices are inserted into edges of  a regular tree.  There, extending the methods of \cite{FHS2} stability of pure absolutely continuous spectrum is proven for this model.

We generalize the geometric setting to a much larger class of trees. In contrast to regular trees, or the model of \cite{Hal1,Hal2}, the statement that the underlying unperturbed operator has pure absolutely spectrum is neither known nor immediate but has to be proven. Having this established, see Theorem~\ref{main1}, we then show stability of the absolutely continuous spectrum. This is done, firstly, for certain deterministic potentials, see Theorem~\ref{main2}, and, secondly, for random potentials, see Theorem~\ref{main3}.

In some sense, the absolutely continuous spectrum of our trees is even more stable than the one of regular trees. For regular trees, it is known that the absolutely continuous spectrum is  generically destroyed completely by radial symmetric potentials, see \cite[Appendix A]{ASW1}.
In clear contrast to this,   for non regular trees of our class, we prove stability of  absolutely continuous spectrum under radial symmetric potentials which are  sufficiently small, see Theorem~\ref{main2}. In this sense the loss of symmetry (i.e., the loss of regularity of the tree) stabilizes the absolutely continuous spectrum.

Secondly, we consider perturbations by random potentials. This extends the major statements of \cite{Kl1,ASW1,FHS2} to a much larger class of trees. We prove that arbitrary fixed parts of the absolutely continuous spectrum are stable almost surely for sufficiently small potentials. While the overall strategy of our approach owes to  \cite{FHS1,FHS2}, the actual steps of our proof are rather different. This reflects also in the results. For instance,  all our estimates are explicit, which could be used to calculate a lower bound on the magnitude of the perturbation parameter. Moreover, we obtain a certain continuity  for the density of the spectral measures in the absolutely continuous regime in the perturbation parameter.

The text is structured as follows. In Chapter~\ref{s:main} the class of trees and the operators of our interest are introduced. We state the results which are proven in the later chapters and discuss several examples. Chapter~\ref{c:basicconcepts} surveys the basic concepts of our analysis. Although many of these concepts are well known, we give full proofs for the sake of completeness. Moreover,  at some places we generalize known results for bounded operators on trees to unbounded operators. In Chapter~\ref{c:freeoperator} we prove two of the main results. This concerns the spectral theory of unperturbed operators and small perturbations by the class of deterministic potentials.  These results are submitted for publication in \cite{KLW}. In Chapter~\ref{c:main3} the issue of small random perturbations is tackled. 


\chapter{Models and Results}\label{s:main}
\begin{quote}
\begin{flushright}
\scriptsize{The fruit of the righteous is a \emph{tree} of life.} Proverbs 11:30 \end{flushright}
\end{quote}

In this chapter we provide the definitions and present the main results of this work. In Section~\ref{s:trees} we introduce the trees of interest, which are constructed by rules encoded in a substitution matrix. These trees exhibit a lot of symmetry although they are not necessarily regular trees. In Section~\ref{s:LIOp} we define operators which are compatible with the geometric structure. We will call them \emph{label invariant}. Their spectrum is purely absolutely continuous and consists of finitely many intervals. Moreover, the absolutely continuous spectrum of the underlying label invariant operators is stable for certain sufficiently small perturbations on subsets. These perturbations are, on the one hand, radial label symmetric potentials, see Section~\ref{s:RLSymPot}, and, on the other hand, random potentials, see Section~\ref{s:RandPot}.

\section{Trees with a substitution type structure}\label{s:trees}
A graph consists of a pair $(\V,\E)$. The set $\V$, called the \emph{vertex set}, is at most countable. The set $\E$, called the \emph{edge set}, consists of subsets of $\V$ with exactly two elements. If $\{x,y\}\in \E$ for two vertices $x,y\in\V$ we write $x\sim y$ and call $x$ and $y$ \emph{adjacent} or \emph{neighbors}. A \emph{path} of \emph{length} $n$ is a subset of $n$ distinct vertices $\{x_0,\ldots,x_n\}$, $n\in\N$, such that $x_{k-1}\sim x_{k}$ for $k=1,\ldots,n$. A graph is called \emph{connected} if every two vertices can be joined by a finite path. The distance $d(x,y)$ between two vertices $x,y\in \V$ in a connected graph is the smallest number $n$ such that $x$ and $y$ can be connected by a path of length $n$. A path $\{x_0,\ldots,x_n\}$ which additionally satisfies $x_0\sim x_n$ is called a \emph{cycle}. A connected graph without cycles is called a \emph{tree}. A tree $\T=(\V,\E)$ with a distinguished vertex $o\in \V$ is called a \emph{rooted tree} and $o$ is called the \emph{root}. We denote a rooted tree by the pair $(\T,o)$.
In a rooted tree the vertices can be ordered according to \emph{spheres}, i.e., the distance $|\cdot|=d(o,\cdot)$ to the root. A vertex in the $(n+1)$-th sphere which is connected to $x$ in the $n$-th sphere is called a \emph{forward neighbor} of $x$. For a vertex $x$ the \emph{forward tree} $\T_x=(\V_x,\E_x)$ is the subgraph of $\T=(\V,\E)$ for which every path connecting the root of $\T$ with a vertex in $\T_x$ passes $x$, i.e., $y\in\T_x$ whenever $d(x,y)=|y|-|x|$.

We now introduce the class of trees which will be the subject of our analysis: Let $\A$ be a finite set, whose elements will be called \emph{labels}. Furthermore, let a matrix $M$ be given
\begin{align*}
M : \mathcal{A}\times \mathcal{A}\to \N_0, \quad (j,k) \mapsto M_{j,k},
\end{align*}
which we refer to as the \emph{substitution matrix}.
To each label $j\in\A$ we construct inductively a tree $\T=\T(M,j)$ with the vertex set $\V=\V(M,j)$ and the edge set $\E=\E(M,j)$. Each tree comes with a labeling of the vertices, i.e., a function
$$a:\V\to\A,$$
assigning to each vertex its label as  follows: The root of the tree gets the label $j$. Each vertex with label $k\in\A$ of the $n$-th sphere is joined by edges to $M_{k,l}$ vertices of label $l$ of the $(n+1)$-th sphere. 

Let us give some examples.
\medskip

\begin{Ex}\label{ex:M} (1.) Assume $\A$ consists of only one element $j$ and $M_{j,j}=k$ for some $k\ge1$. Then, $\T=\T(M,j)$ is a $k$-regular tree, i.e., a tree where each vertex has exactly $k$ forward neighbors.

(2.) Let $\A=\{1,2\}$ and $M=\left(
\begin{array}{cc}
2 & 1 \\
1 & 1 \\
\end{array}
\right).
$
Figure~\ref{f:tree} illustrates the tree $\T=\T(M,2)$.
\begin{figure}[!h]
\centering
\scalebox{0.3}{\includegraphics{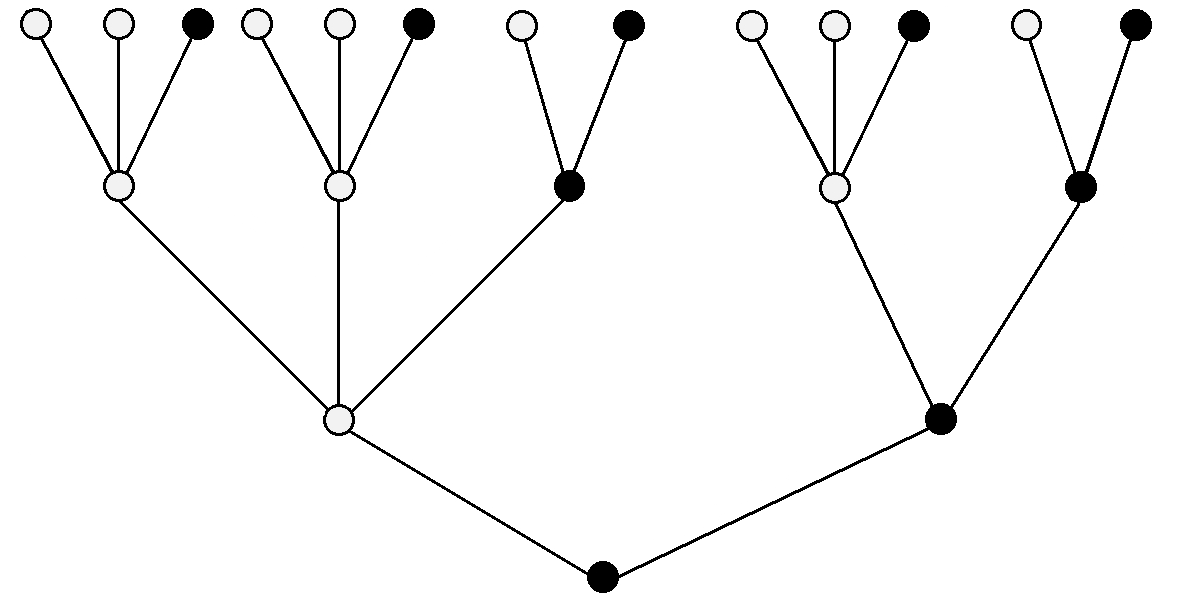}}
\caption{An example of a tree constructed by a substitution matrix.}\label{f:tree}
\end{figure}
\end{Ex}

\begin{figure}[!h]
\centering
\scalebox{0.3}{\includegraphics{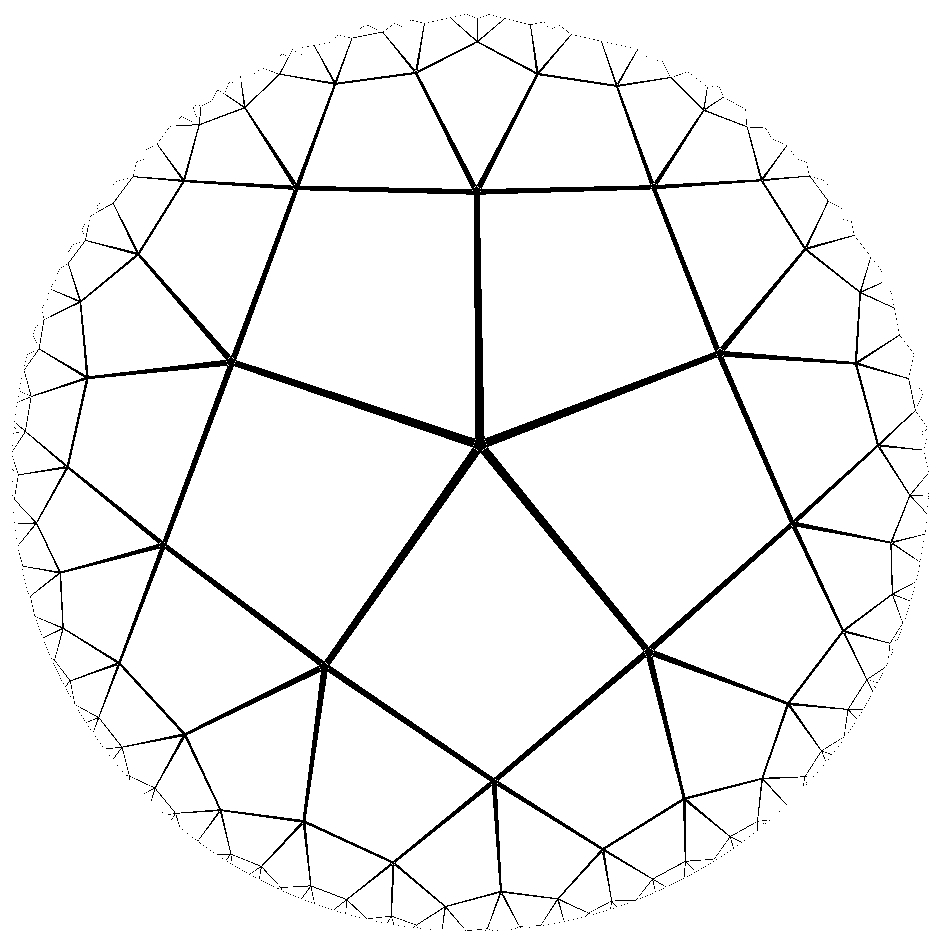}}
\scalebox{0.3}{\includegraphics{tess45_spanntree.jpg}}
\caption{The regular hyperbolic (4,5)-tessellation and a spanning tree which is constructed by a substitution matrix.}
\label{f:tess}
\end{figure}
Let us mention that such trees occur in a very natural way as spanning trees of regular tessellations of the hyperbolic plane. We want to illustrate this connection by giving a particular example in Figure~\ref{f:tess}. There, the $(4,5)$-tessellation is pictured on the left hand side. The pair $(4,5)$ means that five squares meet in any vertex.
Now, consider five copies of the tree $\T=\T(M,1)$ introduced in Example~\ref{ex:M}~(2.) glued together by connecting the roots of the copies to an additional vertex. The resulting tree is now a spanning tree of the $(4,5)$-tessellation which is illustrated at the right hand side of Figure~\ref{f:tess}.

For the purposes pursued in this text we need some additional assumptions on $M$. First of all, we want to exclude the one-dimensional case:
\begin{itemize}
\item [(M0)] If $\A$ contains only one element, then $M>1$,\emph{ (non one-dimensional)}.
\end{itemize}
The assumption excludes the case of the graph with vertex set $\N$ and edges $\{n,n+1\}$, $n\in\N$. Although this assumption will not be necessary for the results on the unperturbed operators, it is essential for the perturbation theory. We will also assume the following:
\begin{itemize}
\item [(M1)] $M_{j,j} \geq1$ for all $j\in \A$ \emph{(positive diagonal)}.
\item [(M2)] There exists $n=n(M)\in\N$ such that $M^n$ has positive entries \emph{(primitivity)}.
\end{itemize}

There are very natural geometric interpretations of (M1) and (M2). Assumption (M1) guarantees that each vertex has a vertex of its own label as a forward neighbor. On the other hand, (M2) implies that in the forward tree of every vertex one can find vertices of every label.

Note that the assumptions (M0), (M1), (M2) imply that the vertex degree of every vertex is larger or equal to three. Suppose the opposite, which implies that there is a vertex which has only one forward neighbor. Then, by (M1) the only forward neighbor has the same label. By (M2) there cannot be any other labels than this one and we are in the one dimensional situation which is excluded by (M0).

We want to discuss (M1) and (M2) by giving some examples and counter examples.
\medskip

\begin{Ex}\label{ex:M2} (1.) If $M$ has strictly positive entries, i.e., $M:\A\times\A\to\N$, then $M$ satisfies (M1) and (M2). This is in particular true for Example~\ref{ex:M}~(1.) and (2.)  above as well as for the spanning tree of Figure~\ref{f:tess}.

(2.) The trees investigated in \cite{Hal1,Hal2} can also be constructed by a substitution matrix. They do satisfy (M2) but not (M1). For the precise definition we refer the reader to \cite{Hal1,Hal2}. Loosely speaking, they are obtained from a regular tree by adding additional vertices within the existing edges.
We give the most prominent example of this class, which is the Fibonacci tree. The Fibonacci tree can be constructed by $\A=\{1,2\}$ and
$$M=\left(
\begin{array}{cc}
1 & 1 \\
1 & 0 \\
\end{array}
\right).
$$
Clearly, (M1) is not satisfied but since
\begin{align*}
M^2=\left(
\begin{array}{cc}
2 & 1 \\
1 & 1 \\
\end{array}
\right),
\end{align*}
we have (M2). The tree constructed by $M^2$ is the one of Example~\ref{ex:M}~(2.) and  Figure~\ref{f:tree} which yields the spanning trees illustrated in Figure~\ref{f:tess}. Sometimes this tree is also referred to as the Fibonacci tree or the two step Fibonacci tree.

(3.) In the literature on random walks and related topics (see for example \cite{Kro,KT,NW,Mai,Tak}) self similar graphs are widely studied. Our class of trees is there often called trees of finite cone type (sometimes known as periodic trees). Another interesting class can be described as follows: Let $\T=(\V,\E)$ be a tree. Two vertices $x,y\in\V$ are called equivalent if there is a graph isomorphism (i.e., a bijective map which preserves the adjacency relation) mapping $x$ to $y$. Now, assume that there are only finitely many equivalence classes. In particular, this means that $\T$ admits only finitely many rooted trees $(\T,x)$, ${x\in\V}$  up to graph isomorphisms. In this case, $\T$ can be constructed by a substitution matrix, but does not satisfy (M2).
\end{Ex}

\section{Label invariant operators}\label{s:LIOp}
We next turn to the operators which are compatible with the tree structure. Let $\T=(\V,\E)$ be a tree $\T(M,j)$ given by a substitution  matrix $M$ on a finite set of labels $\A$ satisfying (M0), (M1), (M2) and whose root carries the label $j\in\A$. Moreover, let $\nu:\V\to(0,\infty)$ be a \emph{label invariant} measure, i.e., for $x,y\in\V$ with $a(x)=a(y)$ one has $\nu(x)=\nu(y)$. We consider the Hilbert space of $\nu$-square-summable functions on the vertices
\begin{align*}
\ell^2(\V,\nu)=\{\ph:\V\to \C\mid \sum_{x\in\V}|\ph(x)|^2\nu(x)<\infty\}.
\end{align*}
For $\nu\equiv1$ we simply write $\ell^2(\V)$. We study self adjoint operators $T:\ell^2(\V,\nu)\to \ell^2(\V,\nu)$ acting as
\begin{align}\label{e:T}
(T \ph)(x) =\sum_{x\sim y} t(x,y) \ph (y) +w(x) \ph(x),
\end{align}
where $t:\V\times\V\to\C$ and $w:\V\to\R$ are functions such that $t$ respects the adjacency relation and is symmetric with respect to $\nu$, i.e., for all $ x, y \in \V $
\begin{itemize}
\item [(T0)] $t(x,y)\neq0$ if and only if $x\sim y$ \emph{(adjacency)},
\item [(T1)] $ t(x,y)\nu(x) = \overline{t(y,x)}\nu(y)$ \emph{(symmetry)},
\end{itemize}
(here the overbar denotes complex conjugation) and $t$ and $w$ are \emph{label invariant}, i.e., there are
a matrix $ (m_{j,k})_{j,k}$ and a vector $ (m_{j})_{j\in\A}$ on the labels such that for all $ x, y \in \V $
\begin{itemize}
\item [(T2)]
$m_{a(x),a(y)}= | t(x,y) |^2 \, M_{a(x),a(y)}$ and $m_{a(x)}=w(x)$ \emph{(label invariance)}.
\end{itemize}
The multiplication by $ M $ in the first term is for convenience further down. Note that the positive diagonal assumption (M1) and the adjacency assumption (T0) imply that $m_{j,j}>0$ for all $j\in\A$. \medskip

\begin{Def}
We call a linear operator $T$ on $\ell^{2}(\V,\nu)$ acting as \eqref{e:T} and satisfying (T0), (T1) and (T2) a \emph{label invariant operator} on $\T=\T(M,j)$.
\end{Def}

Obviously, a label invariant operator is always bounded. We want to discuss the most important examples.
\medskip

\begin{Ex}\label{ex:operators} Let $\deg:\V\to \N$ be the function which associates to a vertex $x\in\V$ the vertex degree, i.e., the number of edges emanating from $x$. Moreover, let  $\de_o$ be the function which equals one at the root $o$ of the tree and zero elsewhere.

We discuss three combinatorial versions of the Laplace operator on a tree. In order to ensure that the operators are  label invariant we add a boundary condition in form of a potential at the root. This is often referred to as Dirichlet boundary conditions.

(1.) Let $\nu\equiv 1$, $t(x,y)=-1$ for all $x\sim y$ and $w=\deg+\de_o$. Then, $T$ is the nearest neighbor Laplacian $\De$ on $\ell^2(\V)$ with Dirichlet boundary condition at the root $o$ given  by
$$(\De\ph)(x)=\sum_{y\sim x}(\ph(x)-\ph(y))+\ph(o)\de_o(x).$$
The operator is widely used in mathematical physics. The operator $\De$ is positive and bounded by the constant $2\max_{x\in\V}\deg(x)=2(1+\max_{j\in\A}\sum_{k\in\A}M_{j,k})$.

(2.) Let $\nu=\deg+\de_o$,
$t(\cdot,y)=-1/\nu(\cdot)$ for $y\in\V$ and $w\equiv1$. Then, one obtains for $T$ the normalized nearest neighbor Laplacian $\ow \De$ on $\ell^2(\V,\nu)$ with Dirichlet boundary condition at the root $o$ given by
\begin{equation*}
(\ow\De\ph)(x)=\frac{1}{\nu(x)}\sum_{y\sim x}(\ph(x)-\ph(y))+\frac{\ph(o)}{\nu(o)}\de_o(x).
\end{equation*}
The operator is often used in the context of discrete spectral geometry and random walks. The operator $\ow \De$ is positive, bounded and its spectrum is contained in $[0,2]$.

(3.) Let $\nu\equiv 1$, $t(x,y)=-1$ if $x\sim y$ and $w\equiv0$. This gives an operator on $\ell^2(\V)$ called the adjacency matrix via
\begin{equation*}
(A\ph)(x)=\sum_{y\sim x}\ph(y).
\end{equation*}

\end{Ex}

We now state our first main result which concerns the spectrum of label invariant operators.
\medskip

\begin{Thm}\label{main1}(Label invariant operators.) Let $T$ be a label invariant operator. Then, the spectrum of $T$ consists of finitely many intervals and is purely absolutely continuous.
\end{Thm}

The proof of this theorem will be given in  Chapter~\ref{c:freeoperator}.
The only case excluded by (M0) is the one of Jacobi operators on $\N$ with constant diagonal and constant modulus on the off diagonal. For this case, it is well known that the theorem remains true. Thus, the theorem remains true if we drop (M0).

However,  the following counter example shows that the statement of the theorem becomes  false  in general if we drop assumption (M1).\medskip

\begin{Ex} Let $\A=\{1,2\}$,
$$M=\left(
\begin{array}{cc}
0 & 2 \\
1 & 0 \\
\end{array}
\right)
$$
and consider $\T=(M,1)$. Then, the adjacency matrix $A$ on $\T$ discussed in Example~\ref{ex:operators}~(3.) has an eigenfunction to the eigenvalue $0$. The eigenfunction vanishes on the odd spheres and takes values $+\frac{1}{2^n}$ and $-\frac{1}{2^n}$ on the $2n$-th spheres.
\end{Ex}

\section{Radial label symmetric potentials}\label{s:RLSymPot}
We introduce a class of potentials on $\T$ which are symmetric with respect to the labeling in each sphere.\medskip

\begin{Def}(Radial label symmetric potentials)
We say a function $v:\V\to[-1,1]$ is \emph{radial label symmetric} if $a(x)=a(y)$ and $|x|=|y|$ implies $v(x)=v(y)$ for all $x,y\in \V$. We define
$$\W_{\mathrm{sym}}(\T):=\{v:\V\to[-1,1]\mid \mbox{$v$ is radial label symmetric}\}.$$
\end{Def}
By the symmetry, a potential $v\in\W_{\mathrm{sym}}(\T)$ can be reduced to a map $\N_0\times\A\to[-1,1]$ which we denote by slight abuse of notation also by $v$, i.e.,
$$v_{s,j}=v_{|x|,a(x)}=v(x),$$
for $x\in\V$ with $|x|=s$ and $a(x)=j$ and $v_{s,j}=0$ whenever there is no $x\in\V$ such that $(s,j)=(|x|,a(x))$.
\medskip

\begin{Ex}\label{ex:W} All radially symmetric potentials are radial label symmetric and, therefore, contained in $\W_{\mathrm{sym}}(\T)$.
\end{Ex}

We want to prove stability of absolutely continuous spectrum for this type of potentials. To this end, we have to exclude a class of operators which act similarly to operators on regular trees.\medskip

\begin{Def}(Regular tree operators)
We call a label invariant operator $T$ a \emph{regular tree operator} if it satisfies
\begin{itemize}
\item [(R1)] $\sum\limits_{l\in\A}m_{j,l}=\sum\limits_{l\in\A}m_{k,l}$  for all $j,k\in\A$ \emph{    (constant branching)},
\item [(R2)] $m_j=m_k$  for all $j,k\in\A$ \emph{(constant diagonal)},
\end{itemize}
where the matrix $(m_{j,k})$ and the vector $(m_{j})$ are taken from (T2) above. On the other hand, if (R1) or (R2) fails we call $T$  a \emph{non regular tree operator}.
\end{Def}

Let us discuss the assumptions of  a (non) regular tree operator by giving some examples.
\medskip

\begin{Ex}\label{ex:regTree}
(1.) Let $T=\De$. In this case, $m_j=\sum_{l\in\A}m_{j,l}+1$ is the vertex degree of a vertex with label $j\in\A$. Hence, $\De$ is a regular tree operator if and only if the underlying tree $\T$ is regular. By similar arguments we see that the same applies to the operators $\ow\De$ and $A$.

(2.) Probably the simplest example of a non regular tree operator is discussed in \cite{FHS1}. Let $\T$ be a binary tree and $v:\V\to\{-\lm,\lm\}$, $\lm>0$ such that, for all vertices, $v$ takes the value $\lm$ on one of the forward neighbors and $-\lm$ on the other. Then, the operators $\De+v$, $\ow\De+v$ and $A+v$ are non regular tree operators.

(3.) Note that  (R1) and (R2) are only assumptions on the operator and not on the underlying geometry of the tree. For instance,  we can define a regular tree operator on the non regular tree introduced in Example~\ref{ex:M}~(2.). Let $\nu\equiv 1$, $w\equiv 0$ and
\begin{align*}
t(x,y)=\left\{
       \begin{array}{ll}
               \frac{1}{\sqrt2} &: a(x)=a(y)=1, \\
               1 &: \mbox{else.}\\
       \end{array}
       \right.
\end{align*}
Then, the corresponding operator is a regular tree operator although the underlying graph is not regular at all.

\end{Ex}

The next theorem deals with stability of absolutely continuous spectrum of $T$ under small perturbations by radial label symmetric potentials in the case of non regular tree operators.
\medskip

\begin{Thm}\label{main2} (Perturbations by label radial symmetric potentials.) Let $T$ be a non regular tree operator. Then there exists a finite subset $\Sigma_0\subset\si(T)$ such that for every compact set $I\subseteq \si(T) \setminus \Sigma_0$ there exists $\lm_0>0$ such that for all $v\in\W_{\mathrm{sym}}(\T)$ and $\lm\in[0,\lm_0]$ we have
$$I\subseteq \si_{\mathrm{ac}}(T+\lm v)\qqand I\cap \si_{\mathrm{sing}}(T+\lm v)=\emptyset.$$
\end{Thm}

Here the set $\si(H)$, ($\si_{\mathrm{ac}}(H)$, $\si_{\mathrm{sing}}(H)$) denotes the (absolutely continuous, singular) spectrum of a linear operator $H$. In  Section~\ref{s:Greenfunction} definitions are recalled and discussed.

The assumption that $T$ is a non regular tree operator is essential. As proven in \cite[Appendix A]{ASW1} for regular trees there are potentials $v:\V\to[-\lm,\lm]$ which destroy the absolutely continuous spectrum of $\De$ (or $\ow\De$ or $A$) completely no matter how small we choose $\lm$. Examples of such potentials are radially symmetric  ones where the common value in each sphere is given by a random variable. The absolutely continuous spectrum of these operators coincides with the one of one-dimensional operators (up to translation and rescaling). Therefore,  the absolutely continuous spectrum vanishes almost surely, confer \cite{CL,PF}. See also Example~\ref{ex:radsymPot} in Section~\ref{s:problemsLI} for a detailed discussion.

There is a corollary about decaying potentials.
\medskip

\begin{cor}\label{main2cor} Let $T$ be a non regular tree operator. Then, for all $v\in\W_{\mathrm{sym}}(\T)$ such that $v(x)\to0$ as $|x|\to\infty$, we have
$$\si_{\mathrm{ac}}(T+ v)=\si_{\mathrm{ac}}(T).$$
\begin{proof} Multiplication by a  $v$ vanishing at infinity is a compact operator. Therefore, we have $\si_{\mathrm{ess}}(T+v)=\si_{\mathrm{ess}}(T)$. Moreover, by Theorem~\ref{main1}, the spectrum of $T$ is purely absolutely continuous.
Thus, $\si_{\mathrm{ac}}(T+ v)\subseteq \si_{\mathrm{ess}}(T+ v)=\si_{\mathrm{ess}}(T) \subseteq \si(T)=\si_{\mathrm{ac}}(T)$.\\
Conversely, the absolutely continuous spectrum is stable under finitely supported perturbations. In particular, setting $v$ zero at the vertices $x$ where $|v(x)|\geq\lm$ leaves the absolutely continuous spectrum of $T+v$ invariant. By Theorem~\ref{main2}  the absolutely continuous spectrum of every compact subset included in the spectrum of $T$ (except for a finite set) can be preserved under perturbations by $v\in\W_{\mathrm{sym}}(\T)$ if $v$ is sufficiently small. Since $v$ vanishes at infinity, every such interval is contained in the absolutely continuous spectrum of $T+v$.
\end{proof}
\end{cor}

\section{Random potentials}\label{s:RandPot}
We now introduce a class of random potentials. We show that sufficiently small perturbations by these potentials preserve parts of the absolutely continuous spectrum of a label invariant operator.

Let us first recall some basic definitions. Let $(\Om,\PP)$ be a probability space and $J$ an arbitrary index set. Let $X:\Om\times J\to\C$, $(\om,j)\mapsto X_{j}^{\om}$ be a stochastic process, i.e., a measurable map. We say the family of random variables $(X_j)_{j\in J}$ is \emph{independently distributed} if
$$\PP({\{\om\in\Om\mid \bigcap_{k=1}^n X^\om_{j_k} \in B_k\}})=\prod_{k=1}^n \PP\ab{\ac{\om\in\Om\mid X^\om_{j_k} \in B_k}},$$
for all $n\in\N$, $j_k\neq j_l$ if $k\neq l$ and Borel sets $B_k\subseteq \C$, $k=1,\ldots,n$. Moreover, we say the random variables $(X_{j})_{j\in J}$ are \emph{identically distributed} if there is a probability measure $\mu$ on $\C$, called the \emph{distribution} or the \emph{push forward measure}, such that
$$\PP\ab{\ac{\om\in\Om\mid X_j^{\om}\in B}}=\mu(B),$$
for Borel sets $B\subseteq\C$ and $j\in J$. In this case, $(X_{j})_{j\in J}$ is said to be \emph{distributed by} $\mu$.  We call two random variables $X$ and $Y$ on the vertex sets of two isomorphic trees $\T_X$ and $\T_Y$  \emph{identically distributed} if for every graph isomorphism $\psi:\T_X\to\T_Y$ the random variables $X$ and $Y\circ\psi$ are identically distributed.

We will consider random potentials whose distribution is strongly related to the structure of the tree. In particular, we want the random potential $v:\Om\times \V\to[-1,1]$, $(\om,x)\mapsto v_x^\om$ to satisfy the following two assumptions:
\begin{itemize}
\item [(P1)] For all $x,y\in\V$ the random variables $v_x$ and $v_y$ are independently distributed if $ \V_x\cap\V_y=\emptyset$.
\item [(P2)] For all $x,y\in\V$ with $a(x)=a(y)$ the random variables $v\vert_{\V_{x}}$ and $v\vert_{\V_{y}}$ are identically distributed.
\end{itemize}

We denote
$$\W_{\mathrm{rand}}:=\W_{\mathrm{rand}}(\Om,\T):=\ac{v:\Om\times \V\to[-1,1]\mid \mbox{ $v$ satisfies (P1) and (P2)}}.$$

We give some examples of potentials which satisfy (P1) and (P2).\medskip

\begin{Ex}
(1.)
Let  $v:\Om\times \V\to[-1,1]$ be a potential such that all $v({x})$, $x\in \V$  are independently distributed and $v({x})$ and $v({y})$ are identically distributed whenever $a(x)=a(y)$. Then, $v$ satisfies (P1) and (P2). The most important special case is the one of independently and identically distributed potentials.

(2.) Let $w:\Om\to[-1,1]^{ \V}$ be such that $\om\mapsto w_{x}^{\om}$,  ${x\in\V}$  are independently distributed random variables and $f:\R^\V\times \V\to[-1,1]$, $(r,x)\mapsto f_x(r)$ such that $f_x(r)$ depends only on the $r_y$ with $y\in \V_x$ for $x\in \V$. Then, $v=f\circ w:\Om\times \V\to[-1,1]$, $(\om,x)\mapsto f(w(\om),x)$ satisfies (P1).\\
A special case which also satisfies (P2) is the following: Let $\T$ be a binary tree and let the random variable $w:\V\to\{-1,1\}$ take the value $-1$ and $1$ with probability $1/2$ each on every vertex. Then
$$v:\Om\times V\to[-1,1],\quad (\om,x)\mapsto \sum_{y\in \V_x}\frac{1}{2^{d(x,y)+1}}w_y^\om,$$
satisfies (P1) and (P2).

\end{Ex}

We denote the operator of multiplication by a bounded function $f$ on $\ell^2(\V,\nu)$ also by $f$.
Let $\lm\ge 0$ and $v\in\W_{\mathrm{rand}}(\Om,\T)$ be given. We define the family of random operators $H^{\lm,\om}$, ${\om\in\Om}$, on $\ell^2(\V,\nu)$ by
\begin{align*}
H^{\lm,\om}:=T+\lm v^\om.
\end{align*}

We will prove the following theorem.\medskip

\begin{Thm}\label{main3} (Perturbations by random potentials) Let $T$ be a label invariant operator.  Then there exists a finite subset $\Sigma_0\subset\si(T)$ such that for every compact set $I\subseteq \si(T) \setminus \Sigma_0$ there exists $\lm_0>0$ such that for all $v\in\W_{\mathrm{rand}}(\Om,\T)$, $\lm\in[0,\lm_0]$ and almost every $\om\in\Om$
\begin{align*}
I\subseteq \si_{\mathrm{ac}}(H^{\lm,\om})\qqand I\cap \si_{\mathrm{sing}}(H^{\lm,\om})=\emptyset.
\end{align*}
\end{Thm}
The proof will be given in Chapter~\ref{c:main3}.
In contrast to Theorem~\ref{main2} we do not have to exclude the case of regular trees. A similar result for operators with random off diagonal perturbations, Theorem~\ref{t:offdiag}, is stated and proven in Section~\ref{s:offdiagonal}.


Let us put the statement of the theorem above in the context of the present literature. For regular trees a similar statement was proven by \cite{Kl1} and for the binary tree in \cite{FHS2}. The method of \cite{FHS2} was generalized in \cite{Hal1} to make it work for regular trees with arbitrary branching number as well. Moreover, in \cite{FHS2} potentials which take larger values with small probability are allowed. The authors of \cite{ASW1} allow for some dependence within the spheres for the random potentials. They are able to show preservation of some absolutely continuous spectrum. However,  their method does not yield purity of the absolutely continuous spectrum. Dependent random potentials were also studied later, among other questions, in \cite{FHS3}.


\chapter{Basic concepts}\label{c:basicconcepts}
 \begin{quote}
\begin{flushright}
\scriptsize{ A good \emph{tree} cannot bring forth evil fruit, neither can a corrupt \emph{tree} bring forth good fruit. } Matthew 7:18  \end{flushright}
\end{quote}

In this chapter we introduce the objects and the basic techniques of our analysis. The majority of the statements are well known and we will refer to the particular references at the appropriate places. The concepts presented here apply to far more general operators and trees than the ones introduced before. Therefore,  we will treat them in the general setting.

Let $\T=(\V,\E)$ be an arbitrary tree with the only  assumption that it is \emph{locally finite}, i.e., every vertex has only finitely many neighbors.


\section{Self adjoint operators on trees}
We first introduce nearest neighbor operators on trees. Then, we recall some basic notions of spectral theory such as the spectral measures and the Green functions of an operator and discuss some fundamental properties. Finally,  we give a sufficient criterion to exclude singular spectrum.

\subsection{Definition of the operators}
Let $c_c(\V)$ be the space of functions which vanish outside of a finite set. Let $\nu$ be a measure on  $\V$, i.e., a positive function $\nu:\V\to(0,\infty)$. We denote the scalar product of $\ell^2\ab{\V,\nu}$ by $\as{\cdot,\cdot}$ and the corresponding norm by $\|\cdot\|$. In order to introduce a nearest neighbor operator on
$\T$, let a function  $w:\V\to \R$, called the \emph{diagonal}, and a function
$t:\V\times\V\to \C$, called the \emph{off diagonal}, which satisfies \begin{itemize}
\item [(H0)] $t(x,y)\neq 0$ if and only if $x\sim y$ for all $x,y\in\V$ \emph{(adjacency)},
\item [(H1)] $t(x,y)\nu(x)=\ov{t(y,x)}\nu(y)$ for all $x,y\in\V$ \emph{(symmetry)}
\end{itemize}
be given.
We consider self adjoint operators $H$ with $c_c(\V)\subseteq D(H)\subseteq\ell^2(\V,\nu)$ acting as
\begin{equation}\label{e:H}
(H\ph)(x)={\sum_{y\sim x} t(x,y)\ph(y) + w(x) \ph(x)}, \quad \ph\in D(H).
\end{equation}

In the case where the restriction of $H$ to $c_c(\V)$ has a unique self adjoint extension  the operator $H$ is called \emph{essentially self adjoint} on $c_{c}(\V)$. Note that this is not necessarily the case, as it is already known from the theory of Jacobi matrices, (see, for instance, \cite{Be}). In Subsection~\ref{s:UniquenessFixPoints}, Corollary~\ref{c:suffessSA} we give a sufficient criterion for  essential self adjointness.

Note that in contrast to the definition of label invariant operators in the previous chapter, we do not assume that $t$ and $w$ satisfy any type of invariance. Let us give some examples.\medskip

\begin{Ex}(1.) Let $T$ be a label invariant operator. Obviously, (H0) and (H1) follow from (T0) and (T1). As remarked above, the operators are also bounded. Moreover, the operators $T+ \lm v$ with $v\in\W_{\mathrm{sym}}(\T)$ and $H^{\om,\lm}=T+\lm v^\om$ with $v\in\W_{\mathrm{rand}}(\Om,\T)$, $\om\in\Om$ and $\lm\geq0$ also satisfy (H0), (H1) and are bounded since they differ from $T$ only by a bounded potential.

(2.) Other examples are operators arising from regular Dirichlet forms on an $\ell^2$ space of a countable measure space $(\V,\nu)$, see \cite{KL,KL2}.
In our context, where $\T=(\V,\E)$ is a locally finite graph, such forms are given by a symmetric map $b:\V\times\V\to[0,\infty)$ which is non zero if and only if the two vertices in the argument are adjacent and a potential $c:\V\to[0,\infty)$ via
$$h(\ph,\psi) =\frac{1}{2}\sum_{x,y\in\V}b(x,y)(\ph(x)-\ph(y))(\psi(x)-\psi(y)) +\sum_{x\in\V}c(x)\ph(x)\psi(x).$$
It is easy to check that $h\geq0$ on $c_c(\V)$ and the domain of $h$ is the completion of $c_c(\V)$ under the scalar product $\as{\cdot,\cdot}_h=h({\cdot,\cdot})+\as{\cdot,\cdot}$.
The assumptions (H0), (H1) for the corresponding operator can be checked by setting
$$t(x,y)=-\frac{b(x,y)}{\nu(x)}\quad\mbox{and}\quad w(x)=\frac{1}{\nu(x)}\ab{\sum_{y\in \V}b(x,y)+c(x)}.$$
Moreover, we can add to $h$ a bounded sesquilinear form $h'$ with $t'$ satisfying (H0), (H1) as long as $t+t'$ still satisfies (H0).

(3.) A special case of operators arising from (2.) are the Laplacians $\De$ and $\ow \De$ as in Example~\ref{ex:operators}. (Note that the Dirichlet boundary conditions are not necessary in this context but could be imposed by a potential at the root). To see this, let $b:\V\times\V\to\{0,1\}$, $c\equiv 0$ and $\De$ is obtained by choosing $\nu\equiv1$ while $\ow \De$ is obtained by choosing $\nu=\deg$.
\end{Ex}

Let us remark that an operator $H$ given by \eqref{e:H} is bounded if and only if there is a constant $C\geq0$ such that
$$\sup_{x\in\V} \ab{\sum_{y\sim x}|t(x,y)|+|w(x)|} \leq C.$$
The norm of $H$ has the bound $2C$ in this case.
This can be easily seen using the inequality $|(\xi+\zeta)|^2\leq 2|\xi|^2+2|\zeta|^2$, $\xi,\zeta\in \C$ to estimate the terms in the form $h(\ph,\ph)=\as{\ph,H\psi}$.

In the case of bounded operators it indeed suffices to consider operators with positive coefficients as
 the following lemma shows.\medskip

\begin{lemma}
Let $H$ be a bounded self adjoint operator  acting as \eqref{e:H} on $\ell^{2}(\V,\nu)$ and satisfying $\mathrm{(H0)}$, $\mathrm{(H1)}$. For a map $\te:\V\times\V\to\TT:=\{z\in\C\mid\mo{z}=1\}$ satisfying $\mathrm{(H0)}$, $\mathrm{(H1)}$ let $H^{\te}$ be the  operator on $\ell^{2}(V,\nu)$ given by
\begin{align*}
(H^{\te}\ph)(x)=\sum_{y\sim x} \te(x,y)t(x,y)\ph(y)+w(x)\ph(x),\quad\ph\in\ell^{2}(\V,\nu).
\end{align*}
Then, $H^{\te}$ is unitary equivalent to $H$.
\begin{proof}
Let $U:\ell^2(\V,\nu)\to\ell^2(\V,\nu)$ be the unitary diagonal operator with entries
\begin{align*}
    U(x,x)=\prod_{j=0}^{n-1}\te(x_j,x_{j+1}),
\end{align*}
where $x_0\sim\ldots\sim x_n$ is the unique path connecting the root $o$ of $\T$ to $x$. One directly checks that $U^{\ast}=U^{-1}$ and $
  U H  U^{\ast}=H^{\te}$.
\end{proof}
\end{lemma}

A similar statement can be proven for unbounded operators $H$ provided that $U$ is a bijection between $D(H)$ and $D(H^{\te})$.

\subsection{Spectral measures and Green functions}\label{s:Greenfunction}

We recall some basic concepts of spectral theory, such as the spectrum, spectral measures and the Green function. Moreover, we prove some fundamental properties of the Green function  and give a sufficient condition to exclude singular spectrum on an interval.

Let $\T=(\V,\E)$ be a tree and $H$ be a self adjoint operator with domain $D(H)\subseteq\ell^2(\V,\nu)$ acting as \eqref{e:H}. The \emph{resolvent set} of the operator $H$ is the set of $z\in \C$ such that the inverse operator $(H-z)^{-1}=\frac{1}{H-z}$, called the \emph{resolvent}, exists and is a bounded operator on $\ell^2(\V,\nu)$. The complement of the resolvent set of $H$ is called the \emph{spectrum} of $H$ and it is denoted by $\si(H)$. The spectrum of a self adjoint operator is always a closed subset of $\R$. It is compact if and only if the operator is bounded. This is, in particular, the case for the label invariant operators defined in Section~\ref{s:LIOp}.

The spectral theorem (see for instance \cite[Theorem~2.5.1]{Da}) tells us that there exists a finite measure $\mu$ on $\si(H)\times\N$ and a unitary operator $$U: \ell^2(\V,\nu)\to L^2(\si(H)\!\times\!\N,\mu)\quad\mbox{ with }\quad H=U^{-1} M_{\mathrm{id}} U.$$ Here, $M_{f}$ denotes the operator of componentwise multiplication by a function $f:\si(H)\to\C$ and $\id$ denotes the identity function.
Let $f:\si(H)\to\C$ be a Borel-measurable function. We define the operator $f(H)$ by
$$D(f(H)):=U^{\ast}\{\ph\in L^2(\si(H)\times\N,\mu)\mid \ph\in D(M_f)\subseteq L^2(\si(H)\times\N,\mu)\}$$
and
$$f(H):=U^{\ast} M_{f} U.$$
The operators $f(H)$ are bounded if and only if $f$ is bounded. Let $\mathds{1}_B$ be the characteristic function of a Borel-measurable set $B$.
Then, the spectral projections
$ \mathds{1}_{B}(H)$ define bounded operators on $\ell^2(\V,\nu)$. We can write
$$f(H)=\int_{\si(H)}f(t)d\chi_t,$$
where $\chi_t:=\mathds{1}_{(-\infty,t]}(H)$.
For $\ph,\psi\in \ell^2(\V,\nu)$ we define the \emph{spectral measure} $\mu_{\ph,\psi}$ of $H$ with respect to $\ph$ and $\psi$ via
$$\mu_{\ph,\psi}(B)=\int_B d\ip{\ph,\chi_t\psi}$$
for Borel-measurable sets $B\subseteq \si(H)$.
For a vertex $x\in\V$, we let $\de_x:\V\to\C $ be defined by
$$\de_x(y):=\left\{
\begin{array}{cl}
{\nu(x)}^{-\frac{1}{2}}&:y=x, \\
0&: \mbox{else} \\
\end{array}
\right.
$$
and we denote
$$\mu_x:=\mu_{\de_x,\de_x}.$$
Since $\{\de_x\mid x\in \V\}$ is an orthonormal basis of $\ell^2(\V,\nu)$, the whole spectral information of $H$ is encoded in the measures $\mu_x$, $x\in \V$.

Let $x\in \V$. A useful tool to analyze the spectral measure $\mu_x$ is the Green function, which is also known as the Borel transform of $\mu_x$.  The \emph{Green function} for $x\in\V$ and $z\in\C\setminus \si(H)$ is defined as
$$G_x(z,H):=\as{\de_x,\frac{1}{H-z}\de_x} =\int_{\si(H)}\frac{1}{t-z} d\mu_x(t).$$
For convenience we will also write  $G_x$ and $G_x(z)$ for $G_x(z,H)$.

Although $G_x$ is defined on $\C\setminus\si(H)$, we will consider it only as a function on the upper half plane $\h$, which is defined as
$$\h:=\{z\in\C\mid \Im z>0\}.$$
Since $H$ is self adjoint we have $\si(H)\subseteq \R$ and thus $\h\subseteq \C\setminus \si(H)$. We want to mention some well known properties of the Green function.
\medskip

\begin{lemma}\label{l:Herglotz} The Green function $z\mapsto G_x(z,H)$ is a {Herglotz function}  for all $x\in\V$, i.e., it is analytic on $\h$ and maps $\h$ into $\h$.
\begin{proof}Let $z_0\in\h$ and $r:=\|(H-z_0)^{-1}\de_x\|$. Then, for $z\in\h$ with $|z-z_0|<r$, one can easily check that
$$G_x(z,H)=\as{\de_x,\frac{1}{H-z}\de_x}= \sum_{n=0}^\infty(z-z_0)^n\as{\de_x,(H-z)^{-(n+1)}\de_x},$$
which proves the analyticity. To see that $G_x$ maps $\h$ into $\h$ we calculate for $z\in\C$
\begin{align*}
\Im G_x(z,H)&=\frac{1}{2i} \ap{G_x(z,H)+\ov{G_x(z,H)}} =\frac{1}{2i}\as{\de_x,\ab{\frac{1}{H-z}+\frac{1}{H-\ov{z}}}\de_x}\\
&=\Im z\as{\de_x,\frac{1}{H-z}\frac{1}{H-\ov{z}}\de_x}= \Im z \aV{\frac{1}{H-z}\de_x}^2.
\end{align*}
We used a resolvent identity in the third equation. Hence, $\Im G_x(z,H)>0$ if and only if $\Im z>0$.
\end{proof}
\end{lemma}

Moreover, the spectral measures $\mu_{x}$, $x\in\V$ are given as the vague limit of the measures $\pi^{-1}\Im G_x(E+i\eta,H)dE$ as $\eta\downarrow0$.
\medskip

\begin{lemma}\label{l:mu}(Vague convergence of spectral measures.) For all $x\in \V$ the measures ${\pi^{-1}}\Im
G_x(E+i\eta)dE$ converge vaguely to $\mu_x$ as $\eta\downarrow0$ in the sense that
$$\lim_{\eta\downarrow 0}\frac{1}{\pi}\int_\R f(E)\Im G_x(E+i\eta)dE=\int_{\R} f(E) d\mu_x(E),$$
for all compactly supported continuous functions $f$ on $\R$.
\begin{proof}
Let $a,b\in\R$ such that $b\geq a$. The functions $f_\eta:\R \to\R$, given by
\begin{align*}
f_\eta(t)&:= \frac{1}{2\pi i} \int_a^b\ab{\frac{1}{t-E-i\eta}-\frac{1}{t-E+i\eta}}dE\\
&= \frac{1}{\pi}\ap{\arctan\ap{\frac{b-t}{\eta} } -\arctan\ap{\frac{a-t}{\eta}}},
\end{align*}
converge pointwise to the function $\frac{1}{2}(\mathds{1}_{[a,b]}+\mathds{1}_{(a,b)})$ as $\eta\downarrow 0$. Clearly, $|f_\eta(t)|$ is uniformly bounded in $\eta$ we obtain by the spectral theorem
$$\frac{1}{2\pi i}\int_a^b \as{\de_x,\ab{\frac{1}{H-E-i\eta}-\frac{1}{H-{E+i\eta}}}\de_x}dE\to
\frac{1}{2}\as{\de_x,\ap{\mathds{1}_{[a,b]}(H)+\mathds{1}_{(a,b)}(H)}\de_x},$$
as $\eta\downarrow0$.
This equation is known as Stone's formula.
By a resolvent identity one directly computes that
$2i\Im G_x(z,H)=\as{\de_x,((H-z)^{-1}-(H-\ov{z})^{-1})\de_x}$. Approximating compactly supported continuous functions  by characteristic functions, we obtain the statement.
\end{proof}
\end{lemma}

One can say even more about pointwise convergence of the Green function.  A proof of the following statement can e.g. be found in \cite[Theorem 1.4.6]{DeKr}.\medskip

\begin{lemma}\label{l:Gpointwise}(Pointwise Convergence.) Let $x\in\V$. The limits
$\lim_{\eta\downarrow0}G_{x}(E+i\eta)$ exist and are finite for almost every $E\in\R$.
\end{lemma}


\subsection{A criterion for absolutely continuous spectrum}

Let $x\in\V$ be fixed for this subsection. By Lebesgue's decomposition theorem \cite[Theorem I.14]{RS} the measure $\mu_x$ is the sum of a measure $\mu_{\mathrm{ac},x}$ which is absolutely continuous with respect to Lebesgue measure and a singular measure $\mu_{\mathrm{sing},x}$, i.e.,
$$\mu_{x}=\mu_{\mathrm{ac},x} + \mu_{\mathrm{sing},x}.$$
If $E\in\R$ is in the support of $\mu_{\mathrm{ac},x}$ (respectively, $\mu_{\mathrm{sing},x}$) we say $E$ is in the absolutely continuous (respectively, singular) spectrum of $H$ and we write $E\in\si_{\mathrm{ac}}(H)$ (respectively, $E\in\si_{\mathrm{sing}}(H)$).

The following variant of the well known limiting absorption principle gives a sufficient condition for a subset $I\subseteq \si(H)$ to be included in $\si_{\mathrm{ac}}(H)$. We give a proof for the sake of completeness. See also
\cite{Hal1,Kl1,Si1}.
\medskip

\begin{thm}\label{t:Klein}(Absence of singular spectrum.) Suppose that for an open interval $I\subseteq \R$ and $p>1$
$$\liminf_{\eta\downarrow 0}\int_{I}|G_x(E+i\eta)|^pdE < \infty.$$
Then
$$I\cap \si_{\mathrm{sing}}(H)=\emptyset.$$
\begin{proof} Let $I\subseteq \R$ and consider the space $L^p(I,dE)$ of $p$-integrable functions with respect to the Lebesgue measure and norm $\no{ \cdot}_p$.
By assumption, there exists a sequence $\eta_n\downarrow 0$ such that $$\lim_{n\to\infty}\int_{I}|G_x(E+i\eta_n)|^pdE=\liminf_{\eta\downarrow 0}\int_{I}|G_x(E+i\eta)|^pdE=:C<\infty.$$
By Lemma~\ref{l:mu} the measures $\pi^{{-1}}\Im G_x(E+i\eta_n)dE$ converge vaguely to $\mu_{x}$ for $n\to\infty$. We compute for $f$ continuous with support in $I$ by the Hölder inequality with $q\in(1,\infty)$ such that $1/p+1/q=1$.
\begin{align*}
\mo{\int_{I}f(E)d\mu_{x}(E)}&= \frac{1}{\pi}\lim_{n\to0}\mo{\int_{I}f(E)\Im G_x(E+i\eta_n)dE}\\
&\leq \frac{1}{\pi}\no{f}_q\lim_{n\to\infty}\no{\Im G_x(\cdot+i\eta_n)}_p\\
&= \frac{1}{\pi}\no{f}_q\liminf_{\eta\downarrow0}\no{\Im G_x(\cdot+i\eta)}_p\\
&= C\no{f}_q.
\end{align*}
Hence, $\mu_{x}$ defines a continuous linear functional on a dense subset of $L^p(I,dE)$. Therefore,  there exists $g\in L^q(I,dE)$ such that ${\int_{I}f(E)d\mu_{x}(E)}={\int_{I}f(E)g(E)dE}$. This implies that $\mu_{x}$ has a density  in $L^q(I,dE)$. Therefore, the measure $\mu_{x}$ is absolutely continuous with respect to the Lebesgue measure.
\end{proof}
\end{thm}

\section{Recursion relations for the Green functions}\label{s:Recursion}

We next introduce Green functions associated to operators restricted to forward trees. We refer to them as truncated Green functions. These Green functions satisfy certain recursion relations. That will be the starting point of our analysis.

There are three equivalent formulations of the recursion relations. Firstly, there are the recursion formulas for the Green functions. They  can be derived by applying a  resolvent identity twice or by constructing the unique solution of a certain difference equation. This recursion formula is given in Proposition~\ref{p:Gm} and we present two alternative proofs. Secondly, the recursion formula can be translated into an infinite system of polynomial equations which have  the truncated Green function as roots. Finally,  there is a so called recursion map which has the truncated Green functions as fixed points.
Each  viewpoint on the recursion relations reveals particular properties of the Green functions.

Denote by $o\in\V$ the root of the tree $\T=(\V,\E)$. For a vertex $x\in \V$, we define the \emph{(forward) sphere} $S_x^{n}$ of distance $n\in\N_0$ by
$$S_x^{n}:=\{y\in \V\mid d(x,y)=|y|-|x|=n\}$$
and, for $n=1$, we write $S_x:=S_x^1$.
For $x=o$ we drop the subscript writing
$S^n:=S_o^n$. Recall that we denoted by $\T_x=(\V_x, \E_x)$ the forward tree of a vertex $x\in\V$ with respect to the root. Then, the vertex set $\V_{x}$ can be decomposed into
$$\V_x=\bigcup_{n\in\N_0}S_x^n.$$

As above, let $\nu:\V\to(0,\infty)$ be a measure and denote the restriction of $\nu$ to a subtree $\T'=(\V',\E')$ of $\T=(\V,\E)$ by $\nu_{\T'}$. The Hilbert space $\ell^2(\V',\nu_{\T'})$ is a closed subspace of $\ell^2(\V,\nu)$. Let $p_{\T'}:\ell^2(\V,\nu)\to\ell^2(\V',\nu_{\T'})$ be the canonical projection and $i_{\T'}:\ell^2(\V',\nu_{\T'})\to \ell^2(\V,\nu)$ be its adjoint operator which is the continuation by zero. For a the self adjoint operator $H$ on $D(H)$, we define the restriction of $H$ to $\T'$ by
\begin{align*}
D(H_{\T'}):=p_{\T'}D(H)=\{\ph\in \ell^2(\V',\nu_{\T'})\mid i_{\T'}\ph\in D(H)\}
\end{align*}
and
\begin{align*}
H_{\T'}&:=p_{\T'}Hi_{\T'}.
\end{align*}

The \emph{truncated Green function} of an operator $H$ with respect to the forward tree $\T_x=(\V_x,\E_x)$, $x\in\V$ is denoted by
\begin{align*}
\Gm_x(z,H):=G_x(z,H_{\T_x})=\as{\de_x,\frac{1}{H_{\T_x}-z}\de_x},\quad z\in \h.
\end{align*}
Note that, if $x=o$ we have $\Gm_o(z,H)=G_o(z,H)$. For convenience we will sometimes  write $\Gm_{x}$ or $\Gm_x(z)$ for $\Gm_x(z,H)$.


\subsection{Recursion formulas}\label{ss:Gm}
We now present the recursion formulas for $\Gm_x$. The formulas can be already found in a similar form in \cite{Kl1,ASW1,FHS1,FHS2}. We present two versions of the proof.
\medskip

\begin{prop}\label{p:Gm} (Recursion formulas.)
For $x\in \V$ and $z\in\h$ we have
\begin{equation}\label{e:Gm}
-\frac{1}{\Gm_x(z,H)}= z-w(x)+\sum_{y\in S_x} |t(x,y)|^2 \Gm_y(z,H).
\end{equation}
\begin{proof}
Let $\Lambda$ be the self adjoint operator which connects $x\in\V$ to its forward neighbors, i.e., $\as{\Lambda\de_x,\de_{y}}=\ov{\as{\Lambda\de_y,\de_x}}=t(x,y)$ for all $y\in S_x$ and all other matrix elements vanish. Then, $H':=H-\Lambda$ is a direct sum of the operators $H_{\T_y}$, $y\in S_x$ and $H_{\T'}$ where $\T'=\T\setminus \bigcup_{y\in S_x}\T_y$. We can think of $H'$ as the operator on the tree where all edges connecting $x$ with vertices in $S_{x}$ are removed. Applying a resolvent identity twice yields
\begin{align*}
\frac{1}{H-{z}}&=\frac{1}{H'-{z}} -\frac{1}{H'-{z}}\Lambda\frac{1}{H-{z}}\\
&=\frac{1}{H'-{z}} -\frac{1}{H'-{z}}\Lambda\frac{1}{H'-{z}} +\frac{1}{H'-{z}}
\Lambda\frac{1}{H'-{z}}\Lambda\frac{1}{H-{z}}.
\end{align*}
Since $H'$ is a direct sum of operators, so is the resolvent $(H'-z)^{-1}$. In particular,  off diagonal matrix elements $\as{\de_y,\ap{H'-z}^{-1}\de_{y'}}$ are non zero if and only if $y$ and $y'$ are in the same component after removing the corresponding edges by subtracting $\Lambda$ from $H$.

We start with $x=o$.  Note that $H_{\T'}$ is the number $w(o)$ in this case. We consider the $(o,o)$ matrix element and see that the left hand side above is equal to $G_{o}(z,T)=\Gm_{o}(z,T)$.
We get, by the considerations above, via a direct calculation
\begin{align*}
\Gm_{o}(z,T)=\frac{1}{w(o)-z}+\frac{1}{w(o)-z}\Gm_{o}(z,T)\sum_{y\in S_o}|t(o,y)|^2\Gm_y(z,T).
\end{align*}
This yields the statement for $x=o$. As an arbitrary vertex $x$ is the root of the forward tree $\T_{x}$, the statement follows.
\end{proof}
\end{prop}

There is also a more `pedestrian' way of proving the recursion formulas. It follows by constructing the unique solution $\ph\in D(H_{x})$ to the equation $(H_{\T_{x}}-z)\ph=\de_x$. However, for convenience we only treat the case $\nu\equiv1$

\begin{proof}[Alternative proof of Proposition~\ref{p:Gm} for $\nu\equiv 1$] Fix $x\in\V$. For $y\in \V_x$ let
\begin{align*}
\Gm_{x,y}:=\Gm_{x,y}(z,H) &:= \as{\de_y,\frac{1}{H_{\T_x}-z}\de_x}.
\end{align*}
Note that $y\mapsto\Gm_{x,y}(z,H)$ is the unique function $\ph\in D(H_{\T_x})$ such that
\begin{equation*}
(H_{\T_x}-z)\ph=\de_x.
\end{equation*}
For $v\in\V_x$ let
$$\ph(v):=
\frac{1}{z-w(x)+\sum_{y\in S_x}|t(x,y)|^2\Gm_{y}}
\left\{
\begin{array}{cl}
\ov{t(x,y)}\Gm_{y,v} &: v\in \V_y,\,y\in S_x, \\
-1 &: v=x. \\
\end{array}
\right.
$$
We claim that $\ph\in D(H_{\T_x})$ and $(H_{\T_x}-z)\ph=\de_x$. By the discussion above this yields $\ph=\Gm_x$. It is clear that $\ph\in D(H_{\T_x})$ since $\Gm_{y,\cdot}$ are Green functions. To check the equation we have to consider three cases. For $v=x$ we get
$$(H_{\T_x}-z)\ph(x)=
\sum_{y\in S_x}t(x,y)\ph(y)+(w(x)-z)\ph(x)
=1.$$
For $v\in S_x$ we get
\begin{align*}
(H_{\T_x}-z)\ph(v) &= t(v,x)\ph(x)+
\sum_{y\in S_v}t(v,y)\ph(y)+(w(v)-z)\ph(v) \\
&= \frac{1}{(\dots)}\ab{-t(v,x)+t(v,x)(H_{\T_v}-z)\Gm_{v,v}}\\
&=\frac{1}{(\dots)}\ab{-t(v,x)+t(v,x)\de_v(v)}\\
&=0.
\end{align*}
For $v\in \V_y\setminus \{ y \}$, $y\in S_x$ note that $H_{\T_x}\ph(v)=H_{\T_y}\ph(v)$. Thus,
$$(H_{\T_x}-z)\ph(v) =(H_{\T_y}-z)\ph(v)= \frac{1}{(\dots)}(H_{\T_y}-z)\Gm_{v,v}=\frac{1}{(\dots)}\de_y(v)=0.
$$
\end{proof}


\subsection{Polynomial equations}\label{ss:Polynom}
We give an equivalent formulation of the recursion relation \eqref{e:Gm} by a system of polynomial equations. For $x\in \V$ we define the polynomial
\begin{equation*}
P_x:\C\times\C^{\V}\to\C,\quad(z,\xi)\mapsto \ab{z-w(x)+\sum_{y\in S_x}|t(x,y)|^2\xi_y}\xi_x+1
\end{equation*}
and consider the infinite system of polynomial equations
\begin{equation}\label{e:Polynom}
P(z,\xi):=(P_x(z,\xi))_{x\in\V}\equiv 0.
\end{equation}
By $\eqref{e:Gm}$ we see that $(z,\Gm(z))$, with $\Gm(z):=(\Gm_x(z))_{x\in\V}$ is a solution of \eqref{e:Polynom}, i.e.,
$$P_x(z,\Gm(z))=0.$$
We show in Section~\ref{s:UniquenessFixPoints} that \eqref{e:Polynom} has a unique solution in $\h\times\h^{\V}$ which must then be $(z,\Gm(z))$.

\subsection{Recursion maps}\label{ss:Psi}

Another equivalent formulation of the recursion formulas \eqref{e:Gm} is given by recursion maps. Define for $z\in\h\cup\R$ the function $\Psi_z:=\Psi^{(H)}_z:\h^\V\to\h^\V$ via the components $\Psi_{z,x}:=\Psi^{(H)}_{z,x}$ in $x\in \V$  by
\begin{equation}\label{e:Psi}
\Psi_{z,x}:\h^{S_x}\to\h,\quad g\mapsto-\frac{1}{z-w(x)+\sum_{y\in S_x}|t(x,y)|^2g_y}.
\end{equation}
For a subset $W\subseteq \V$ and $g\in\h^{\V}$ we write $g_{W}\in\h^{W}$ for the restriction of $g$ to $W$. We denote the restriction $\Psi_{z,S^n}:=\Psi^{(H)}_{z,S^{n}}$ of $\Psi_z$ to the spheres $S^{n}$, $ n\in\N$ of the tree by $$\Psi_{z,S^n}:\h^{S^{n+1}}\to\h^{S^n}, \quad g\mapsto \ab{\Psi_{z,x}^{(H)}(g_{S_x})}_{x\in S^{n}}.$$
By the recursion formula $\Gm(z)=(\Gm_x(z))_{x\in\V}\in\h^\V$ is a fixed point of $\Psi_z$, i.e.,
$$\Psi_z(\Gm(z))=\Gm(z).$$
We will show in Section~\ref{s:Contraction} that $\Psi_{z,S^n}$ is a contraction with respect to a (semi) hyperbolic metric whenever $z\in\h$. We want to mention at this point that the concept of studying the Green function on graphs via hyperbolic contraction properties of maps similar to $\Psi_{{z,S^n}}$ was first introduced in \cite{FHS1}. There, operators on $\ell^2(\V)$ on general graphs $(\V,\E)$ are considered which satisfy several additional conditions. For example,  these conditions imply that the operators are bounded.

\subsection{Application to label invariant operators}\label{ss:application}

Let us discuss the definitions of the previous subsections in the special case of label invariant operators and radial label symmetric potentials.

Let $\T=\T(M,j)$ be a tree $(\V,\E)$ given by a substitution matrix $M$  on a finite label set $\A$ and $j\in\A$. Moreover, let $a:\V\to\A$ be the labeling function of $\T$.

For an operator $T$  the label invariance (T2) yields that the restrictions $T_{\T_x}$ and $T_{\T_y}$ are unitarily equivalent operators whenever $a(x)=a(y)$.
Therefore,  the truncated Green functions $\Gm_x(z,T)$ and $\Gm_y(z,T)$ agree in this case. Consequently, the vector of truncated Green functions $(\Gm_x(z,T))_{x\in\V}$ can be reduced to a finite dimensional vector
$$\Gm(z,T):=(\Gm_j(z,T))_{j\in\A},$$
 by letting $\Gm_j(z,T)$, $j\in\A$ be any $\Gm_x(z,T)$ with $x\in\V$ such that $a(x)=j$. Therefore,  the recursion formulas \eqref{e:Gm} also reduce to finitely many equations
$$-\frac{1}{\Gm_j(z,T)}=z-m_j+\sum_{k\in\A}m_{j,k}\Gm_k(z,T),\quad j\in\A.$$
Similarly, the infinite system of polynomial equations \eqref{e:Polynom} can be reduced to a finite system of polynomial equations
\begin{align*}
P_j(z,\xi) =\ap{z-m_j+\sum_{k\in\A}m_{j,k}\xi_k}\xi_j+1=0 ,\quad j\in\A,
\end{align*}
with $z\in\h\cup\R$, $\xi\in\h^\A$. We have $P_j(z,\Gm(z,T))=0$ for all $j\in\A$.
We can also reduce the recursion map $\Psi_z:\h^{\V}\to\h^{\V}$  to a map $$\Phi_z:\h^\A\to\h^\A, \quad g\mapsto \ab{-\frac{1}{z-m_j+\sum_{k\in\A}m_{j,k}g_k}}_{j\in\A}$$
and we have that $\Gm(z,T)$ is a fixed point of $\Phi_z$, i.e., $\Phi_{z}(\Gm(z,T))=\Gm(z,T)$.

Let us turn to radial label invariant operators $H=T+\lm v$, i.e., $T$ is a label invariant operator, $\lm\geq0$ and $v\in\W_{\mathrm{sym}}(\T)$ is a radial label symmetric potential. Recall that this means that $v(x)=v(y)$ whenever $a(x)=a(y)$ and $|x|=|y|$ for $x,y\in\V$. For these operators we have less symmetry than in the case of label invariant operators. Still, by the radial label symmetry of $v$ the operators $H_{\T_x}$ and $H_{\T_y}$ are unitarily equivalent operators whenever $a(x)=a(y)$ and $|x|=|y|$. Hence, the vector $(\Gm_x(z,H))_{x\in\V}$ in $\h^\V$ can be reduced to a vector $(\Gm_{s,j}(z,H))_{s\in\N_0,j\in\A}$ in $\h^{\N_0\times\A}$.

As discussed in Section~\ref{s:RLSymPot} for $s\leq n(M)$ there are finitely many sites $(s,j)$ for which $v_{s,j}$ is not defined by $v\in\W_{\mathrm{sym}}(\T)$. (Recall that $n(M)$ is the smallest number $n\in\N$ such $M^{n}$ has only non zero entries which is defined in (M2).) So one may ask how $\Gm_{s,j}(z,T)$ shall be defined on these sites.  On the one hand, this is not a relevant problem since these components do not enter anywhere in the analysis. On the other hand, keeping these components void might be considered as a flaw. Therefore,  we present a way that these components can be reasonably defined.

Let $\T'=(\V',\E')$ be a tree given by $\T(M,k)$ for some $k\in\A$ and $\T=(\V,\E)$ a subtree $\T'_{x}=(\V_x',\E_x')$ of $\T'$ which is given by $\T(M,j)$ for some $j\in\A$, i.e., the $x\in\V'$ has label $a(x)=j$.
By the primitivity assumption, a label invariant measure $\nu$ and a label invariant operator $T$ on $\T$  uniquely define a label invariant measure $\nu'$ and a label invariant operator $T'$ on $\T'$. Similarly, a radial label symmetric potential $v$ defined on $\T=\T'_{x}$ induces a radial label symmetric potential $v'$ on $\T'$ by
\begin{align*}
v'(y)=\left\{\begin{array}{cl}
v_{|y|,a(y)}  &:\mbox{there exists $x_0\in\V$ such that $(|x_0|,a(x_0))=(|y|,a(y))$,}  \\
0  & :\mbox {else} . \\
\end{array}\right.
\end{align*}

In this manner a radial label symmetric operator $H=T+\lm v$ on $\T$ induces a radial label symmetric operator $H'$ on the super tree $\T'$ via $H'=T'+\lm v'$.  As $\T=(\V,\E)$ is a subtree of $\T'=(\V',\E')$ the components of the truncated Green functions coincide on $\V$. We choose the super tree $\T'$ and $\T=\T_x'$ such that $|x|\geq n(M)$ where $|\cdot|$ is considered with respect to the root of $\T'$.  We now define
\begin{align*}
    \Gm_{|y|-|x|, a(y)}(z,H)=\Gm_{y}(z,H'),\quad y\in\V.
\end{align*}
Hence, we defined $\Gm_{s,j}(z,T)$ for all $(s,j)\in\N_0\times\A$ and the definition does not depend on the choice of the particular super tree $\T'$.

With these conventions we can write the recursion formulas for the truncated Green functions in the reduced way
\begin{align*}
\frac{1}{\Gm_{s,j}(z,H)}={z-m_j-v_{s,j}+ \sum_{k\in\A}m_{j,k}\Gm_{s+1,j}(z,H)},
\end{align*}
for $(s,j)\in\N_0\times\A$. Since we do not need the viewpoint of polynomial equations for this model we skip it here. On the other hand, the recursion maps will be important. The restrictions $\Psi_{z,S^s}$ of the maps $\Psi_z$ to the spheres $S^{s}$ can be reduced to a map
$$\Psi_{z,s}:\h^\A\to\h^\A, \quad g\mapsto \ab{-\frac{1}{z-m_j-v_{s,j}+\sum_{k\in\A}m_{j,k}g_y}}_{j\in\A},$$
such that we have the following equation for the reduced truncated Green function
$$\Psi_{z,s}\ab{(\Gm_{s+1,j}(z,H))_{j\in\A}}=(\Gm_{s,j}(z,H))_{j\in\A}.$$

\subsection{More formulas for the Green functions}

In this subsection we want to relate  $G_x(z,H)=\ip{\de_x,(H-z)^{-1}\de_x}$ and $\Gm_x(z,H)=\ip{\de_x,(H_{\T_x}-z)^{-1}\de_x}$ for general self adjoint $H$ acting as \eqref{e:H}. Moreover, we will give a formula for the off diagonal elements of the resolvents.
\medskip

\begin{prop}\label{p:G}For $x\in \V$, $y\in S_x$ and $z\in\h$ we have
\begin{equation}\label{e:G}
G_y(z,H)= \Gm_y(z,H)+ |t(x,y)|^2 \Gm_y(z,H)^2G_x(z,H).
\end{equation}
\begin{proof}
As in the proof of Proposition~\ref{p:Gm} let $\Lambda$ be the self adjoint operator which connects $x\in\V$ to its forward neighbors. So, $H':=H-\Lambda$ is a direct sum of the operators $H_{\T_y}$, $y\in S_x$, and $H_{\T'}$ where $\T'=\T\setminus \bigcup_{y\in S_x}\T_y$. Again applying a resolvent identity twice yields
\begin{align*}
\frac{1}{H-{z}}
&=\frac{1}{H'-{z}} -\frac{1}{H'-{z}}\Lambda\frac{1}{H'-{z}} +\frac{1}{H'-{z}}\Lambda\frac{1}{H-{z}}\Lambda\frac{1}{H'-{z}}
\end{align*}
and $\eqref{e:G}$ follows by a similar reasoning as in the proof of Proposition~\ref{e:Gm}.
\end{proof}
\end{prop}

This formula combined with the recursion formula \eqref{e:Gm} in Proposition~\ref{p:Gm} yields the following statements.
\medskip

\begin{prop}\label{p:G2}(Extension from $\Gm$ to $G$.) Let $E\in\R$.

$\mathrm{(1.)}$ If $\Gm_{x}(E+i\eta,H)$ is uniformly bounded in $\eta>0$ for all $x\in \V$, then $G_{x}(E+i\eta,H)$ is uniformly bounded in $\eta>0$ for all $x\in \V$.

$\mathrm{(2.)}$ If $\Gm_{x}(E+i\eta,H)$ is uniformly bounded in $\eta>0$ and  $\lim_{\eta\downarrow0}\Im\Gm_{x}(E+i\eta,H)=0$ for all $x\in\V$, then     $\lim_{\eta\downarrow0}\Im G_{x}(E+i\eta,H)=0$ for all $x\in\V$.

$\mathrm{(3.)}$ If $\Gm_{x}(E,H):=\lim_{\eta\downarrow0}\Gm_{x}(E+i\eta,H)$ exists and $\Im \Gm_{x}(E,H)>0$  for all $x\in\V$, then $G_{x}(E,H):=\lim_{\eta\downarrow0}G_{x}(E+i\eta,H)$ exists and $\Im G_{x}(E,H)>0$ for all $x\in\V$.
\begin{proof}
(1.), (2.) and the statement about the existence of the limits in (3.) directly  follows from Proposition~\ref{p:G} by induction over the distance to the root.

It remains to show the statement  in (3.) about positivity of the imaginary parts. For $x\in\V$ let $x_0\sim x$ be the vertex which lies on the path connecting $x$ with the root $o$. We single out $x$ to be the root of the rooted tree $(\T,x)$. Applying the recursion relation \eqref{e:Gm} with respect to the rooted tree $(\T,x)$ yields for $z=E+i\eta$
\begin{align*}
-\frac{1}{G_{x}(z,H)}=z-w(x)+\sum_{y\in S_x}
|t(x,y)|^{2}\Gm_{y}(z,H) +|t(x,x_0)|^{2}G_{x_0}(z,H_{\T\setminus \T_{x}}).
\end{align*}
We estimate $\Im G_{x_0}(z,H_{\T\setminus \T_{x}})>0$, go over to the limit $\eta\downarrow0$, take imaginary parts and multiply by $|G_{x}(E,H)|^{2}$ to get
\begin{align*}
\Im G_{x}(E,H)&\geq\ap{\sum_{y\in S_x}|t(x,y)|^{2}\Im\Gm_{y}(E,H)}|G_{x}(E,H)|^{2}.
\end{align*}
To conclude positivity of the left hand side we have to show $|G_{x}(E,H)|>0$. To see this we apply the recursion formula \eqref{e:Gm} to $G_{x_0}(z,H_{\T\setminus \T_{x}})$ (with respect to the rooted tree $(\T,x_0)$) in the first equation of the proof. We  take the modulus, go over to the limit $\eta\downarrow0$  and  obtain
\begin{align*}
\frac{1}{|G_{x}(z,H)|}\leq\mo{E}+\mo{w(x)}+\sum_{y\in S_x}|t(x,y)|^{2}|\Gm_{y}(z,H)| +\frac{|t(x,x_0)|^{2}}{|t(x_0,y_0)|^{2}\Im\Gm_{y_0}(E,H)},
\end{align*}
where we estimated the denominator of the last term first by its imaginary part and then dropped all but one term for some vertex  $y_0$  in $S_{x_0}\setminus\{x_0\}$. Since we assumed $\Im \Gm_{y}(E,H)>0$ for all $y\in\V$ the statement follows.
\end{proof}
\end{prop}

At the end of this section we present a formula for the off diagonal elements of the resolvents. For $x,y\in\V$ let
\begin{align*}
G_{x,y}(z,H):=\as{\de_y,\frac{1}{H-z}\de_x} =\int_{\si(H)}\frac{1}{t-z} d\mu_{x,y}(t),
\end{align*}
where $\mu_{x,y}:=\mu_{\de_x,\de_y}$ is an off diagonal spectral measure.
\medskip

\begin{prop}\label{p:Goffdiag}Let $x,y\in\V$ and $z\in\h$. Moreover, let $x_0,\ldots,x_n$ be a path from $x$ to $y$ in $\V$.
Then
\begin{align*}
    G_{x,y}(z,H)=G_{x}(z,H)\prod_{j=1}^{n}\Gm_{x_{j}}(z,H).
\end{align*}
\begin{proof} The statement is clear for $x=y$. As in the proof of the Propositions~\ref{e:Gm} and~\ref{e:G} let $\Lambda$ be the self adjoint operator which connects $x\in\V$ to its forward neighbors and $H':=H-\Lambda$. Applying the resolvent identity once yields
\begin{align*}
\frac{1}{H-z}=\frac{1}{H'-z}+\frac{1}{H-z}\Lambda\frac{1}{H'-z}.
\end{align*}
As $\as{\de_y,(H'-z)^{-1}\de_x}=0$, we get by taking the $(x,y)$ matrix element
\begin{align*}
G_{x,y}(z,H)=G_{x}(z,H)\Gm_{x_1,y}(z,H).
\end{align*}
Iterating this argument with $\Gm_{x_1,y}(z,H)$ we obtain the statement.
\end{proof}
\end{prop}

These formulas relating the Green functions to the truncated Green functions give us a criterion to decide whether two operators acting as \eqref{e:H} on $c_c(\V)$ agree.\medskip

\begin{cor}\label{c:essSA} Let $\T$ be a tree and $H$, $H'$ two self adjoint operators with diagonals $w$, $w'$ and off diagonals $t$, $t'$ acting as \eqref{e:H} on $c_c(\V)$. If the off diagonals agree and  $\Gm_{x}(z,H)=\Gm_{x}(z,H')$ for all ${x\in\V}$, then $H=H'$.
\begin{proof}
By Proposition~\ref{p:G} the assumption yields that the diagonal elements of the resolvents of $H$ and $H'$ coincide. By Proposition~\ref{p:Goffdiag} this follows as well for the off diagonal elements. Since the matrix elements of the resolvents $(H-z)^{-1}$ and $(H'-z)^{-1}$ coincide the operators must be equal.
\end{proof}
\end{cor}


\section{A hyperbolic semi metric}\label{s:semi metric}

We will introduce a hyperbolic semi metric $\gm$ in order to study contraction properties of the recursion maps introduced in the previous section. A \emph{semi metric} is a map which satisfies the axioms of a metric except for the triangle inequality. Moreover, an \emph{extended (semi) metric} is a map which possibly takes the value $+\infty$ but otherwise satisfies the axioms of a (semi) metric.

We discuss in this section the relation of $\gm$ to the standard hyperbolic metric, prove two limit point principles for general semi metrics, show a substitute for the triangle inequality and study some characteristics of distance balls in this semi metric. One of the limit point principles will be needed at the end of this chapter. The other results will be needed in Chapter~\ref{c:freeoperator} and Chapter~\ref{c:main3}. The reason why we prove them here is that they are of a purely geometric nature. However,  they are not necessary for the understanding of the rest of this chapter.

Let $J$ be a finite index set.
We introduce the metric $\dist_{\h^J} $ on $\h^J$  by
$$\dist_{\h^J}(g,h)=\cosh^{-1}\ab{\frac{1}{2}\gm_J(g,h)+1},$$
where $\gm_J:\h^J\times\h^J\to[0,\infty)$ is given by
\begin{align*}
\gm_J(g,h):=\max_{j\in J}\gm(g_j,h_j),\qquad g,h\in\h^{J},
\end{align*}
and $\gm:\h\times\h\to[0,\infty)$
\begin{align*}
\gm(g,h):=\frac{|g-h|^2}{\Im g\Im h},\qquad g,h\in\h.
\end{align*}

Indeed, $\dist_{\h}$ is the standard hyperbolic metric on $\h$. See \cite[Theorem~1.2.6]{Ka} for a proof and discussion. Obviously, $\gm$ is positive definite and symmetric. One easily checks that the triangle inequality is not satisfied. (For instance,  let $h_{1}=i$, $h_{2}=2+i$, $h_{3}=1+i$ and observe that $4=\gm(h_1,h_2)>\gm(h_1,h_3)+\gm(h_3,h_2)=2$.)

Let us define distance balls with respect to
$\gm_J$. For $r\geq 0$ and $h\in\h^J$ let
$$B_r(h):=\{g\in\h^J\mid \gm_J(g,h)\leq r\}.$$

We want to recall the notions of isometry, quasi contraction and (uniform) contraction. Let $(X,d_X)$ and $(Y,d_Y)$ be two semi metric spaces and $X_0\subseteq X$. Then, a map $\ph:(X,d_X)\to (Y,d_Y)$ is called an
\begin{align*}
\ac{\begin{array}{c}
\mbox{\emph{isometry}} \\
\mbox{\emph{quasi contraction}} \\
\mbox{\emph{contraction}}
\end{array}}
\mbox{on $X_0$, if }\;
d_Y(\ph(x),\ph(x'))
\ac{\begin{array}{c}
= \\
\leq \\
<
\end{array}}
d_X(x,x'),
\end{align*}
for all $x,x'\in X_0$ with $x\neq x'$.
A contraction $\ph$ is called \emph{uniform} on $X_0\subseteq X$ if there exists $c_0<1$ such that $d_Y(\ph(x),\ph(x'))\leq c_0 d_X(x,x')$ for all $x,x'\in X_0$. In this case $c_0$ is called a \emph{contraction coefficient} of $\ph$.
Note that a quasi contraction (or an isometry or a contraction) must be continuous.
For a subset $U$ of a semi metric space $(X,d_{X})$ let
\begin{equation*}
\diam (U):=\sup_{x,y\in U}d_X(x,y).
\end{equation*}


\subsection{Limit point principles}\label{ss:limitpointprinciple}

A limit point principle is a generalization of a fixed point principle. Instead of studying a contraction on a (semi) metric space into itself, we draw our attention to a sequence of functions mapping between a sequence of (semi) metric spaces.

Let $(X,d_X)$, $(Y,d_Y)$ be semi metric spaces. Let $\Lip(X,Y)$ be the space of all quasi contractions from $X$ to $Y$. The space $\Lip(X,Y)$ can be equipped  with an extended semi metric $d_{X,Y}$ via
\begin{align*}
d_{X,Y}(\ph,\psi) :=\sup_{x\in X}d_{Y}\ap{\ph(x),\psi(x)}.
\end{align*}
Whenever $X$ or $Y$ is compact then $d_{X,Y}$ is a semi metric. In the case where $X$ and $Y$ are metric spaces $d_{X,Y}$ is an extended metric. It is even a metric whenever $X$ or $Y$ is additionally compact.

Let a sequence of semi metric spaces $X=((X_j,d_j))_{j\in\N_0}$  be given. We define
\begin{align*}
\Lip(X):=\{\ph=(\ph_{j})_{j\in\N_0}\mid\ph_j\in \Lip(X_{j+1}, X_{j}),\, j\in\N_0\}.
\end{align*}

We call a sequence $h\in X$ a \emph{limit point} of $\ph\in \Lip(X)$ if
\begin{align*}
h_j\in\bigcap_{n\geq j} \ph_{j}\circ\ldots\circ\ph_{n}(X_{n+1})\subseteq X_j\;\mbox{ for all $j\in\N_0$}.
\end{align*}
Notice that the intersection above is decreasing. Therefore,
if the spaces $X_{j}$ are compact, then the intersection is non empty. Then, $\ph$ has at least one limit point. If the limit point
$h$ is unique, then it is a fixed point of $\ph$, i.e., $\ph(h)=h$.

The question of uniqueness is addressed in the following lemma.
\medskip

\begin{lemma}\label{l:fixpoint} (Existence of limit points.)
Let $X=((X_j,d_j))_{j\in\N_0}$ be a sequence of compact semi metric spaces, $x_j=\diam (X_j)$ and the components $\ph_{j}$ of $\ph\in \Lip(X)$  uniform contractions with contraction coefficients $c_{j}\in[0,1)$, $j\in\N_0$.
If for some $\be>0$
\begin{align*}
\limsup_{n\to\infty}\frac{1}{x_n^{\be}} \sum_{j=1}^{n-1}\frac{(1-c_{j}^{\be})}{c_{j}^{\be}} =\infty,
\end{align*}
then $\ph$ has a unique limit point $h\in X$ which can be obtained for arbitrary sequences $g\in X$ via
\begin{align*}
    h_j=\lim_{n\to\infty} \ph_{j}\circ\ldots\circ \ph_{n}(g_n).
\end{align*}
\begin{proof}
Let $B_{n}=\ph_{0}\circ\ldots\circ \ph_{n-1}(X_{n})\subseteq X_{0}$ for $n\in\N$. Let $(n_{k})$ be a sequence realizing the $\limsup$ in the assumption. For given $n\in\N$, let $l\in\N$ be the largest number such that $n_l\leq n$.
Since $B_{n}\subseteq B_{n_l}$ we have
\begin{align*}
\diam (B_{n})\leq \diam (B_{n_{l}})\leq x_{n_l}\prod_{j=1}^{n_l-1}c_{j}.
\end{align*}
Clearly, convergence of the right hand side to zero is invariant under taking powers of $\be>0$.
Moreover, by taking inverses, we get
\begin{align*}
\prod_{j=1}^{n_l-1}\frac{1}{c_{j}^{\be}} =\prod_{j=1}^{n_l-1}
\ap{1+\frac{(1-c_{j}^{\be})}{c_{j}^{\be}}} \geq
\sum_{j=1}^{n_l-1}\frac{(1-c_{j}^{\be})}{c_{j}^{\be}}.
\end{align*}
This yields, by our assumption, that the product in the estimate of $\diam (B_n)$ above converges to zero.
Since the $X_j$ are compact and $\ph_{j}$ are continuous the sets $B_j$ are compact. Moreover, since for $n\geq m$ we have $B_{n}\subseteq B_{m}$ and therefore
\begin{align*}
 \bigcap_{n\in\N}\ph_{0}\circ\ldots\circ\ph_{n-1}(X_{n})=  \bigcap_{n\in\N} B_n\neq\emptyset.
\end{align*}
Moreover, since $\diam (B_n)\to0$ as $n\to\infty$ there is a unique  $h_0\in X_0$ which can be obtained as a limit as claimed. We can apply the same argument for every $j\in\N_0$ to find the unique limit point $h=(h_j)$ of $\ph$.
\end{proof}
\end{lemma}

Suppose that all spaces $(X_j,d_j)$ and all maps $\ph_j$, $j\in\N_0$ are equal. If  $h$ is a \emph{fixed point} of $\ph_j$, i.e. $\ph_j(h)=h$, $j\in\N_0$, then the sequence of elements $h_j=h$, $j\in \N_0$ is a limit point of $(\ph_j)$. However,  the contrary is not true. For example, consider a rotation of the unit circle $\Sp^{1}$. Such a map has $\Sp^{1}$ as limit points but it only has fixed points if the rotation is the identity map. Nevertheless, whenever there is a unique limit point $h=(h_j)$ then all $h_{j}$, $j\in\N_0$ must be equal and, therefore, be fixed points. Thus, the previous lemma implies an immediate corollary which is a fixed point principle.\medskip

\begin{cor}\label{c:fixpoint} Let $(X,d_X)$ be a compact semi metric space and $\ph:X\to X$ a uniform contraction. Then,
$\ph$ has a unique fixed point $h\in X$ which can be obtained for arbitrary sequences $g=(g_j)$ in $X$ via
\begin{align*}
    h=\lim_{n\to\infty} \ph^{n}(g_n).
\end{align*}
\end{cor}

We next show that if two sequences of maps $\ph,\psi\in \Lip(X)$ are close and $\ph_j$ are uniform contractions, then  $\psi$ has a limit point which is close to the unique limit point of $\ph$.
Let $X=((X_j,d_j))_{j\in\N_0}$ be a sequence of locally compact metric spaces. We define, for $h\in X$ and $R \geq 0$, the set $B(R):=(B_{j}(R))_{j\in\N_0}\subseteq X$ via
\begin{align*}
B_{j}(R)=\{g\in X_j\mid d_j(g,h_j)\leq R\}.
\end{align*}
Note that the restrictions of $\ph_{j}\in\Lip(X_{j+1},X_{j})$ to $B_{j+1}(R)$ is in $\Lip(B_{j+1}(R),X_{j})$  and for $\ph,\psi\in\Lip(X)$ and $j\in\N_0$.
\begin{align*}
d_{B_{j+1}(R),X_{j}}(\ph_{j},\psi_{j})\leq d_{X_{j+1},X_{j}}(\ph_{j},\psi_{j}).
\end{align*}

\begin{lemma}\label{l:fixpoint2} (Stability of limit points.)
Let $X=((X_j,d_j))_{j\in\N_0}$ be a sequence of locally compact metric spaces and $\ph\in \Lip(X)$ with a limit point $h\in X$.
Assume  that there exists $R>0$ such that the $B_j(R)$ are compact and $\ph_j:B_{j+1}(R)\to B_{j}(R)$ are uniform contractions with  contraction coefficient $c\in[0,1)$ for all $j\in\N_0$.
Then,
\begin{itemize}
\item  [(0.)]
$h$ is the unique limit point of $\ph$ in $B(R)$,
\item  [(1.)]
for every $\eps>0$ there exists $\de=\de(\eps)>0$ such that all $\psi\in \Lip(X)$ with $\sup_{j\in\N_0}d_{B_{j+1}(R),X_j}(\ph_{j},\psi_{j})\leq \de$ have a limit point $g\in B(\eps)$.
\end{itemize}
Moreover, for some $\eps>0$ let $U\subseteq \Lip(X)$ be such that $\sup_{j\in\N_0}d_{B_{j+1}(R),X_j}(\ph_{j},\psi_{j})\leq \de(\eps)$ for  $\psi\in U$ and $U_0\subseteq U$ be a dense subset with respect to componentwise convergence in $\Lip(B_{j+1}(R),B_{j}(R+\de))$, $j\in\N_0$. Suppose the limit points $g(\psi)\in B(R)$ of  $\psi\in U_0$ are unique in $X$ and the maps $g_j:U_0\to B_j(R)$, $\psi\mapsto g_j(\psi)$ are continuous.
Then,
\begin{itemize}
\item [(2.)]  the maps $g_j:U_0\to B_j(R)$ have continuous extensions $\oh g_j:U\to B_j(R)$ such that $\oh g(\psi)=(\oh g_j(\psi))$ is the unique limit point of $\psi\in U$ in $B(R)$.
\end{itemize}
Moreover, if all of the preceding assumptions  hold for every limit point $h$ of $\ph$, then
\begin{itemize}
\item [(3.)] $h$ is the unique limit point of $\ph$ in $X$.
\end{itemize}
\begin{proof}
Let $B:=(B_j)_{j\in\N_0}$ with $B_{j}:=B_j(R)$, $j\in\N_0$.

(0.) By Lemma~\ref{l:fixpoint} we know that $h$ is the unique limit point of $\ph$ in $B$.

(1.) For given $\eps\in(0,R]$ let $\de=(1-c)\eps$, where
$c\in[0,1)$ is the contraction coefficient from the assumption.
Let $\psi\in \Lip(X)$ be given such that $\sup_{j\in\N_0}d_{B_{j+1},X_{j}}(\ph_{j},\psi_{j})\leq \de$.
Then, for every $g\in B_{j+1}(\eps)$, we have
\begin{align*}
d_{j}\ap{h_{j},\psi_{j}(g)} &=d_{j}\ap{\ph_{j}(h_{j+1}),\psi_{j}(g)}\leq d_{j}\ap{\ph_{j}(h_{j+1}),\ph_{j}(g)} +d_{j}\ap{\ph_{j}(g),\psi_{j}(g)}\\
&\leq c\,d_{j+1}\ap{h_{j+1},g}+d_{B_{j+1},X_{j}}(\ph_j,\psi_j)\leq c\,\eps+(1-c)\eps=\eps.
\end{align*}
Hence, $\psi_j(B_{j+1}(\eps))\subseteq B_{j}(\eps)$.
This implies that $\psi$ has a limit point in $B(\eps)$.

(2.) Let $\psi\in U$ and $\psi^{(n)}\in U_0$ converge componentwise to $\psi$  with respect to the metrics $d_{B_{j+1},X_j}$. Let $g^{(n)}$ be the unique limit points of the sequence  $\psi^{(n)}$. By compactness and continuity the limit points $g^{(n)}\in X$ converge to some $g\in X$. Let us show that $g$ is a limit point of $\psi$. Let $\eps>0$, $j\in\N_0$ be arbitrary and $N\in\N$ such that $\sup_{j\in\N_0}d_{B_{j+1},X_j}(\psi_j,\psi_j^{(n)})\leq\eps/3$ and $d_{i}(g^{(n)}_{i},g_{i})\leq \eps/3$ for  $i\in\{j,j+1\}$ and all $n\geq N$. Since $(g^{(n)}_{j})$ are limit points and  $\psi_{j}^{(n)}$ are quasi contractions we have
\begin{align*}
d_{j}\ap{\psi_j^{(n)}(g_{j+1}),g_{j}^{(n)}}
=d_{j}\ap{\psi_j^{(n)}(g_{j+1}),\psi_j^{(n)}\ap{g_{j+1}^{(n)}}}
\leq d_{j+1}\ap{g_{j+1},g_{j+1}^{(n)}}\leq \frac{\eps}{3}.
\end{align*}
Therefore,  we get
\begin{align*}
d_{j}\ap{\psi_j(g_{j+1}),g_{j}}&\leq d_{j}\ap{\psi_j(g_{j+1}),\psi_j^{(n)}\ap{g_{j+1}}}+ d_{j}\ap{\psi_j^{(n)}(g_{j+1}),g_{j}^{(n)}}+ d_{j}\ap{g_{j}^{(n)},g_{j}}\\
&\leq \eps.
\end{align*}
Since $\eps$ was arbitrary, $g$ must be a limit point of $\psi$. \\
Suppose  $\psi\in U$ has two limit points $g$ and $g'$ in $B$. Then, by (1.) there is $\ow \psi\in U_0$ whose unique limit point is arbitrarily close to both $g$ and $g'$ for $\ow \psi$  sufficiently close to $\psi$. However,  this is only possible if $g$ and $g'$ coincide.

(3.) By the arguments above, every limit point of $\ph$ must be close to a limit point of $\psi$ whenever $\psi\in U_0$ is close to $\ph$. However,  since the limit points of $\psi\in U_0$ are unique in $X$, this is only possible if the limit point of $\ph$ is unique.
\end{proof}
\end{lemma}


\subsection{Comparison to the hyperbolic standard metric}

In this subsection we relate the semi metric space $(\h^J,\gm_J)$ to the metric space $(\h^J,\dist_{\h^J})$.\medskip

\begin{lemma}\label{l:compactness}(Comparison of $\gm$ and $\dist_{\h}$.)
Let $I,J$ be finite index sets.
\begin{itemize}
\item [(1.)] Let $\ph:(\h^{I},\gm_{I}) \to (\h^{J},\gm_{J})$ be an isometry (resp. a contraction) on $\h^{I}$. Then, $\ph: (\h^{I},\dist_{\h^{I}}) \to (\h^{J},\dist_{\h^{J}})$ is an isometry (resp. a contraction) on $\h^{I}$.
\item [(2.)] Let $\ph:(\h^{I},\gm_{I}) \to (\h^{J},\gm_{J})$ be a uniform contraction on a compact set $K$ with contraction coefficient $c_0$. Then, $\ph: (\h^{I},\dist_{\h^{I}}) \to (\h^{J},\dist_{\h^{J}})$ is a uniform contraction on $K$ and the contraction coefficient depends only on $\diam (K)$ and $c_0$.
\end{itemize}
\begin{proof} The first statement is clear by definition and monotonicity of $\cosh^{-1}$. To prove (2.) we start with the claim that the function
$$c:[0,\infty)\to[0,\infty),\quad r\mapsto\frac{\cosh^{-1}(c_0r+1)}{\cosh^{-1}(r+1)}.$$
is uniformly less than one on every bounded subset of $[0,\infty)$ whenever $c_0\in[0,1)$. The claim is obvious for bounded subsets of $(0,\infty)$ so we only have to check the case $r\to0$. To this end, we apply L'Hospitals theorem. The derivative of $x\mapsto\cosh^{-1}(x)$ is $x\mapsto \ap{x^2-1}^{-\frac{1}{2}}$ and therefore
$$\frac{\frac{d}{dr}\cosh^{-1}(c_0r+1)}{\frac{d}{dr}\cosh^{-1}(r+1)} =\frac{c_0\sqrt{r^2+2r}}{\sqrt{c_0^2r^2+2c_0r}}\to \sqrt{c_0},$$
as $r\to 0$. This proves the claim.\\
By definition of $\dist_{\h^J}$ and the assumptions of the lemma we have that
$$\frac{\dist_{\h^J}(\ph(g),\ph(h))}{\dist_{\h^{I}}(g,h)}\leq c\ab{r},$$
with $r=\gm_{I}(g,h)/2$ for all $g,h\in K$. Let $r_0=\diam(K)/2$. We know by the claim above that $c$ is uniformly smaller than one on $[0,r_0]$. This yields the existence of a contraction coefficient $c_1=\sup_{r\in[0,r_0]}c(r)<1$ which proves the statement.
\end{proof}
\end{lemma}

As $c(r)\to1$ for $r\to\infty$, for $c$ from the proof, the statement of  the previous lemma  becomes false if we drop the compactness assumption. Hence, approaching the boundary of $\h^J$, we lose any uniformity of contraction of $\gm_{J}$ in $\dist_{\h^J}$.


\subsection{A substitute for the triangle inequality}

As discussed above, $\gm $ does not satisfy the triangle inequality. We will prove a similar estimate instead. The idea is to estimate $\gm(g+\lm,h)$ by a linear function of $\gm(g,h)$ where the coefficients depend only on $h$ and $\lm$. The first inequality can be used to estimate diagonal perturbations, i.e., perturbations by a potential. The second one applies to off diagonal perturbations.
\medskip

\begin{lemma}\label{l:ti}(Substitute triangle inequality.) For all $g,h,\lm\in\h$ we have
\begin{align*}
\gm(g+\lm,h)\leq c_0(\lm)\gm(g,h)+(c_0(\lm)-1),
\end{align*}
with $c_0(\lm)=({1+{2|\lm|}/{\Im h}}^{2})$.
Moreover, for $\lm\in(-1,\infty)$,
\begin{align*}
\gm((1+\lm)g,h)\leq{(1+\lm)^{-1}}\ap{c_1(\lm){\gm(g,h)+(c_1(\lm)-1)}},
\end{align*}
with $c_1(\lm)=(1+2\sqrt2 {|\lm||h|}/{\Im h})^{2}$.
\begin{proof}
We start with the first inequality and distinguish two cases. If $|g-h|\geq \Im h/2$, then, by the triangle inequality of the modulus $|\cdot|$, we get
$$\gm(g+\lm,h)\leq \ab{1+\frac{|\lm|}{|g-h|}}^2 \gm(g,h)\leq \ab{1+\frac{2|\lm|}{\Im h}}^2 \gm(g,h).$$
If, on the other hand, $|g-h|\leq \Im h/2$, then $\Im g\geq\Im h/2$ and we obtain by direct computation
\begin{align*}
\gm(g+\lm,h)\leq \gm(g,h)+\frac{2|\lm||g-h|+|\lm|^2}{\Im g\Im h}\leq \gm(g,h)+2\frac{|\lm|\Im h+|\lm|^2}{(\Im h)^2}.
\end{align*}
The first inequality of the statement now follows from the definition of $c_0$.

Let us turn to the second inequality. If $|g-h|\geq \Im h/2$, then, by the triangle inequality of the modulus $|\cdot|$, we get
$$\gm((1+\lm) g,h)\leq\frac{1}{1+\lm} \ab{1+\frac{|\lm||g|}{|g-h|}}^2 \gm(g,h)\leq \frac{1}{1+\lm}\ap{1+|\lm|+\frac{|\lm||h|}{\Im h}}^2 \gm(g,h),$$
by estimating $|g|\leq|g-h|+|h|$ in the last step. On the other hand, if $|g-h|\leq \Im h/2$, then  $\Im g\geq\Im h/2$ and $|g|\leq 2|h|$, so we get
\begin{align*}
\gm((1+\lm) g,h)&\leq \frac{1}{(1+\lm)} \ap{\gm(g,h)+\frac{2|\lm||g||g-h|+|\lm|^2|g|^2}{\Im g\Im h}}\\
&\leq \frac{1}{(1+\lm)} \ap{\gm(g,h)+ \frac{4|\lm||h|}{\Im h}+\frac{8|\lm|^{2}|h|^{2}}{\Im h}}.
\end{align*}
By the definition of $c_1$ the second inequality of the statement follows.
\end{proof}
\end{lemma}

\subsection{The center of balls and euclidean distances}\label{ss:eps1}

Finally,  we  discuss two quantities which play a role in Section~\ref{s:kaproof}. This discussion will also give  an insight into the geometry of hyperbolic balls.

Let $J$ be a finite index set.
For $h\in\h^J$ and $r\geq 0$ we define
\begin{align*}
\eps_1:=\eps_1(r) &:= \inf_{g\in \h^{J}\setminus B_r(h)}\; \min_{j\in J}\; |g_j-h_j|,\\
\eps_2:=\eps_2(r) &:= \min_{g\in B_r(h)} \;\min_{j\in J}\; \Im g_j.
\end{align*}

The quantity $\eps_1$ describes the shortest euclidean distance from the boundary of a $\gm$-ball to the center and $\eps_2$ describes the shortest distance from the boundary of the ball to the `real line' boundary of $\h^{J}$. Figure~\ref{f:eta} illustrates the quantities $\eps_1$ and $\eps_2$ for the one dimensional case, i.e., $J$ being a singleton set. This will be proven in Lemma~\ref{l:eta_0dim}.

\begin{figure}[!h]
\centering
\scalebox{.4}{\includegraphics{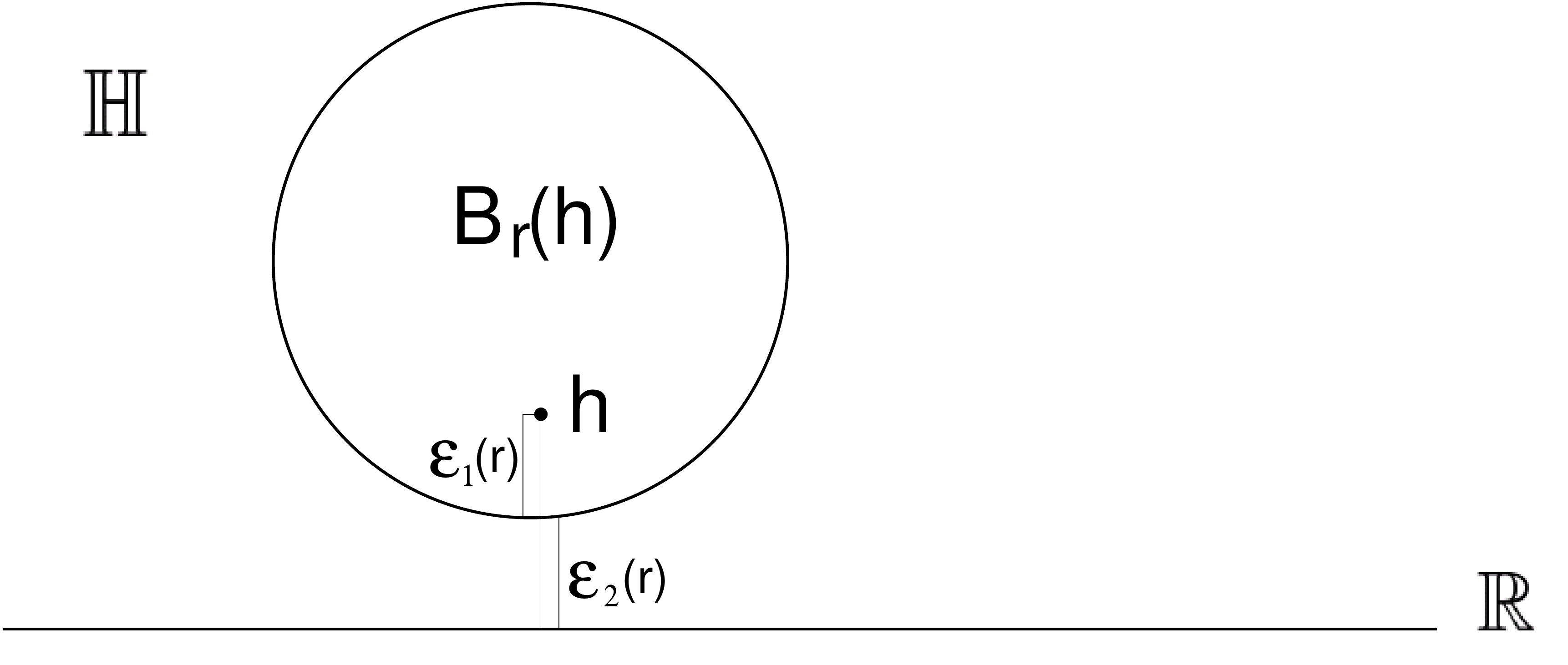}}
\caption{Illustration of the quantities $\eps_1$ and $\eps_2$ in the one dimensional case.} \label{f:eta}
\end{figure}

Let us discuss the properties of the corresponding quantities in the euclidean metric in the one dimensional case. Let $\eps_1^{\mathrm{eucl}}=\eps_1^{\mathrm{eucl}}(r)$ and $\eps_2^{\mathrm{eucl}}=\eps_2^{\mathrm{eucl}}(r)$ be quantities corresponding to $\eps_1$ and $\eps_2$, where the ball  is taken with respect to the euclidean metric.
If $\zeta$ is the center and $r$  is the radius  of a euclidean ball we have $\eps_1^{\mathrm{eucl}}(r)=r$ and
$\Im \zeta=\eps_1^{\mathrm{eucl}}(r)+\eps_2^{\mathrm{eucl}}(r)$.
Moreover, for $\eps_2^{\mathrm{eucl}}$ there is a unique $\xi\in B_r(h)$ for which the minimum is realized, while $\eps_1^{\mathrm{eucl}}$ is realized by all boundary points of a euclidean ball.

In the hyperbolic space this is in some aspects very different. First of all, notice that a $\gm$-ball with radius $r$ describes a hyperbolic $\dist_{\h}$ ball with radius $\cosh^{-1}(r/2+1)$. Below we find that $\eps_1$ and $\eps_2$ are realized by the same unique element in the boundary of the $\gm$-ball. For the radius $r$ and the center $\zeta$ of a $\gm$-ball we have that
$$r=\eps_1(r)^{2}/(\eps_2(r)\Im\zeta)\quad\mbox{ and } r=\eps_1(r)+\eps_2(r).$$ For the one dimensional case, this will be proven in the next lemma and, for the  higher dimensional situation, this will be shown afterwards.

Denote by $\dd B_r$ the boundary of $B_r(\zeta)=\{\xi\in\h\mid\gm(\xi,\zeta)\leq r\}$ for $\zeta\in\h$ and $r\geq0$.\medskip

\begin{lemma}\label{l:eta_0dim}(Hyperbolic balls - the one dimensional case.)
Let $\zeta\in\h$, $r\geq 0$ and $\xi_1,\xi_2\in \h$ be such that
\begin{align*}
|\xi_1-\zeta|= \inf_{\xi'\in \h\setminus B_r(\zeta)}|\xi'-\zeta|\;(=\eps_1(r))\qand
\Im\xi_2 = \min_{\xi'\in B_r(\zeta)}\Im\xi'\;(=\eps_2(r)).
\end{align*}
Then, $\xi_1$ and $\xi_2$ are equal to the unique element $\xi=\xi(\zeta)\in\dd_r B$ which satisfies
\begin{align}
\Re\xi&=\Re \zeta,\label{e:Re}\\
\Im \xi&=\Im \zeta -|\xi-\zeta|, \label{e:Im}
\end{align}
i.e., $\xi=\xi_1=\xi_2$. Moreover, the functions $\zeta\mapsto|\xi(\zeta)-\zeta|$ and $\zeta\mapsto\xi(\zeta)$ are constant in $\Re \zeta$ and are increasing in $\Im \zeta$.
\begin{proof}
It is clear that $\xi_1,\xi_2\in \dd B_r$, i.e.,
\begin{align*}
|\xi_1-\zeta|= \min_{\xi'\in \dd B_r} |\xi'-\zeta| \qand
\Im\xi_2 = \min_{\xi'\in \dd B_r} \Im \xi'.
\end{align*}
We first show that, if $\xi_1$ and $\xi_2$ satisfy \eqref{e:Re}, then they also satisfy \eqref{e:Im}. This can be seen since there are only two points $\xi_\pm\in \dd B_r$ which satisfy \eqref{e:Re}. In particular,  $\Im \xi_\pm=\Im \zeta \pm|\xi_\pm-\zeta|$ which implies $\Im \xi_+\geq \Im \xi_-$ and $|\xi_+-\zeta|\geq|\xi_--\zeta|$. Thus, $\xi_1,\xi_2$ both satisfy \eqref{e:Im} and $\xi_1=\xi_2=\xi_-$ if $\xi_1,\xi_2$ satisfy \eqref{e:Re}.
This also shows that the element $\xi(\zeta)$ is indeed unique.\\
We continue by showing that $\xi_2$ satisfies \eqref{e:Re}. The idea is to write $ \xi_2$ as $\Re \zeta+x+iy$ for suitable $(x,y)\in\R\times\R^+$. Indeed, the function $(x,y)\mapsto y$ attains its minimum over the set
$$R:=\ac{(x,y)\in\R\times\R_+\mid \gm(\Re \zeta+x+iy,\zeta)=\frac{ x^2+(y-\Im \zeta)^2}{y \Im \zeta}=r}$$
at $(0,y')$ where $y'$ is the smaller root of the quadratic polynomial $\ph:y\mapsto y^2-\Im \zeta(r+2)y+(\Im \zeta)^{2}$. Hence, $\xi_2=\Re \zeta +iy'$, i.e., $\Re \xi_2=\Re\zeta$ which implies $\xi_2=\xi$ by the uniqueness of $\xi$.\\
By writing out the formula for $y'$ and the discussion above, it can be  seen that the function $\zeta\mapsto\Im\xi(\zeta)=y'(\zeta)$ is constant in $\Re \zeta$ and uniformly increasing in $\Im \zeta$.\\
We now turn to $\xi_1$ which we also write as $\Re \zeta+x+iy$. The function $(x,y)\mapsto |\Re \zeta+x+iy-\zeta|=\sqrt{x^2+(y-\Im \zeta)^2}$ attains its minimum over $R$ at $(0,y'')$ for some $y''>0$.
This implies $\xi_1=\Re \zeta +iy''$, i.e., $\Re \xi_1=\Re\zeta$ which implies $\xi_1=\xi_2=\xi$ again by the uniqueness of $\xi$. Moreover, we have that $y''=y'$ is the minimum of the quadratic polynomial $\ph$ from above. Hence, it can be checked by direct computation, that the function $\zeta\mapsto |\xi(\zeta)-\zeta|=|y'(\zeta)-\Im\zeta|$ is constant in $\Re \zeta$ and monotone increasing in $\Im \zeta$.
\end{proof}
\end{lemma}

The next lemma considers the higher dimensional case. We will prove three formulas. The third one plays a crucial role in the proof of Proposition~\ref{p:al>0} in  Chapter~\ref{c:main3}. It tells us how to compute the radius $r$ of a $\gm$ ball for a given value of $\eps_1(r)$.
The proof of the lemma follows basically from the one dimensional case.
\medskip

\begin{lemma}\label{l:eta}(Hyperbolic balls - the higher dimensional case.) Let $h\in\h^J$ and $\eps_0:=\min_{j\in J}\Im h_j$. Then
\begin{itemize}
\item [(1.)] $\eps_1(r)+\eps_2(r)=\eps_0$. In particular, $\eps_1(r),\eps_2(r)\in(0,\eps_0)$ for all $r> 0$.
\item [(2.)] $\eps_1(r)^2= r\eps_0\eps_2(r)$ for all $r>0$.
\item [(3.)] For given $\de\in(0,\eps_0)$  we have that $\eps_1(r)=\de$ for $r=\frac{\de^{2}}{(\eps_0-\de)\eps_0}$.
\end{itemize}
\begin{proof}
We start with an argument which shows that it suffices to consider the one dimensional case. Let $h\in\h^J$ and $g\in \h^J$ be in the boundary of $B_r(h)$ such that for each $j\in\A$ the components $g_j$ equal $\xi=\xi_1=\xi_2$ from Lemma~\ref{l:eta_0dim} for $\zeta=h_j$, $j\in J$. Let $j_0\in J$ be such that $\Im h_{j_0}=\eps_0$. By the statements of Lemma~\ref{l:eta_0dim} about the monotonicity behavior of $\eps_1$ and $\eps_2$ with respect to $\Re \zeta$ and $\Im \zeta$ in the one dimensional case we deduce that the second minimum in the definition of $\eps_1$ and $\eps_2$ is realized by the ${j_0}$-th component of $g$. Hence, we only have to check the statements in the one dimensional case.

(1.) The statement follows now directly from \eqref{e:Im}.

(2.) Recall that $g_{j_0}$ is in the boundary of $ B_r(h_{{j_0}})$, i.e., $\gm(g_{j_0},h_{j_{0}})=r$. Using the definition of $\gm$, we get
\begin{equation*}
\eps_1=|g_{j_{0}}-h_{j_{0}}|=\sqrt{r\Im g_{j_{0}}\Im h_{j_{0}}}=\sqrt{r\eps_0\eps_2}.
\end{equation*}

(3.) We employ (1.) into (2.) to get
$\eps_1(r)^{2}={r\eps_0(\eps_0-\eps_1(r))}$. We resolve the identity with respect to $\eps_1$. The positive solution gives a formula for $\eps_1$ depending on $r$. We set this formula equal to $\de$ and resolve it this time with respect to $r$ and obtain (3.).
\end{proof}
\end{lemma}

\section{Contraction properties of the recursion map}\label{s:recursionmap}

The concepts introduced in this section are fundamental for the analysis of the upcoming chapters. We decompose the recursion maps $\Psi_{z}$ to study their contraction properties with respect to the semi metric $\gm$. We are interested in contractions for several reasons. First of all, we want to show that $\Gm(z,H)$ is the unique fixed point of $\Psi_{z}$ for $z\in\h$ and thus the unique vector in $\h^{\V}$ which satisfies the recursion formula \eqref{e:Gm}. Secondly, we use contraction properties to study continuity of $\Gm(z,H)$ as $\Im z\downarrow 0$. The third reason concerns the stability of $\Gm(z,H)$ in the limit $\Im z\downarrow 0$ under small perturbations of $H$.



\subsection{A decomposition and criterions for contraction}\label{s:Contraction}

Let $z=E+i\eta\in\h\cup\R$ and $g,h\in\h^{S_{x}}$ for some $x\in\V$. In this subsection we
derive a formula of the following type
\begin{align*}
\gm(\Psi_{z,x}(g),\Psi_{z,x}(h))=c(\eta)\sum_{y\in S_{x}}p_{x,y}(h)c_{x,y}(g,h)\gm(g_{y},h_{y}),
\end{align*}
where $c(\eta)\in(0,1)$ is called the \emph{contraction coefficient} which tends to one as $\eta\downarrow0$. Moreover, $p_{x,y}=p_{x,y}(h)\in(0,1)$ are weights, i.e., $\sum_{y\in S_{x}}p_{x,y}=1$ and $c_{x,y}=c_{x,y}(g,h)\in[-1, 1]$ are called the \emph{contraction quantities}. In the case $\eta>0$ we have `contraction' by $c(\eta)<1$. In order to show contraction in the limit case $\eta\downarrow0$ one has to prove that one of the contraction quantities $c_{x,y}$ is uniformly smaller than one. Moreover, we have to ensure that the corresponding weight $p_{x,y}$ stays bounded from below to make the contraction `visible'. We will call the sum on the right hand side a \emph{contraction sum}.

Let us be more precise and introduce the quantities sketched above. Let
$g,h\in \h^\V$. We might think of $h$ as a quantity we have control over and of $g$ as a free variable.

To shorten notation we will write for $x\in\V$
\begin{equation*}
\gm_x:=\gm(g_x,h_x).
\end{equation*}
We introduce the weights for $x,y\in \V$ with $y\in S_x$ by
\begin{equation*}
p_{x,y}:=p_{x,y}(h):= \frac{|t(x,y)|^{2}\Im h_y}{\sum_{v\in S_x}|t(x,v)|^{2}\Im h_v}.
\end{equation*}
Clearly, $\sum_{y\in S_x}p_{x,y}=1$. Moreover, define
\begin{align*}
q_{x,y}:=q_{x,y}(g):=p_{x,y}(g)=\frac{|t(x,y)|^{2}\Im g_y}{\sum_{v\in S_x}|t(x,v)|^{2}\Im g_v}.
\end{align*}
Define the contraction quantities
\begin{align*}
c_{x,y}:=c_{x,y}(g,h):=\sum_{v\in S_x}q_{y,v} Q_{y,v}\cos\al_{y,v},
\end{align*}
where  $Q_{x,y}$ is given by
\begin{equation*}
Q_{x,y}:=Q_{x,y}(g,h):=\frac{\ab{\Im g_x\Im g_y\Im h_x\Im h_y\gm_x\gm_y}^\frac{1}{2}}{\frac{1}{2}\ab{\Im g_x\Im h_y\gm_y+\Im g_y\Im h_x\gm_x}},
\end{equation*}
for arbitrary $x,y\in \V$ with $\gm_{x},\gm_{y}\neq0$, and $Q_{x,y}=0$ otherwise.
Note that $Q_{x,y}$ is the quotient of a geometric and an arithmetic mean. Therefore,  $Q_{x,y}\in[0,1]$.

Next, we turn to the definition of the \emph{relative arguments}
$\al_{x,y}$. The argument  of non zero complex numbers is defined as a map $\arg:\C\setminus\{0\}\to\R/2\pi\Z\cong \Sp^{1}$ which is a continuous group homomorphism.  For  $x,y\in \V$ with $g_{x}\neq h_{x}$ and $g_{y}\neq h_{y}$, we let
\begin{align*}
\al_{x,y}:=\al_{x,y}(g,h):=\arg{(g_x-h_x)\ov{(g_y-h_y)}}.
\end{align*}
The expression $\cos\al_{x,y}$ is well defined as the cosine is a $2\pi$-periodic function.
Note that there is no reasonable way to define $\al_{x,y}$ for $g_{x}=h_{x}$ or $g_{y}=h_{y}$. However,  in the formula of the contraction coefficients $c_{x',x}$ the $\al_{x,y}$'s always come accompanied by $Q_{x,y}$ which is zero in the critical case anyway. Therefore, we set the product $Q_{x,y}\cos\al_{x,y}$ to be zero  whenever $g_{x}=h_{x}$ or $g_{y}=h_{y}$.

We introduce a modulus function $\ma{\cdot}$ for the arguments. For $\al\in \Sp^{1}$ we let
\begin{align*}
\ma{\al}:=d_{\Sp^{1}}(\al,1),
\end{align*}
where $d_{\Sp^{1}}(\cdot,\cdot)$ is the canonical translation invariant metric in $\Sp^{1}$. Of course,   we have a triangle inequality for $\al,\be\in\Sp^{1}$
\begin{align*}
\ma{\al+\be}\leq\ma{\al}+\ma{\be}.
\end{align*}
Note that $\Sp^{1}$ can be mapped bijectively to $[-\pi,\pi)$ which makes it reasonable to write  $\arg\xi\in (a,b)$ (respectively $\arg\xi\not\in(a,b)$) for $\xi\in\C\setminus\{0\}$ and $(a,b)\subseteq [-\pi,\pi)$ if the canonical bijection $\Sp^{1}\to[-\pi,\pi)$ maps the $\arg \xi$ into $(a,b)$ (respectively into $[-\pi,\pi)\setminus (a,b)$).
A corresponding convention applies to closed intervals.

By these definitions, $c_{x,y}$ itself can be considered as a weighted sum of contraction quantities $Q_{x,y}\cos \al_{x,y}\in[-1,1]$ and weights $q_{x,y}$ adding up to one. Thus, $c_{x,y}\in[-1,1]$.
Observe that the quantities $Q_{x,y}$ and $\cos \al_{x,y}$ are symmetric in $x$ and $y$, while the $x$ and $y$ in the quantities $p_{x,y}$ and $c_{x,y}$ are elements of different spheres.

It is obvious that the quantities defined above depend only on finitely many components of $g,h\in \h^\V$. Therefore,  with slight abuse of notation, we sometimes write these quantities as functions of the components they actually depend on. For example, we write $c_{x,y}(g',h')$ with $g',h'\in \h^{S_x}$.

We can now decompose the maps $\Psi_{z}$ into an averaging, a  translating and a reflecting part, i.e.,
\begin{align*}
\Psi_{z}=\tau\circ\si_{z}\circ \rho,
\end{align*}
as it can be found in \cite{FHS1}. Precisely, the components of $\rho$, $\si_{z}$ and $\tau$ in $x\in S^{n}$ are given as functions
\begin{align*}
&\rho_x:\h \to\h,&&\hspace{-3.5cm} g\mapsto -\frac{1}{g}, \\
& \si_{z,x} :\h \to\h,&&\hspace{-3.5cm} g\mapsto z-w(x)+g, \\
&\tau_x: \h^{S_x}\to\h, &&\hspace{-3.5cm} g\mapsto \sum_{y\in S_x}|t(x,y)|^2 g_y.
\end{align*}
We denote the restrictions of $\rho$, $\si$ and $\tau$ to the subspaces  $\h^{S^{n}}$ induced by the spheres $S^{n}$ by $\rho_{S^{n}}$, $\si_{z,S^{n}}$ and $\tau_{S^{n}}$.
In particular, $\rho_{S^{n}}$, $\si_{z,S^{n}}$ are maps $\h^{S^{n}}\to \h^{S^{n}}$ and $\tau_{S^{n}}$ is a map $\h^{S^{n+1}}\to \h^{S^{n}}$.

The basic contraction properties of $\Psi_{z,x}$ and its decomposition are summarized in the next proposition.
\medskip

\begin{lemma}\label{l:Psicontr}(Basic contraction properties.)
Let $x\in\V$, $z\in\h\cup\R$. Then, $\Psi_{z,x}:(\h^{S_x},\gm_{S_{x}})\to(\h,\gm)$ is a quasi contraction. More precisely,
\begin{itemize}
\item [(1.)] $\rho_x:(\h,\gm)\to (\h,\gm)$ is an isometry.
\item [(2.)] $\si_{z,x}:(\h,\gm)\to (\h,\gm)$ is an isometry if $\Im z=0$ and a contraction if $\Im z>0$ which is uniform on sets $\set{h\in\h\mid \Im h\leq C}$ with contraction coefficient $c_0=(1+\Im z/C)^{-2}.$
\item [(3.)] $\tau_x:(\h^{S_x},\gm_{S_x})\to (\h,\gm)$ is a quasi contraction with
\begin{equation*}\label{e:tau}
\gm(\tau_{x}(g),\tau_{x}(h))=\sum_{y\in S_x}p_{x,y}(h)c_{x,y}(g,h)\gm(g_y,h_y).
\end{equation*}
\end{itemize}
\begin{proof}
(1.) Since $ 1/\xi=\ov{\xi}/|\xi|^2$ for $\xi\in\C\setminus\{0\}$ we have for $g,h\in\h$
\begin{equation*}
\gm(\rho(g),\rho(h))= \frac{|g|^{-2}|h|^{-2}|\overline{g}-\overline{h}|^2}{\Im g|g|^{-2}\Im h|h|^{-2}} =\gm(g,h).
\end{equation*}
(2.) A direct calculation yields
\begin{align*}
\gm(\si_{z,x}(g),\si_{z,x}(h))&=\frac{|g-h|^2}{\Im (g+z)\Im(h+z)}=\frac{1}{(1+\Im z/\Im g )(1+\Im z/\Im h)}\gm(g,h).
\end{align*}
The coefficient on the right hand side is smaller than one if and only if $\Im z>0$. On the set $U_C=\{g\mid\Im g\leq C\}\subset \h$ the map $\si_{z,x}$ is a uniform contraction with contraction coefficient $c_0=\sup_{g\in U_C}(1+\Im z/\Im g )^{-2}=(1+\Im z/C)^{-1}$.

(3.) We  compute
\begin{align*}
{\av{\sum_{y\in S_x} |t(x,y)|^{2} (g_y-h_y)}^2}
&=\sum_{y,v\in S_x} |t(x,y)|^{2}|t(x,v)|^{2}\av{g_y-h_y}\av{g_v-h_v}\cos\al_{y,v}\\
\end{align*}
\begin{align*}
\ldots&=\sum_{y,v\in S_x} |t(x,y)|^{2}|t(x,v)|^{2}(\Im g_y\Im g_v\Im h_y\Im h_v\gm_y\gm_v)^{\frac{1}{2}} \cos\al_{y,v}\\
&=\sum_{y,v\in S_x}|t(x,y)|^{2}|t(x,v)|^{2}\ab{\frac{1}{2}\ab{\Im g_y\Im h_v\gm_v}+\frac{1}{2}\ab{\Im g_v\Im h_y\gm_y}}Q_{x,y}\cos\al_{y,v}\\
&=\sum_{y\in S_x}|t(x,y)|^{2}\Im h_y \ab{\sum_{v\in S_x}|t(x,v)|^{2}\Im g_v Q_{y,v}\cos\al_{y,v}}\gm_y.
\end{align*}
Dividing this expression by $\Im\tau_{x}(g)\Im\tau_{x}(h)$ and plugging in the definitions of  yields the formula of (3.).
We continue to estimate using the formula of (3.)
\begin{align*}
\gm(\tau_x(g),\tau_x(h))&
= \ab{\sum_{y\in S_x} p_{x,y}(h)c_{x,y}(g,h)\frac{\gm(g_y,h_y)}{\gm_{S_x}(g,h)}} \gm_{S_x}(g,h)\\
&\leq\ap{ \sum_{y\in S_x} p_{x,y}(h)}
\gm_{S_x}(g,h)= \gm_{S_x}(g,h),
\end{align*}
where the first inequality follows since $c_{x,y}(g,h)\leq1$ and
${\gm(g_y,h_y)}\leq{\gm_{S_x}(g,h)}$ by definition of $\gm_{S_x}$. The last equality follows
by $\sum_{y\in S_x}p_{x,y}(h)=1$. Hence, we have shown that $\tau_{x}$ is a quasi contraction.
\end{proof}
\end{lemma}

Let us end this section with a discussion about what we can learn from the previous proposition about the contraction properties of $\Psi_{z}$.

It follows that $\Psi_{z}$ is a contraction for $\Im z>0$. The set where the contraction is uniform will turn out to be large enough  to show uniqueness of a fixed point if the operators have moderate off diagonal growth. This will be done in the next section.

Moreover, $\si_x\circ\rho_x$ is a hyperbolic isometry for $\Im z=0$. So, in the case where $\Psi_{z,x}$ stays a uniform contraction in the limit $\Im z\downarrow0$, the map $\tau_x$ must be a uniform contraction. Let us take a closer look at $\tau_x$. According to Proposition~\ref{l:Psicontr} it can be expanded into a contraction sum with weights $p_{x,y}$ and contraction quantities $c_{x,y}$ which are contraction sums on their own. Let us discuss under which circumstances we can expect $\tau_{x}$ to be a contraction.

Let us first explain that there is a subset where $\tau$ is not a contraction. On the subset of $\h^{S_x}$ of elements with linearly dependent components over $\R$, the map $\tau_x$ is multiplication by a positive number, which is a hyperbolic isometry. This can be seen right away. Assume that the components of $g$ and $h$ are linearly dependent over $\R$, i.e., there are $g_0,h_0\in\h$ and $\be_y>0$, $y\in S_x$ such that $g_y=\be_y g_0$ and $h_y=\be_y h_0$.
Then, for all $y\in S_x$, one easily checks
\begin{equation*}
\gm(\tau_x(g),\tau_x(h))=\gm(g_0,h_0)=
\gm(g_y,h_y)=\gm_{S_x}(g,h).
\end{equation*}
(Indeed, it can easily be seen that the set of pairs $(g,h)\in\h\times\h$ for which $\tau_x$ is not contracting is even larger.)
We can decompose $\tau_x $ into a multiplication and an averaging part via $\tau_x=\be \cdot\tau_x(g/\be)$, where $\be=|S_x|$. As seen above, multiplication by $\be$ is a hyperbolic isometry.
In the euclidean distance, the average of two vectors $g,h$ of complex numbers is a contraction if and only if the components of the differences are not all pairwise linearly dependent over $\R$. This means there are $y$ and $y'$ such that the argument of the complex numbers $g_y-h_y$ and $g_{y'}-h_{y'}$ is non zero, i.e., $\cos\al_{y,y'}<1$. This is also true in the hyperbolic case, in particular, we have in this case $\cos\al_{y,y'}<1$. Thus, $c_{y,y'}<1$ which yields $\gm(\tau_x(g),\tau_x(h))<\gm(g,h)$. Of course,  this still does not tell us anything about uniformness of the contraction.

The considerations above can teach us two things. First of all, $\tau_x$ is definitely no contraction on $\h^{S_x}$. Secondly, one can try to show  that multiple applications of $\Psi_{z}$ map the linearly dependent elements into linear independent ones. In this case $\tau_{x}$ is a contraction in at least multiple steps. Such a strategy will be pursued in the analysis of label invariant operators in Section~\ref{s:LI_RecursionMaps}.

\subsection{Uniqueness of fixed points}\label{s:UniquenessFixPoints}

We now prove that the recursion map $\Psi_{z}^{(H)}$ of a given $H$ has a unique fixed point in $\h^\V$ for $z\in\h$ under suitable conditions. This directly implies that there is a unique vector in $\h^\V$ satisfying recursion relations \eqref{e:Gm} which is the truncated Green function $\Gm(z,H)$. Likewise, $\Gm(z,H)$ is then the unique solution of the system of polynomial equations \eqref{e:Polynom}. Moreover, as a direct consequence, the operator $H$ is essentially self adjoint.

The following statement is found in \cite{FHS1} for a class of operators on graphs which satisfy certain  assumptions. These assumptions allow for much more general graphs than trees. But in contrast to our setting, they also imply that the operators are bounded and the considered Hilbert space is $\ell^2(\V)$, i.e., the measure $\nu\equiv1$.

Define
\begin{equation*}
t_n:=\sup_{x\in S^{n}}\sum_{y\in S_x}|t(x,y)|^2.
\end{equation*}
We say an operator $H$ acting as \eqref{e:H} has \emph{moderate off diagonal growth} if
\begin{align*}\label{e:sum}
\limsup_{n\to\infty} \frac{1}{t_{n+1}}\sum_{k=0}^{n}\frac{1}{t_k} =\infty.
\end{align*}

\medskip

\begin{thm}\label{t:uniqueness}(Uniqueness of fixed points.)
Suppose that $H$ has moderate off diagonal growth.
Then, $\Psi_{z}=\Psi^{(H)}_z$ has a unique fixed point in $h^{\V}$ for  $z\in\h$. This fixed point is the truncated Green function $\Gm(z,H)\in\h^\V$. Its restriction to the spheres $\Gm_{S^{j}}$, $j\in\N_0$, can be computed for arbitrary $g\in\h^\V$ via the limit
\begin{align*}
\Gm_{S^{j}}(z,T)=\lim_{n\to\infty}\Psi_{z,S^{j}}\circ\ldots\circ\Psi_{z,S^{j+n}}(g_{S^{j+n+1}}).
\end{align*}
\begin{proof}
Let us fix some notation at the beginning.
Denote $z=E+i\eta$, $U({C}):=\{g\in\h^{\V}\mid |g_x|\leq C, x\in\V\}$ for $C>0$ and by $U^{n}(C)$ the canonical projection of $U(C)$ into $\h^{S^{n}}$. Moreover, denote $\Psi:=\Psi_{z}^{(H)}$. We start with a claim.

\emph{Claim 1: $\Psi$ maps $\h^{\V}$ into $U({{1}/{\eta}})$.} \\
Proof of Claim 1. Let $g\in\h^{S_{x}}$ and calculate
\begin{align*}
\mo{\Psi_{x}(g)}^2
&=\frac{1}{\ap{E-w(x)+ \Re \tau_x(g)}^2+\ap{\eta+ \Im \tau_x(g)}^2}\leq \frac{1}{\eta^2},
\end{align*}
as $ \Im\tau_x(g)>0$. This proves the claim.

\emph{Claim 2: $\Psi_{S^{n}}$ is a uniform contraction on the set $\Psi_{S^{n+1}}(\h^{S^{n+2}})$ with contraction coefficient $(1+\eta^2/t_n)^{-2}$.}\\
Proof of Claim 2. By Claim~1 we have that $\Psi_{S^{n+1}}(\h^{S^{n+2}})\subseteq U^{n+1}({{1}/{\eta}})$. One directly checks that $\tau_{S^{n}}$ maps $U^{n+1}({{1}/{\eta}})$ into $U^{n}({{t_n}/{\eta}})$. By Lemma~\ref{l:Psicontr}, the map $\si_{z,S^n}$ is a uniform contraction on $U^{n}({{t_n}/{\eta}})$ with contraction coefficient $(1+\eta^2/t_n)^{-2}$. As $\rho$ is an isometry,  the map $\Psi_{S^{n}}$ on $\Psi_{S^{n+1}}(\h^{S^{n+1}})$ is a contraction with contraction coefficient $(1+\eta^2/t_n)^{-2}$.

\emph{Claim 3: $\Psi_{S^{n}}\circ\Psi_{S^{n+1}}$ maps $\h^{S^{n+2}}$ into a set $B\subseteq\h^{S^{n}}$ with $\diam (B)\leq (2t_n/{\eta^2})^2$, (where the $\diam$ is taken with respect to $\gm$).}\\
Proof of Claim 3.
In the proof of Claim~2 we found that
\begin{align*}
\tau_{S^{n}}\circ\Psi_{S^{n+1}}\ap{\h^{S^{n+2}}}\subseteq \tau_{S^{n}}(U^{n+1}({{1}/{\eta}}))\subseteq  U^{n+1}({{t_{n}}/{\eta}}).
\end{align*}
We estimate
\begin{equation*}
\sup_{g,h\in U^{n}({{t_n}/{\eta}})}\gm\ap{\si_{S^{n}}(g),\si_{S^{n}}(h)}
\leq\sup_{g,h\in U^{n}({{t_n}/{\eta}})} 2\frac{|g|^{2}+|h|^2}{\eta^{2}}=\frac{4}{\eta^4}t_n^2.
\end{equation*}
For the last equality one first checks that the inequality '$\leq$' holds and then one finds the elements which make the estimate sharp.
Since $\rho_{S^{n}}$ is an isometry we obtain, by putting the two arguments together,
\begin{equation*}
\sup_{g,h\in \h^{S^{n+2}}} \gm\ap{\Psi_{S^{n}}\ap{\Psi_{S^{n+1}}\ap{g}}, \Psi_{S^{n}}\ap{\Psi_{S^{n+1}}\ap{h}}}\leq \frac{4}{\eta^4}t_n^2.
\end{equation*}
This proves the claim.


We now want to apply the uniqueness of limit points, Lemma~\ref{l:fixpoint}, with $X_j=\clos(\Psi_{S^{j}}\circ\Psi_{S^{j+1}} (\h^{S^{j+2}}))$, $d_{j}=\gm_{S^{j}}$ and $\ph_{j}=\Psi_{S^{j}}$ for $j\in\N_0$. Here, $\clos U$ of a set $U$ means the closure of $U$.
By Claim~2, we have that $\ph_{j}$ is a uniform contraction with contraction coefficient $c_{j}=(1+\eta^2/t_{j})^{-2}$.
By Claim~3, the spaces $X_j$ are compact and $x_j=\diam (X_j) \leq(2t_{j}/{\eta^2})^2$. One now easily checks that  by the assumption of moderate off diagonal growth we get
\begin{align*}
\limsup_{k\to\infty} \frac{1}{\sqrt{x_{k}}} \sum_{j=0}^{k-1}\frac{(1-\sqrt{c_{j}})}{\sqrt{c_{j}}}
=2\eta^2 \limsup_{j\to\infty} {\frac{1}{t_{k}} \sum_{j=0}^{k-1}\frac{1}{t_{j}}}=
\infty.
\end{align*}
Hence, the assumptions of  Lemma~\ref{l:fixpoint} are fulfilled with $\be=1/2$.  This yields the uniqueness of the limit point $\Gm({z},H)$ and the formula for the restrictions to the spheres. As the limit point is unique, it is a fixed point of $\Psi_{z}$.
\end{proof}
\end{thm}

Before providing some examples of operators which satisfy the assumption of the theorem above, we want to give two immediate corollaries. The first one follows directly from the definitions of Section~\ref{s:Recursion}.
\medskip

\begin{cor}\label{c:uniqueness}(Uniqueness of solutions.)
If $H$ has moderate off diagonal growth, then the recursion formulas \eqref{e:Gm} and the system of polynomial equations \eqref{e:Polynom} have a unique solution which is given by $\Gm(z,H)$ for each $z\in\h$.
\end{cor}

Moreover, the theorem above gives a sufficient criterion for essential self adjointness of the operator $H$ on $c_{c}(\V)$.
\medskip

\begin{cor}\label{c:suffessSA}(Essential self adjointness.) If the operator  $H$ has moderate off~diagonal growth and $c_c(\V)\subseteq D(H)$, then  $H$ is essentially self adjoint on $c_{c}(\V)$.
\begin{proof} Suppose there are two self adjoint operators on $\ell^{2}(\V,\nu)$ acting as \eqref{e:H}. In this case they share the same recursion map. By the theorem above their truncated Green functions agree. Hence, by Corollary~\ref{c:suffessSA} the operators coincide.
\end{proof}
\end{cor}

The assumption of moderate off diagonal growth has some similarity to a criterion for  essential self adjointness of Jacobi matrices on $\ell^{2}(\N)$ or  $\ell^{2}(\Z)$  (see \cite[p.504, Theorem 1.3]{Be}. It would be interesting to give a unified treatment of these assumptions.\medskip

\begin{Ex}
(1.) Let the operator $H$ have moderate off diagonal growth. Then, any self adjoint operator $H+v$, where  $v$ is the multiplication by a potential, has moderate off diagonal growth.

(2.) Let  $t_n\leq C$ for some $C\geq0$ and all $n\in\N$. In this case, $H$ has moderate off diagonal growth. Indeed, it suffices that there exists a sequence $n_k$ such that $t_{n_k}\leq C$ for all $k\in\N$.
Consider, for example, the adjacency operator $A$ or the Laplace operator $\De$ defined on a tree whose vertex degree is a function of the distance to the root. If the tree satisfies
\begin{equation*}
\deg(x)=\left\{
\begin{array}{ll}
2 &: |x| \mbox{ prime}, \\
|x| &: \mbox{else,} \\
\end{array}
\right.
\end{equation*}
then $A$ and $\De$ have moderate off diagonal growth.

(3.) Assume that $t_n\sim n^\be$ for some $\be<1/2$ is satisfied. Since
$\sum_{j=1}^{n}j^{-\be}\sim\int_{0}^{n}x^{-\be}dx=(1-\be)^{-1}n^{1-\be}$ we get
\begin{equation*}
\frac{1}{t_{n+1}}\sum_{j=1}^{n}\frac{1}{t_j}\sim \frac{n}{(n(n+1))^{\be}}\sim n^{1-2\be}\to\infty, \quad n\to\infty.
\end{equation*}
It follows that $H$ has moderate off diagonal growth. For instance
the operators $A$ or $\De$ defined on a tree which vertex degree satisfies $\deg(\cdot)\leq(\cdot)^{\be}$ with  $\be<1/2$ have moderate off diagonal growth.
\end{Ex}
\chapter{Label invariant operators and label radial symmetric potentials}\label{c:freeoperator}
 \begin{quote}
\begin{flushright}
\scriptsize{He took our sins on himself, giving his body to be nailed on the \emph{tree}, so that we, being dead to sin, might have a new life in righteousness, and by his wounds we have been made well.} 1 Peter 2:24   \end{flushright}
\end{quote}

Our setting in this chapter is the following: Let $\A$ be a finite set and $M:\A\times\A\to \N_0$ a substitution matrix which excludes the one dimensional situation, has positive diagonal and is primitive, i.e., it satisfies (M0), (M1) and (M2). For each $j\in\A$ we get a tree $\T=\T(M,j)$ with vertex set $\V$ and labeling function $a:\V\to\A$. Choose a label invariant measure $\nu$ on $\V$ and let $T:\ell^{2}(\V,\nu)\to\ell^{2}(\V,\nu)$ be a label invariant operator acting as
\begin{align*}
(T\ph)(x)=\sum_{y\sim x}t(x,y)\ph(y) +w(x)\ph(x).
\end{align*}
Recall that label invariance means, apart from the tree compatibility (T0), i.e., $t(x,y)\neq0$ if and only if $x\sim y$ and the symmetry (T1), i.e., $t(x,y)\nu(x)=\ov{t(y,x)}\nu(y)$,  that there is a matrix $(m_{j,k})_{j,k\in\A}$ and a vector $(m_{j})_{j\in\A}$ such that
\begin{align*}
    m_{a(x),a(y)}=|t(x,y)|^{2}M_{a(x),a(y)}\qand m_{a(x)}=w(x),
\end{align*}
for all $x,y\in \V$, $x\sim y$, which we referred to as label invariance (T2).
Note that a label invariant operator is always bounded. In particular, they have moderate off~diagonal growth. So, everything which was proven in the previous section applies to our present setting.

The aim of the chapter is to prove Theorem~\ref{main1} and Theorem~\ref{main2}. Recall that Theorem~\ref{main1} claims that a label invariant operator $T$ has pure absolutely continuous spectrum on finitely many intervals. On the other hand, Theorem~\ref{main2} claims that the absolutely continuous spectrum of $T$ is stable under small perturbations of label radial symmetric potentials $v\in\W_{\mathrm{sym}}(\T)$, i.e.,  the operator $T+\lm v$ exhibits pure absolutely continuous spectrum on compact intervals included in an open subset of $\si(T)$ for sufficiently small $\lm\ge0$ whenever $T$ is a non regular tree operator.

To this end, we  study the Green functions $\Gm_x=\Gm_x(z,H)=\ip{\de_x,(H_{\T_x}-z)^{-1}\de_x}$  and $G_x=G_x(z,H)=\ip{\de_x,(H-z)^{-1}\de_x}$ in the  limit $\Im z\downarrow 0$ with $H$ as the operator $T$ or $T+\lm v$.
Let $\Sigma\subset\R$ be the largest open set, where the maps
\begin{align*}
\h\to\h,\quad z\mapsto \Gm_{x}(z,T),
\end{align*}
have unique continuous extensions to $\h\cup\Sigma \to\h$ for all $x\in\V$. In particular,  this entails that for $E\in\Sigma$  and $x\in \V$ the limits
\begin{align*}
 \Gm_{x}(E,T):=\lim_{\eta\downarrow0}\Gm_{x}(E+i\eta,T)
\end{align*}
exist, are continuous in $E$ and
\begin{align*}
\Im \Gm_{x}(E,T)>0.
\end{align*}

We will prove the following  theorem for $T$, which implies Theorem~\ref{main1}.
\medskip

\begin{thm}\label{t:LIG} Let $T$ be a label invariant operator. Then, the set $\Sigma$ consists of finitely many open intervals and
\begin{align*}
    \si(T)=\clos \Sigma.
\end{align*}
Moreover, for every $x\in\V$ the Green function
\begin{align*}
\h\to\h,\quad z\mapsto G_{x}(z,T),
\end{align*}
has a unique continuous extension to $\h\cup\Sigma \to\h$ and is uniformly bounded in $z$.
\end{thm}

The proof relies heavily  on the analysis of the recursion relation for the truncated Green function $\Gm_x(z)=\Gm_{x}(z,T)$.   We look at these recursion relations from the three view points discussed in Section~\ref{s:Recursion}. These view points  are given by the recursion formula   \eqref{e:Gm} itself, the system of polynomial equations \eqref{e:Polynom} and the recursion maps \eqref{e:Psi}. Every viewpoint contributes to the proof which is  given in Section~\ref{s:LIac}.

Afterwards, we use similar techniques to prove Theorem~\ref{main2} also in Section~\ref{s:LIac}.
Actually, a similar statement as Theorem~\ref{t:LIG} can be proven for the operators $T+\lm v$ with $v\in\W_{\mathrm{sym}}(\T)$ on compact subsets of $\Sigma$ and $\lm$ sufficiently small. This is done in Theorem~\ref{t:LabRadPotG} and a proof is sketched at the end of Section~\ref{s:LIac}.


\section{Recursion formulas}

As already discussed in Subsection~\ref{ss:application}, the infinitely many recursion formulas \eqref{e:Gm} reduce in the present setting to finitely many  ones. These  are given by
\begin{align}\label{e:LIGm}
-\frac{1}{\Gm_{j}(z,T)}={z-m_{j}+\sum_{k\in\A}m_{j,k}\Gm_{k}(z,T)}, \quad j\in\A.
\end{align}
We denote for the rest of this chapter
\begin{align*}
\Gm(z):=(\Gm_{j}(z,T))_{j\in\A}, \quad z\in\h.
\end{align*}

We  derive  some immediate bounds from the recursion formulas.
\medskip

\begin{lemma}\label{l:Gmbounds}(Uniform bounds for the truncated Green functions.)
For all $z\in\h$ and $j\in \A$
\begin{align*}
    \frac{1}{|z|+m_{j}+\sum_{k\in\A}m_{j,k}/\sqrt{m_{j,j}}}\leq |\Gm_{j}(z)|\leq \frac{1}{\sqrt{m_{j,j}}}.
\end{align*}
\begin{proof}
By taking imaginary parts in \eqref{e:LIGm} we get
\begin{align*}
\Im \Gm_{j}(z)=\ap{\Im z+\sum_{k\in\A}m_{j,k}\Im\Gm_{k}(z)}    |\Gm_{j}(z)|^{2}.
\end{align*}
Since the $\Gm_{k}(z)$ are Herglotz functions, see Lemma~\ref{l:Herglotz}, we have $\Im \Gm_{k}(z)>0$ for all $k\in\A$. Dropping the positive terms $\Im z$ and $m_{j,k}\Im \Gm_{k}(z)$ for $k\neq j$ on the right hand side of the equation we obtain
\begin{align*}
\Im \Gm_{j}(z)\geq m_{j,j}\Im\Gm_{j}(z)|\Gm_{j}(z)|^{2}.
\end{align*}
We also know that $m_{j,j}>0$ by the axioms (M1) and (T0). Hence, we can divide by $\Im \Gm_{j}(z)$ to obtain the upper bound. Given the upper bound the lower bound follows immediately by taking the modulus in \eqref{e:LIGm}.
\end{proof}
\end{lemma}

Note that this lemma already yields that the spectral measure $\mu_o$ of $T$ with respect to the characteristic function of the root $o$ is purely absolutely continuous. By the extension from $\Gm$ to $G$, Proposition~\ref{p:G2}, we can derive that the  spectrum of $T$ is purely absolutely continuous. However,  we need the continuity of the Green function stated in Theorem~\ref{t:LIG} in order to study perturbations of the operator.

There is a nice alternative argument to exclude the existence of eigenvalues.  Assume there exists an eigenfunction $\ph$ to $T$. As $\mu_o$ is purely absolutely continuous,  $\ph$  must vanish at the root $o$. This, however implies that we can decouple the tree $\T(M,a(o))$ at the root into a union of trees, viz $M(j,k)$ trees with root $o$ of label $k\in\A$ and $\ph$ will induce an eigenfunction on each of these trees as well. By the same reason as above, these induced eigenfunctions must vanish at the roots of the new trees. By induction we see that $\ph\equiv0$.

\section{Polynomial equations}\label{s:LIPolynom}

Similar to the recursion formulas,  the system of infinitely many polynomial equations \eqref{e:Polynom} can be reduced in the present setting to a finite system $P(z,\xi)=0$ given by
\begin{align}\label{e:LIPolynom}
P_{j}(z,\xi)=\ap{z-m_{j}+\sum_{k\in\A}m_{j,k}\xi_{k}}\xi_{j}+1=0, \quad j\in\A,
\end{align}
for $z\in\R\cup\h$ and $\xi\in\h^{\A}$. We have $P(z,\Gm(z))=0$ for the truncated Green function $\Gm(z)$. For this section we set $N=|\A|$ and $\A=\{1,\ldots,N\}$.

The analysis of regular tree operators is very similar to what is done for the Laplace operator on regular trees. However the analysis of non regular tree operators is much more complicated.  To this end, the following sets will be important.
We let $\Sigma_1$ be the subset of $\R$ where \eqref{e:LIPolynom} has a solution in $\h^{\A}$, i.e.,
\begin{align*}
\Sigma_1:=\{E\in\R\mid P(E,\xi)=0\mbox{ for some } \xi\in\h^{\A} \}.
\end{align*}
We let $\Sigma_0$ be the subset of $\Sigma_1$ where there is a solution in $\h^{\A}$ such that all components are linear multiples of each other, i.e.,
\begin{align*}
\Sigma_0:=\{E\in\Sigma_1\mid P(E,\xi)=0 \mbox{ and } \arg\xi_j\ov{\xi_k}=0 \mbox{ for some }\xi\in\h^{\A} \mbox{ and all } j,k\in\A \}.
\end{align*}
For non regular tree operators  the sets $\Sigma_1$, $\Sigma_0$ stand in a close relation to the set $\Sigma$ defined prior to Theorem~\ref{t:LIG}. (This will be shown much later in Lemma~\ref{l:Sigma}.)

In this section we will show the following:
\begin{itemize}
 \item[(1.)] For regular tree operators the spectrum is an interval and we obtain an explicit formula for the Green function.
\item[(2.)] For non regular tree operators $\Sigma_0$ is finite.
\item[(3.)] The set $\Sigma_1$ consists of finitely many  intervals.
\end{itemize}

We now turn to the proof of (1.), (2.) and (3.). Each of the proofs is given in a separate subsection.


\subsection{Regular tree operators}

We first take a closer look at regular tree operators.
Recall that a label invariant operator $T$ is called a regular tree operator if (R1) and (R2) are fulfilled, i.e., if there are $k\in(0,\infty)$, $w\in\R$ with
\begin{align*}
k=\sum_{l\in\A}m_{j,l} \qand w=m_{j}\quad\mbox { for all $j\in\A$.}
\end{align*}
These are only assumptions on the operator and not on the underlying geometry of the tree, see Example~\ref{ex:regTree}~(3.). However,  we will see that the truncated Green functions of these operators behave exactly as for operators on regular trees.

Let a regular tree operator $T$ and $z\in\h$ be given. One immediately  sees that $\xi=(\xi_1,\ldots,\xi_N)=(\zeta_0,\ldots,\zeta_0)$
solves  \eqref{e:LIPolynom}, i.e., $P(z,\xi)=0$, where
$\zeta_0\in\C$ is the unique root in $\h$ of the quadratic polynomial equation
\begin{align*}
    k\zeta^{2}+(z-w)\zeta+1=0.
\end{align*}
Corollary~\ref{c:uniqueness} tells us that for $z\in\h$  the system~\eqref{e:LIPolynom} has the truncated Green function as a unique solution  in $\h^{\A}$. Hence, $\zeta_0=\Gm_x(z,T)$ for all $x\in\V$. By the square root $\sqrt z$ of a complex number $z=re^{i\ph}$ with $r\geq0$ and $\ph\in[0,2\pi)$  we understand $\sqrt z=\sqrt re^{i\ph/2}$.  We get the following proposition as an immediate consequence, which   proves the statements of Theorem~\ref{t:LIG} for regular tree operators completely.
\medskip

\begin{prop}\label{p:RegTreeOp}
Let $T$ be a regular tree operator with $k\in(0,\infty)$ and $w\in\R$ given as above.
Then, for all $E\in\R$ and $x\in \V$, the limit $\Gm_x(E,T):=\lim_{\eta\downarrow 0}\Gm_x(E+i\eta,T)$ exists, is continuous in $E$ and given by
\begin{align*}
\Gm_x(E,T)&=   \frac{1}{2k}\ap{-E+w+\sqrt{{(E-w)^2-4k}}}.
\end{align*}
The spectrum of $T$ is purely absolutely continuous and satisfies
\begin{align*}
\si(T)=\left[-2\sqrt k+w,2\sqrt k+w\right].
\end{align*}
\begin{proof}The formula for the Green function follows directly from the consideration above by computing the value of $\zeta_0$ and taking the limit as $\eta\downarrow0$. Let $I:=(-2\sqrt k+w,2\sqrt k+w)$.
Clearly, $\Gm_x(E,T)\in\h$ if and only if $E\in I$. By the extension from $\Gm$ to $G$, Proposition~\ref{p:G2}, the function $G_x(\cdot,T)$ is uniformly bounded. Moreover, it maps $\h$ into $\h$ and $\Im G_{x}(E,T)$ vanishes for energies $E\in\R\setminus I$.
The equality concerning the spectrum follows from the vague convergence of the spectral measures, Lemma~\ref{l:mu}.
\end{proof}
\end{prop}

\subsection{Non regular tree operators}

Let us turn to the case of non regular tree operators. We consider solutions of \eqref{e:LIPolynom} for $z=E\in\R$.
We will show that, except for a finite set of energies $E$, the components of the solutions to \eqref{e:LIPolynom} in $\h^{\A}$ cannot be all linear multiples of each other, i.e., they are not in $\Sigma_0$. In other words, considering the components as vectors in the upper half plane, there is always a pair which has a non vanishing angle. These non vanishing angles will be the major ingredient to show contraction of the recursion map in the next section.
\medskip

\begin{lemma}\label{l:Sigma0}
Let $T$ be a non regular tree operator. Then, $\Sigma_0$ is finite and $|\Sigma_0|\leq|\A|-1$. Moreover, if $m_{1}=\ldots=m_{N}$ then $\Sigma_0\subseteq\{m_1\}$.
\begin{proof}
Let $E\in\Sigma_1$. Then there is $\xi\in\h^{\A}$ such that $P(E,\xi)=0$. We assume now that $\arg \xi_j\ov{\xi_k}=0$ for all $j,k\in\A$ which is equivalent to assuming that there are $r_{j}>0$ such that
\begin{align*}
    \xi_{j}=r_{j}\xi_1, \qquad j\in\A.
\end{align*}
In this case $E\in \Sigma_0$ by definition of $\Sigma_0$. With this assumption the polynomial equation \eqref{e:LIPolynom} becomes
\begin{align*}
0=\sum_{k\in\A}m_{j,k}r_{j}r_{k}\xi_1^{2}+(E-m_{j})r_j\xi_{1}+1, \quad j=1,\ldots,N.
\end{align*}
We denote the coefficients of $\xi_1$ by
\begin{align*}
    c_{2,j}=\sum_{k\in\A}m_{j,k}r_{j}r_{k}\qand c_{1,j}=(E-m_{j})r_j.
\end{align*}
Quadratic polynomials with real coefficients have either two complex conjugated roots or two real ones. Since we assumed $\xi\in\h^{\A}$ we already know one root is $\xi_1\in\h$ so the other one must be $\ov{\xi_{1}}$. The polynomials all have the same roots and therefore  must have the same coefficients, because they are normalized. Thus, $c_{1,j}$ and $c_{2,j}$ are independent of $j$.

We now  check that this can only happen for $(N-1)$ energies $E$.  We start with the linear coefficients $c_{1,j}$, $j\in\A$, and distinguish two cases.

\emph{Case 1: $c_{1,j}=0$ for some (all) $j\in\A$.} Since $r_j>0$ we get $E-m_{j}=0$. Therefore,  the $c_{1,k}$, $k\in\A$ can only be equal if $E=m_{1}=\ldots=m_{N}$. Hence, there is only one exceptional energy in this case.

\emph{Case 2: $c_{1,j}\neq0$ for all $j\in\A$.}
As $r_1=1$ we obtain
\begin{align*}
    r_{j}=\frac{E-m_1}{E-m_{j}}
\end{align*}
and conclude for the quadratic term $c_{2,j}$, $j\in\A$,
\begin{align*}
c_{2,j} =\sum_{l\in\A}m_{j,l}\frac{(E-m_1)^2}{(E-m_{j})(E-m_{l})} =\frac{(E-m_1)^2}{\prod\limits_{i\in\A}(E-m_{i})} \sum_{l\in\A}m_{j,l}\prod\limits_{n\neq j,l}(E-m_{n}).
\end{align*}
Dividing the equations $c_{2,j}-c_{2,k}=0$,  $j,k\in\A$, by ${(E-m_1)^2}/{\prod
_{i}(E-m_{i})}$ leads to
\begin{align*}
0&=\sum_{l\in\A}\ap{m_{j,l} \prod\limits_{n\neq j,l}(E-m_{n})-m_{k,l}\prod\limits_{n\neq k,l}(E-m_{n})}\\
&=\ap{\sum_{l\in\A}(m_{j,l}-m_{k,l})}E^{N-2} +\sum_{l\in\A}\ap{m_{j,l}\sum_{n\neq j,l}m_{n}-m_{k,l}\sum_{n\neq k,l}m_{n}}E^{N-3}+\ph_{j,k}(E),
\end{align*}
where $\ph_{j,k}$ is a polynomial in $E$ of order lower or equal to $(N-4)$.

Now, we bring the assumption that $T$ is a non regular tree operator into play. 

Suppose first that (R1) fails, i.e.,  $\sum_{l}m_{j,l}\neq\sum_{l}m_{k,l}$ for some $j,k\in\A$. This implies that the leading coefficient of the right hand side is not equal to zero. Hence, there can be at most $(N-2)$ values $E$ which satisfy the equation above.

Suppose, on the other hand, (R1) holds but (R2) fails, i.e.,  $m_{j}\neq m_{k}$ for some $j,k\in\A$. Then, the leading coefficient in the equation above vanishes and we calculate for the coefficient of $(N-3)$-th order using (R1) twice
\begin{align*}
\sum_{l\in\A}\ap{m_{j,l}\sum_{n\neq j,l}m_{n}-m_{k,l}\sum_{n\neq k,l}m_{n}}=\sum_{l\in\A}\ap{m_{j,l}m_{k}-m_{k,l}m_{j}} =\ap{m_{k}-m_{j}}\sum_{l\in\A}m_{j,l}.
\end{align*}
Note that $\sum_{l}m_{j,l}\geq m_{j,j}>0$ (by the axioms (M1) and (T0)). Hence, the coefficient of $(N-3)$-th order does not vanish and thus there can be at most $(N-3)$ exceptional energies in this case.

In summary, we have at most one exceptional energy in Case~1 and at most $(N-2)$ or $(N-3)$ in Case~2 which gives in total at most $(N-1)$ energies.

To show the last statement we take a closer look at the case  $m_1=\ldots=m_{N}$, which is (R2). Then, the linear coefficients $c_{1,j}$ are all equal if either $E=m_{1}$ or $r_{1}=\ldots=r_{N}$. The latter case could only happen if also (R1) holds as the equations $c_{2,j}=c_{2,k}$ then imply  $\sum_{l}m_{j,l}=\sum_{l}m_{k,l}$ for all $j,k\in\A$. Since $T$ is not a regular tree operator $E=m_1$ is the only possible exceptional energy.
\end{proof}
\end{lemma}

Later we will study the recursion maps and prove certain statements for energies in $\Sigma_1\setminus\Sigma_0$. The previous lemma tells us that for non regular tree operators we only exclude a finite set.

\subsection{An application of a theorem of Milnor}
The following statement about the connected components of $\Sigma_1$ is a consequence of a theorem of Milnor~\cite{Mil}.
\medskip

\begin{lemma}\label{l:milnor}The set $\Sigma_1$ consists of finitely many intervals.
\begin{proof}
For $n\in\N$ let
\begin{align*}
    Z_n:=\{(E,u,v)\in\R\times\R^{\A}\times\R^{\A}\mid P_j(E,u+iv)=0\mbox{ and $v_{j}\geq 1/n$ for all $j\in\A$}\}.
\end{align*}
By a theorem of Milnor~\cite{Mil} the sum of the $\ell^{2}$-Betti numbers of the algebraic varieties $Z_n$ is uniformly bounded by a number $d\in\N$ depending only on the number of polynomial (in)equalities and their degrees.  The zero-th $\ell^{2}$-Betti number $\be_0(Z_n)$ is the number of connected components of $Z_n$. Since $\ell^{2}$-Betti numbers are dimensions, i.e.,  non negative integers, we have by $d$ an upper bound for the number of  connected components of  $Z_n$.
We claim that
\begin{align*}
\be_0(Z)\leq d\quad\mbox{where} \quad Z=\bigcup_{n\in\N}Z_n.
\end{align*}
Since $Z_{n}\subseteq Z_{n+1}$ we have that if $x,y\in Z$ are not in the same component of $Z$, then they cannot be in the same component for any  $Z_n$. Let $C_1,C_2,\ldots$ be the connected components of $Z$ and $x_1\in C_1,x_2\in C_2,\ldots$ be arbitrary points. For each $m\in\N$ there is an $n\in\N$ such that $x_1,\ldots,x_m\in Z_n$. This implies $x_1,\ldots,x_{m}$ are all in different connected components of $Z_n$. Since the number of connected components of $Z_n$ are bounded by $d$, the set $Z$ can have at most $d$ connected components. Since
\begin{align*}
\Sigma_1=\mathrm{Pr}_1 Z,\quad\mbox{ with }
\mathrm{Pr}_1:\R\times\R^{\A}\times\R^{\A}\to\R,\quad(E,u,v)\mapsto E,
\end{align*}
and $\mathrm{Pr}_1$ is continuous, $d$  bounds the number of connected components of $\Sigma_1$.
\end{proof}
\end{lemma}

\section{Recursion maps}\label{s:LI_RecursionMaps}

We next study the contraction properties of the recursion map $\Psi_{z}^{(H)}$ given by \eqref{e:Psi}. By the label invariance we can reduce $\Psi_{z}^{(H)}$, $z\in\h\cup\R$, to a map $\Phi_{z}:=\Phi_{z}^{(T)}:\h^{\A}\to\h^{\A}$ whose components are given by
\begin{align*}
\Phi_{z,j}(g) =-\frac{1}{z-m_{j}+\sum_{k\in\A}m_{j,k}g_{k}},\quad j\in\A.
\end{align*}
By the recursion relation \eqref{e:LIGm} we have that the truncated Green function satisfies $\Gm(z)=\Phi_z(\Gm(z))$ for all $z\in\h$.
Note that $h\in\h^{\A}$ is a fixed point of $\Phi_{z}$ if and only if it solves the system of polynomial equations $P(z,h)=0$. Therefore,  $\Sigma_1$ is exactly the set of energies $E\in\R$ for which $\Phi_{E}$ has a fixed point in $\h^{\A}$.

Before we come to the main goal of this section, let us note that a similar estimate as  Lemma~\ref{l:Gmbounds} holds for fixed points of $\Phi_z$. However,  since $\Phi_z$ is continuous in $z$ we can even give an estimate for $z\in\h\cup\R$ instead of only for $z\in\h$.
\medskip

\begin{lemma}\label{l:FPbounds}(Uniform bounds for fixed points.)
Let $z\in\h\cup\R$ and $h\in\h^{\A}$  a fixed point of $\Phi_{z}$ be given. Then, for all $j\in \A$,
\begin{align*}
    \frac{1}{|z|+m_{j}+\sum_{k\in\A}m_{j,k}/\sqrt{m_{j,j}}}\leq |h_{j}|\leq \frac{1}{\sqrt{m_{j,j}}}.
\end{align*}
\begin{proof}
Inverting the fixed point equation $\Phi_z(h)=h$ and multiplying by $-1$ allows us to follow the arguments of  Lemma~\ref{l:Gmbounds} line by line.
\end{proof}
\end{lemma}

The main goal of this section is the following result which says that for all energies $E\in\Sigma_1\setminus\Sigma_0$ the map $\Phi_E^{n+2}$ is a uniform contraction on some hyperbolic ball. Note that by Lemma~\ref{l:Sigma0} the set $\Sigma_0$ is finite if $T$ is a non regular tree operator. Moreover, recall that the number $n=n(M)$ is the smallest integer such that $M_{j,k}^{n}\geq1$ for all $j,k\in\A$.
\medskip

\begin{thm}\label{t:Contraction}(Contraction in $(n+1)$ steps.) For arbitrary $E\in\Sigma_1\setminus\Sigma_0$ and a fixed point $h\in\h^{\A}$ of $\Phi_E$ there are $c\in[0,1)$ and $R>0$  such that for all $g\in B_{R}(h)$
\begin{align*}
\dist_{\h^{\A}}\ap{\Phi_{E}^{n+1}(g),\Phi_{E}^{n+1}(h)}\leq c\;\dist_{\h^{\A}}\ap{g,h},
\end{align*}
where $\Phi_{E}^{n+1}$ means that $\Phi_E$ is applied $(n+1)$ times.
\end{thm}
The strategy of the proof is the following: We start by taking a look at the decomposition of $\Phi_{z}$ introduced in Section~\ref{s:Contraction} in the reduced picture. Then, we study how the recursion map alters the argument of its input. We use this to prove uniform contraction of $\Phi_E$ in at least every $(n+1)$-th step for $E\in\Sigma_1\setminus\Sigma_0$.
The proof is then given in Subsection~\ref{ss:LIcontraction}.

From the theorem we can derive the following consequence.
For $E\in\R$ let
\begin{align*}
U_{r}(E)&=\{z\in\h\mid |E-\Re z|<r,\; \Im z<r\},\\
\ov U_{r}(E)&=\{z\in\h\cup\R\mid |E-\Re z|\leq r,\; \Im z\leq r\}.
\end{align*}
\medskip

\begin{thm}\label{t:continuity}(Continuity and uniqueness of fixed points.) Let  $E\in\Sigma_1\setminus\Sigma_0$ and $h\in \h^{\A}$ be a  fixed point of $\Phi_E$. Then there is $r>0$ such that for all $z\in \ov U_r(E)$ the map $\Phi_{z}$ has a unique fixed point $h(z)$ which is continuous in $z$.
In particular,  the set $\Sigma_1\setminus\Sigma_0$ is open in $\R$.
\end{thm}
We will prove this theorem also in  Subsection~\ref{ss:LIcontraction} right after the proof of Theorem~\ref{t:Contraction} .


\subsection{The decomposition revisited} \label{ss:decompositionrevisited}
The reduced recursion map $\Phi_{z}:\h^{\A}\to\h^{\A}$  can be decomposed into $\Phi_{z}=\rho\circ\si_{z}\circ\tau$ with
\begin{align*}
    \rho_{j}(g)=-\frac{1}{g_j},\qquad \si_{z,j}(g)=z-m_{j}+g_{j}\qqand \tau_j(g)=\sum_{k\in\A}m_{j,k}g_{k},
\end{align*}
for $j\in\A$. We want to give an analogue of Lemma~\ref{l:Psicontr} in the reduced setting. To this end, we define the reduced versions of the quantities introduced in Subsection~\ref{s:Contraction}
\begin{align*}
p_{j,k}&:=p_{j,k}(h):=\frac{m_{j,k}\Im
h_{k}}{\sum_{i\in\A}m_{j,i}\Im h_{i}},\\
c_{j,k}&:=c_{j,k}(g,h):=\sum_{l\in\A}\frac{m_{j,l}\Im g_{l}}{\sum_{i\in\A}m_{j,i}\Im g_{i}}Q_{k,l}(g,h)\cos\al_{k,l}(g,h),\\
Q_{j,k}&:=Q_{j,k}(g,h):=\frac{\ap{\Im g_{j}\Im g_{k} \Im h_{k}\Im h_{j}\gm(g_{j},h_{j})\gm(g_{k},h_{k}) }^{\frac{1}{2}}}{\frac{1}{2}\ap{\Im g_{j}\Im h_{k}\gm(g_{k},h_{k}) +\Im g_{k}\Im h_{j}\gm(g_{j},h_{j})}},\\
\al_{j,k}&:=\al_{j,k}(g,h):=\arg\ap{\ap{g_{j}-h_{j}}\ov{\ap{g_{k}-h_{k}}}},
\end{align*}
for $j,k\in\A$. As in Subsection~\ref{s:Contraction} we
assume  $g_j\neq h_{j}$ and $g_k\neq h_{k}$ in the definition of $\al_{j,k}$ and set $Q_{j,k}=0$ in the complementary case.
Recall that $\sum_{k}p_{j,k}=1$ for all $j\in \A$. Furthermore, $c_{j,k}\leq1$ since $\cos\al_{j,k}\leq1$ and $Q_{j,k}\leq 1$ as a quotient of a geometric and an arithmetic mean. Moreover, let us recall that the argument $\arg:\C\setminus\{0\}\to\Sp^{1}$ is a continuous group homomorphism and $\ma{\cdot}=d_{\Sp^{1}}(\cdot,1)$ is a modulus function in $\Sp^{1}$.

Using this notation, we extract from  Lemma~\ref{l:Psicontr} the following statements:\medskip

\begin{cor}\label{c:Psicontr} Let $z\in\h\cup\R$.
\begin{itemize}
  \item [(1.)] $\rho:(\h^{\A},\gm_{\A})\to(\h^{\A},\gm_{\A})$ is an isometry.
  \item [(2.)] $\si_{z}:(\h^{\A},\gm_{\A})\to(\h^{\A},\gm_{\A})$ is an isometry for $\Im z=0$ and a contraction for $\Im z>0$.
  \item [(3.)] $\tau:(\h^{\A},\gm_{\A})\to(\h^{\A},\gm_{\A})$ is a quasi contraction with
      \begin{align*}
        \gm(\tau_j(g),\tau_j(h)) = \sum_{k\in\A}p_{j,k}(h)c_{j,k}(g,h) {\gm(g_{k},h_{k})}, \quad j\in\A.
      \end{align*}
\end{itemize}
\end{cor}
Obviously, in the limit $\Im z\downarrow0$ contraction can only come from $\tau$. We next give a sufficient criterion for uniform contraction coming from $\tau$. Note that the assumptions of the next lemma can never be  satisfied for balls. However, the solution is a suitable decomposition of balls, see the Proof of Theorem~\ref{t:Contraction} in Subsection~\ref{ss:LIcontraction}.
\medskip

\begin{lemma}\label{l:tau} (Sufficient criterion for uniform contraction.) Let  $K\subset \h^{\A}$ be compact. Suppose there is $h\in K$ such that
\begin{align*}
    \min_{g\in K'}\max_{j,k\in\A} \ma{\al_{j,k}(g,h)}>0,
\end{align*}
where $K':=\{g\in K\mid g_{j}\neq h_{j}\mbox{ for all }j\in\A\}$.
Then there is $\de>0$ such that for all $g\in K$ and $z\in\h\cup\R$
\begin{align*}
\gm_{\A}\ap{\Phi_z^{n}(g),\Phi_z^{n}(h)}\leq (1-\de)\gm_{\A}(g,h).
\end{align*}
\begin{proof}
We start with a claim.

\emph{Claim~1: For $g\in K\setminus K'$ there are $k,l\in\A$ such that $m_{k,l}>0$ and $Q_{k,l}(g,h)=0$.}
Proof of Claim~1. If $g\in K\setminus K'$ there is $k\in\A$ such that $g_{k}=h_{k}$ which readily gives $Q_{k,l}=0$ for all $l\in\A$ by definition.

\emph{Claim~2: There is $\de'>0$ such that for every $g\in K'$ there are $k,l\in\A$ such that $m_{k,l}>0$ and $\ma{\al_{k,l}(g,h)}\geq\de'$.}\\
Proof of Claim~2. By assumption, there is $\eps>0$ such that for all $g\in K'$ and suitable $i,j\in\A$ (depending on $g$) we have $\ma{\al_{i,j}(g,h)}\geq\eps$.
By the primitivity assumption (M2) there are $l(1),\ldots, l({n+1})\in\A$ with ${l(1)}=i$, ${l(n+1)}=j$ and $m_{l({s}),l({s+1})}>0$ for all $s=1,\ldots, n$. We calculate, using the definition of $\al_{j,k}$ and the triangle inequality,
\begin{align*}
\eps\leq \ma{\al_{i,j}(g,h)}= \na{\sum_{s=1}^{n}\al_{l(s),l({s+1})}(g,h)} \leq \sum_{s=1}^{n}\na{\al_{l(s),l({s+1})}(g,h)}
\end{align*}
and infer the claim by letting $\de'=\eps/n$.

\emph{Claim~3: There is $\de''>0$ such that for every $g\in K$ there exists $k\in\A$ such that}
\begin{align*}
\gm(\tau_k(g),\tau_k(h))
&\leq (1-\de'')\gm_{\A}(g,h).
\end{align*}
Proof of Claim~3. Let $\de'>0$ be taken from Claim~2 and set
\begin{align*}
\de_0&:=1-\cos {\de'},\quad
\de_{1}:=\min_{g\in K}\min_{j,k\in\A}\frac{\Im g_{k}}{\sum_{l\in\A}m_{j,l}\Im g_{l}}, \quad
\de_{2}:=\min_{g\in K}\min_{j\in\A}\frac{\Im h_{j}}{\sum_{k\in\A}m_{j,k}\Im h_{k}}.
\end{align*}
Notice that $\de_1,\de_2>0$ since $K$ is compact. For given $g\in K$, we let $k,l\in\A$ be taken from Claim~1 if $g\in K\setminus K'$ and from Claim~2 if $g\in K'$. Hence, $Q_{k,l}\cos\al_{k,l}\leq 1-\de_0$.
We estimate $Q_{k,i}\leq1$  and $\cos\al_{k,i}\leq1$ for $i\neq l$ to obtain
\begin{align*}
c_{k,k}&\leq\sum_{i\neq l}\frac{m_{k,i}\Im g_{i}}{\sum_{s}m_{k,s}\Im g_{s}}+\frac{m_{k,l}\Im g_{l}}{\sum_{s}m_{k,s}\Im g_{s}}Q_{k,l}\cos\al_{k,l}
\\
&= 1-\frac{m_{k,l}\Im g_{l}}{\sum_{s}m_{k,s}\Im g_{s}}(1-Q_{k,l}\cos\al_{k,l})\leq 1-m_{k,l}\de_1\de_0,
\end{align*}
Invoking Corollary~\ref{c:Psicontr}~(3.) and using $\gm(g_l,h_l)\leq\gm_{\A}(g,h)$, estimating $c_{k,l}\leq1$ for $l\neq k$ and $\sum_{l}p_{k,l}=1$, we compute
\begin{align*}
\gm(\tau_k(g),\tau_k(h))&=\ap{ \sum_{l\in\A}p_{k,l}c_{k,l} \frac{\gm(g_{l},h_{l})}{\gm_{\A}(g,h)}}\gm_{\A}(g,h)\leq\ap{ \sum_{l\in\A, l\neq k}p_{k,l}+p_{k,k}c_{k,k}}\gm_{\A}(g,h)
\\&\leq
(1-p_{k,k}(1-c_{k,k}))\gm_{\A}(g,h)\leq (1-m_{k,k}m_{k,l}\de_2\de_1\de_0)\gm_{\A}(g,h).
\end{align*}
We have $m_{k,l}>0$ by the  choice of $k,l\in\A$ from Claim~2 and $m_{k,k}>0$ by the axioms (M1) and (T0). Hence, we infer Claim~3 by letting $\de''$ be the minimum of $m_{k,k}m_{k,l}\de_2\de_1\de_0$ over all $k,l\in\A$ such that $m_{k,l}>0$.

We now prove the statement of the lemma. Let $g\in K$ and $k\in\A$ be taken from Claim~3.
By the primitivity of $M$, we have for all $j\in\A$ the existence of $j({1}),\ldots, j({n})\in\A$ such that $j(1)=j$, $j(n)=k$ and
$m_{j({s}),j({s+1})}>0$ for $s=1,\ldots,n$. We compute by iteration, using that $\rho\circ\si_z$ is an isometry and employing the formula for $\tau$ in Corollary~\ref{c:Psicontr}~(3.),
\begin{align*}
\gm(\Phi_{z,j}^{n}(g),\Phi_{z,j}^{n}(h))&\leq \sum_{i(1),\ldots,i(n)\in\A,i(1)=j} \ap{\prod_{s=1}^{n}c_{i(s),i(s+1)}p_{i(s),i(s+1)}} \gm(\tau_{i(n)}(g),\tau_{i(n)}(h)).
\end{align*}
Let $J=\{(i(1),\ldots,i(n))\in\A^{n}\mid i(1)=j\}\setminus\{(j(1),\ldots,j(n))\}$.
We factor out $\gm_{\A}(g,h)$ and get, since $c_{i(s),i(s+1)}\leq 1$ and $\gm(\tau_{i(n)}(g),\tau_{i(n)}(h))\leq\gm_{\A}(g,h)$,
\begin{align*}
\ldots&\leq \ab{\sum_{(i(1),\ldots,i(n))\in J} {\prod_{s=1}^{n}p_{i(s),i(s+1)}} +\ap{\prod_{s=1}^{n}p_{j(s),j(s+1)}} \frac{\gm(\tau_{j(n)}(g),\tau_{j(n)}(h))}{\gm_{\A}(g,h)}}\gm_{\A}(g,h).
\end{align*}
As $\sum_{i(1),\ldots,i(n)} \prod_{s=1}^{n}p_{i(s),i(s+1)}=1$ and $j(n)=k$, we get
\begin{align*}
\ldots&\leq\ap{1-\ap{\prod_{s=1}^{n}p_{j(s),j(s+1)}}\ap{1- \frac{\gm(\tau_k(g),\tau_k(h))}{\gm_{\A}(g,h)}}}\gm_{\A}(g,h)\\ &\leq \ap{1-\ap{\prod_{s=1}^{n}p_{j(s),j(s+1)}}
\de''}\gm_{\A}(g,h),
\end{align*}
where we used Claim~3 in the second estimate.

By our choice of $j(1),\ldots,j(n)$ the product over the $p_{j(s),j(s+1)}$'s is positive.  We take the minimum over all such positive products to obtain the desired constant $\de>0$.
\end{proof}
\end{lemma}

From the previous lemma we can learn how to prove uniform contraction in $n$ steps. It suffices to ensure the existence of $j,k\in\A$ such that $\al_{j,k}(g,h)\not\in[-\de,\de]$ for all $g$ contained in a ball about a fixed $h\in\h^{\A}$. However,  for a ball about $h\in\h^{\A}$ this always fails as discussed in Section~\ref{s:Contraction}. Nevertheless, we will show in the next subsection that $\Phi_{z}$ maps the set of $g$ where $\al_{j,k}(g,h)$ is small to the set where $\al_{j,k}(\cdot,h)$ is large. This will allow us to show uniform contraction in every $(n+1)$ steps.


\subsection{The recursion map on the relative arguments}\label{ss:relative_arguments}
Note that $\tau:\h^{\A}\to\h^{\A}$ extends to a linear map $\C^{\A}\to\C^{\A}$. Moreover, it is easy to check that
$\Phi_{z,j}(g)\neq h_{j}$ implies $\tau_{j}(g-h)\neq 0$ for $j\in\A$.
\medskip

\begin{lemma}\label{l:taual}
Let $z\in\h\cup\R$ and $h\in\h^{\A}$ be a fixed point of $\Phi_z$. Then, for all $g\in\h^{\A}$, $j,k\in\A$ with $\Phi_{z,j}(g)\neq h_{j}$, $\Phi_{z,k}(g)\neq h_{k}$, we have
\begin{align*}
\al_{j,k}\ap{\Phi_{z}(g),\Phi_{z}(h)} =\arg\ap{\tau_j(g-h)\ov{\tau_k(g-h)}} +\arg\ap{\Phi_{z,j}(g)\ov{\Phi_{z,k}(g)}}
+\arg\ap{h_{j}\ov{h_{k}}}.
\end{align*}
\begin{proof}
We calculate directly using the decomposition $\Phi_z=\rho\circ\si_z\circ\tau$
\begin{align*}
\al_{j,k}\ap{\Phi_{z}(g),\Phi_{z}(h)} &=\arg\ap{\frac{-1}{\si_{z,j}(\tau(g))}-\frac{-1}{\si_{z,j}(\tau(h))}} \ov{\ap{\frac{-1}{\si_{z,k}(\tau(g))}-\frac{-1}{\si_{z,k}(\tau(h))}}}\\
&=\arg\ap{\frac{\tau_j(g-h)}{\si_{z,j}(\tau(g))\si_{z,j}(\tau(h))}} \ov{\ap{\frac{\tau_k(g-h)}{\si_{z,k}(\tau(g))\si_{z,k}(\tau(h))}}}\\
&=\arg\ap{\tau_j(g-h) \ov{\tau_k(g-h)}}\ap{\Phi_{z,j}(g)\ov{\Phi_{z,k}(g)}} \ap{h_{j}\ov{h_{k}}},
\end{align*}
where we used $\Phi_{z,l}(\cdot)=-1/\si_{z,l}(\tau(\cdot))$, $l\in\A$, and the assumption $\Phi_{z}(h)=h$.
\end{proof}
\end{lemma}

Let us discuss the idea of how we will use the formula of the lemma above.
If $\al_{j,k}(g,h)$ is large for some $j,k$ we can apply Lemma~\ref{l:tau} directly. Otherwise, we appeal to Lemma~\ref{l:taual} in the following way:
Suppose $\al_{j,k}(g,h)$ is small for all $j,k$. Then, $\arg(\tau_{j}(g-h)\ov{\tau(g_{k}-h_{k})})$ is small by a geometric argument. Moreover, if $g$ is very close to $h$, then the last two terms in the formula of the lemma are equal up to a small error. As we know from Lemma~\ref{l:Sigma0}, the last term is non zero except for a finite set of energies, Lemma~\ref{l:taual} then proves that $\al_{j,k}(\Phi_E(g),h)$ is large. Therefore,  we can apply  Lemma~\ref{l:tau} either for $g$ or $\Phi_{z}(g)$.

Let us make these arguments precise. We start with some geometric observations.\medskip

\begin{lemma}\label{l:taugeometry}
Let $z\in\h\cup\R$ and $h\in\h^{\A}$ be a fixed point of $\Phi_z$.
\begin{itemize}
  \item [(1.)] For all $g\in\h^{\A}$ with $g_{j}\neq h_{j}$, $\Phi_{z,j}(g)\neq h_{j}$ for all $j\in\A$,
\begin{align*}
\max_{j,k\in\A}\na{\arg \ap{\tau_{j}(g-h)\ov{\tau_k(g-h)}}}\leq\max_{j,k\in\A}\na{\arg \ap{(g_{j}-h_{j})\ov{(g_{k}-h_{k})}}}.
\end{align*}
  \item [(2.)] For all $\de>0$ there exists $R>0$ such that
\begin{align*}
\max_{g\in B_{R}(h)}\max_{j\in\A} \na{\arg( g_{j}\ov{h_{j}})}\leq\de.
\end{align*}
\end{itemize}
\begin{proof}
The first statement follows since $\tau_j$, $j\in \A$, maps the cone in $\C$ which is spanned by the vectors $\{g_{j}-h_{j}\mid{j\in\A}\}$ into itself.\\
For the second statement we choose $R$ so small that the ball $B_R(h)$ is included in the cone spanned by all vectors $g$ which satisfy  $\arg (g_{j}\ov h_{j})\in[-\de,\de]$ for all $j\in\A$.
\end{proof}
\end{lemma}

Recall that $\Sigma_1$ is the set of energies $E\in\R$ where $\Phi_{E}$ has a fixed point in $\h^{\A}$ and $\Sigma_0$ is the  subset of $\Sigma_1$ where the components of a fixed point to a given $E$ are linear multiples of each other. Note that for non regular tree operators the set  $\Sigma_0$ is finite (see Lemma~\ref{l:Sigma0}).\medskip

\begin{lemma}\label{l:al}
For all $E\in\Sigma_1\setminus\Sigma_0$ and  a fixed point $h\in\h^{\A}$ of $\Phi_{E}$ there are $R>0$ and $\de>0$ such that for all $g\in B_{R}(h)$ with  $g_{j}\neq h_{j}$ and $\Phi_{z,j}(g)\neq h_{j}$ for all $j\in\A$,
\begin{align*}
\max_{j,k\in\A}\na{\al_{j,k}(g,h)}\geq \de \quad \mbox{ or }\quad \max_{j,k\in\A}\na{\al_{j,k}(\Phi_{E}(g),\Phi_{E}(h))}\geq \de.
\end{align*}
\begin{proof}
By definition,  there are $j,k\in\A$ for each $E\in \Sigma_1\setminus\Sigma_0$ such that $\de':=\ma{\arg(h_{j}\ov{h_{k}})}>0$. As $h_{j},h_{k}\in\h$, we have $\de'\in(0,\pi)$. We fix $j$, $k$ and $\de'$ for the rest of the proof and set
\begin{align*}
    \de:=\frac{1}{2}\min\set{\de',\pi-\de'}.
\end{align*}
Let $R>0$ be chosen according to Lemma~\ref{l:taugeometry}~(2.) with respect to $\de>0$. For $g\in B_{R}(h)$ we get by the triangle inequality of $\ma{\cdot}$
\begin{align*}
\na{\arg\ap{g_{j}\ov{g_{k}}h_{j}\ov{h_{k}}}}
\geq \na{2\arg\ab{h_{j}\ov{h_{k}}}}\hspace{-.2cm} -\na{\arg\ab{g_{j}\ov{h_{j}}}}\hspace{-.2cm}-\na{\arg\ab{\ov{g_{k}}{h_{k}}}}\geq
4\de-2\de=2\de,
\end{align*}
for all $g\in B_{R}(h)$. Since $\Phi_{E}$ is a quasi contraction and $h$ is a fixed point we have $\Phi_E(B_R(h))\subseteq B_R(h)$. Therefore,  we directly have by the previous inequality
\begin{align*}
\na{\arg\ap{\Phi_{E,j}(g)\ov{\Phi_{E,k}(g)}} +\arg\ap{h_{j}\ov{h_{k}}}}\geq  2\de.
\end{align*}
Now, combining Lemma~\ref{l:taual}, Lemma~\ref{l:taugeometry}~(1.) and the inequality above yields that $\na{\al_{j,k}(\Phi_{E}(g),\Phi_{E}(h))}\geq2\de-\na{\al_{j,k}(g,h)}\geq\de$ whenever  $g\in B_{R}(h)$ is such that $\ma{\al_{l,m}(g,h)}<\de$  for all $l,m\in\A$.
\end{proof}
\end{lemma}

\subsection{Uniform contraction on balls}\label{ss:LIcontraction}

We have now all the ingredients to prove Theorem~\ref{t:Contraction}.

\begin{proof}[Proof of Theorem~\ref{t:Contraction}]
Let $E\in\Sigma_1\setminus\Sigma_0$ and $h\in\h^{\A}$ be a fixed point of $\Phi_{E}$. Let $R>0$ and $\de>0$ be taken from Lemma~\ref{l:al}. We divide the set $B_R(h)$ into two subsets via
\begin{align*}
B_{\geq}(h)&:=\set{g\in B_{R}(h)\mid g_{j}=h_{j} \mbox{ for some } j\in\A\mbox{ or } \max_{j,k\in\A}\na{\al_{j,k}(g,h)}\geq\de},\\
B_{<}(h)&:=\set{g\in B_{R}(h)\mid g_{j}\neq h_{j} \mbox{ for all } j\in\A\mbox{ and }  \max_{j,k\in\A}\na{\al_{j,k}(g,h)}< \de},
\end{align*}
We first apply Lemma~\ref{l:tau} with $K=B_{\geq}(h)$: As $\Phi_{E}$ is a quasi contraction, we obtain for $g\in B_{\geq}(h)$
\begin{align*}
\gm_{\A}\ap{\Phi_E^{n+1}(g),\Phi_E^{n+1}(h)}\leq \gm_{\A}\ap{\Phi_E^{n}(g),\Phi_E^{n}(h)} \leq (1-\de')\gm_{\A}(g,h),
\end{align*}
with some $\de'>0$ which is independent of $g$.
For $g\in B_{<}(h)$ we have by Lemma~\ref{l:al}, as $\Phi_{E}$ is a quasi contraction and $h$ is a fixed point,
\begin{align*}
    \Phi_{E}(B_{<}(h))\subseteq B_{\geq}(h).
\end{align*}
Therefore,   Lemma~\ref{l:tau} with $K= \Phi_{E}(B_{<}(h))$ applied to $\Phi_E(g)$ for $g\in B_{<}(h)$ yields
\begin{align*}
\gm_{\A}\ap{\Phi_E^{n+1}(g),\Phi_E^{n+1}(h)}= \gm_{\A}\ap{\Phi_E^{n}(\Phi_{E}(g)),\Phi_E^{n}(\Phi_{E}(h))} \leq (1-\de')\gm_{\A}(g,h),
\end{align*}
since $\Phi_{E}(g)\in B_{\geq}(h)$. By Lemma~\ref{l:compactness},
we get the existence of $c\in[0,1)$ such that
$$\dist_{\h^{\A}}\ap{\Phi_E^{n+1}(g),\Phi_E^{n+1}(h)} \leq c\, \dist_{\h^{\A}}(g,h),$$
for $g\in B_R(h)$ as $\gm_{\A}\ap{\Phi_E^{n+1}(g),\Phi_E^{n+1}(h)} \leq (1-\de')\gm_{\A}(g,h)$ for $g\in B_R(h)$.
\end{proof}

We next prove Theorem~\ref{t:continuity}.\medskip

\begin{proof}[Proof of Theorem~\ref{t:continuity}]
Let $E\in\Sigma_1\setminus\Sigma_0$, $h\in \h^{\A}$ be a fixed point of $\Phi_E$ and $R>0$ be taken from Theorem~\ref{t:Contraction}.

We want to use the stability of limit points, Lemma~\ref{l:fixpoint2} with $(X_j,d_j)=(\h^{\A},\dist_{\h^{\A}})$, $B_j(R)=B_R(h)$ and $\ph_j=\Phi^{n+1}_E$ for all $j\in\N_0$. By Theorem~\ref{t:Contraction}, the maps $\Phi^{n+1}_E$ are uniform contractions on $B_{R}(h)$. Let $\eps>0$ and let $\de(\eps)>0$ be given by Lemma~\ref{l:fixpoint2}~(1.).

Since the map $\h\cup\R \to \Lip(B_R(h),\h^{\A})$, $z\mapsto \Phi_{z}^{n+1}$ is continuous  there  is  $r>0$ such that for all $z\in\ov U_r(E)$ we have $d_{B_R(h),\h^{\A}}(\Phi_{E}^{n+1},\Phi_{z}^{n+1})\leq \de(\eps)$. The first part of Lemma~\ref{l:fixpoint2} now yields the existence of limit points of $\Phi_{z}^{n+1}$ in $B_R(h)$ for all $z\in\ov U_r(E)$.

By Theorem~\ref{t:uniqueness} we know that $\Phi_z$, and thus $\Phi_{z}^{n+1}$, have  unique fixed points in $\h^{\A}$ whenever $z\in \h$. Since these fixed points are the truncated Green functions of the operator $T$, they are analytic in $z$, (see Lemma~\ref{l:Herglotz}). Furthermore, the set $U_0:=\{\Phi_{z}^{n+1}\mid z\in U_r(E)\}$ is dense in $U:=\{\Phi_{z}^{n+1}\mid z\in \ov U_r(E)\}$  as the map $\ov U_r(E)\to \Lip(B_R(h),\h_{\A})$, $z\mapsto \Phi_{z}^{n+1}$ is continuous. Hence,  by the stability of limit points, Lemma~\ref{l:fixpoint2}~(2.), the maps $\Phi_z^{n+1}$ have unique limit points $h(z)=(h_{j}(z))$ such that $h_{j}(z)\in B_R(h)$, $j\in\N_0$, which depend continuously on $z\in \ov U_r(E)$.  Indeed, since the limit points are unique the elements $h_{j}(z)$ are equal for all $j\in\N_0$. Hence, they are fixed points of $\Phi_{z}^{n+1}$.

As Theorem~\ref{t:Contraction} applies to every fixed point, we get by Lemma~\ref{l:fixpoint2}~(3.) that the map $\Phi_{E}^{n+1}$ has a unique limit point. By the preceding considerations,  this limit point is a fixed point and has elements in $\h^{\A}$ for all $E\in\Sigma_1\setminus \Sigma_0$.

It remains to check that, if $h(z)$ is a fixed point of  $\Phi_{z}^{n+1}$, then it is also a fixed point of $\Phi_{z}$, $z\in \ov U_{r(E)}$. For $z\in\h$, this is clear  by Theorem~\ref{t:uniqueness}. For $E\in \Sigma_1\setminus\Sigma_0$ let $h(E)$ be a fixed point of $\Phi^{n+1}_E$.  As $\Phi_{E}^{n+1}$ is a uniform contraction with contraction coefficient $c\in[0,1)$ according to Theorem~\ref{t:Contraction}, we obtain
\begin{align*}
\dist_{\h^{\A}}\ap{\Phi_{E}(h(E)),h(E)} &=\dist_{\h^{\A}}\ap{\Phi_{E}^{n+2}(h(E)),\Phi_{E}^{n+1}(h(E))}\\ &\leq c\,\dist_{\h^{\A}}\ap{\Phi_{E}(h(E)),h(E)}.
\end{align*}
This is only possible if $\Phi_E(h(E))=h(E)$.
\end{proof}

\section{Spectral analysis}\label{s:LIac}

In this section we will prove Theorem~\ref{main1} and Theorem~\ref{main2}.
Recall the definition of the sets $\Sigma_1$ and $\Sigma_0$ in Section~\ref{s:LIPolynom}, i.e.,
\begin{align*}
\Sigma_1&=\{E\in\R\mid P(E,\xi)=0\mbox{ for some } \xi\in\h^{\A} \},\\
\Sigma_0&=\{E\in\Sigma_1\mid P(E,\xi)=0 \mbox{ and } \arg\xi_j\ov{\xi_k}=0 \mbox{ for some }\xi\in\h^{\A} \mbox{ and all } j,k\in\A \}.
\end{align*}

\subsection{Continuity of the Green function}

We start by proving some statements for the truncated Green function for non regular tree operators. For regular tree operators the statements are already covered by Proposition~\ref{p:RegTreeOp}.\medskip

\begin{lemma}\label{l:Gmlimits}  Let $T$ be a non regular tree operator and $E\in\R\setminus\Sigma_0$. Then, \\
(1.) the limit
\begin{align*}
\Im    \Gm(E):=\lim_{\eta\downarrow 0}\Im \Gm(E+i\eta)
\end{align*}
exists as a vector in $[0,\infty)^{\A}$ in the topology of pointwise convergence and it is continuous in $E$ on $\R\setminus \Sigma_0$,\\
(2.)  $E\in\Sigma_1$ is equivalent to $\Im \Gm_j(E)>0$ for some $j\in\A$. In this case, the limit
\begin{align*}
\Gm(E):=\lim_{\eta\downarrow 0}\Gm(E+i\eta)
\end{align*}
exists in $\h^{\A}$ (in particular, $\Im \Gm_j(E)>0$ for all $j\in\A$) and it is continuous in a neighborhood of $E$.
\begin{proof}
We prove the lemma in two steps. We first show the statement of the lemma for fixed points of $\Phi_E$ and then derive the statement for the truncated Green function.
Note that the map $\Phi_z$ has a unique fixed point whenever $z\in\h\cup(\Sigma_1\setminus\Sigma_0)$ by Theorem~\ref{t:uniqueness} and Theorem~\ref{t:continuity}.

Let $E\in\R$. By the uniform bounds of Lemma~\ref{l:FPbounds} every
sequence of fixed points $h_n\in\h^{\A}$ of $\Phi_{z_{n}}$ with  $z_n\in\h\cup\R$, $\Re z_n\to E$ and $\Im z_{n}\to 0$ has a converging subsequence. We denote the set of accumulation points of all such sequences by $A_E\subseteq(\h\cup\R)^{\A}$. By continuity of $\Phi_z$ in $z$,  we also have $\Phi_{E}(h)=h$ for all $h\in A_E$ and  $E\in\R$. We proceed with a claim.

\emph{Claim. If $h\in A_E$ for some $E\in\R$, then either $h\in \R^{\A}$ or $h\in\h^{\A}$.}\\
Proof of the claim. Assume there is $j\in\A$ such that $\Im h_j>0$. Taking reciprocals and imaginary parts in the fixed point equation $\Phi_E(h)=h$, we get $\Im h_{k}\geq m_{j,k}\Im h_{j}|h_{k}|^2$ (compare proof of Lemma~\ref{l:Gmbounds}). The lower bounds on $|h_{k}|$ of Lemma~\ref{l:FPbounds} imply $\Im h_{k}>0$ for all $k\in\A$ with $m_{j,k}>0$. By the primitivity assumption (M2), this can be iterated for all $k\in\A$. This proves the claim.

If $E\in \R\setminus\Sigma_0$ is such that there is $h\in A_E$ with $h\in\h^{\A}$, then $E\in\Sigma_1$ by the definition of $\Sigma_1$. For $E\in\Sigma_1\setminus\Sigma_0$, we know by uniqueness and continuity of fixed points, Theorem~\ref{t:continuity}, this $h$ is the unique fixed point of $\Phi_E$.
We denote $\Gm(E):=h$.  The continuity follows also by Theorem~\ref{t:continuity}.

If, on the other hand, for some $E\in \R\setminus\Sigma_0$ all elements of $A_E$ are in $\R^{\A}$, then $E\in\R\setminus\Sigma_1$. In this case, we have $\Im \Gm(E):=\lim_{z_n\to E}\Im \Gm(z_n)=0$ for all $z_n\to E$.

By the considerations above we have proven statements (1.) and (2.).
\end{proof}
\end{lemma}

Recall that $\Sigma\subset\R$ was defined as the largest open set, where the maps
\begin{align*}
\h\to\h,\quad z\mapsto \Gm_{x}(z,T),
\end{align*}
have unique continuous extensions to $\h\cup\Sigma \to\h$ for all $x\in\V$.\medskip

\begin{lemma}\label{l:Sigma} The set $\Sigma$ consists of finitely  many open intervals. Moreover, for non regular tree operators we have
\begin{align*}
\Sigma_1\setminus\Sigma_0\subseteq\Sigma\subseteq\Sigma_1.
\end{align*}
\begin{proof}
For regular tree operators the statement follows from Proposition~\ref{p:RegTreeOp}. For non regular tree operators we first check the inclusions.
The first inclusion follows from Theorem~\ref{t:continuity} and the fact that $\Gm_{x}(z,T)$ is label invariant.
The second inclusion is due to the fact that the truncated Green functions solve the polynomial equations.
We know that $\Sigma_1$ has finitely many connected components by Lemma~\ref{l:milnor} and  the set $\Sigma_0$ is finite by Lemma~\ref{l:Sigma0}. Hence, $\Sigma$ consists of finitely many intervals and it is open by definition.
\end{proof}\end{lemma}

We are now prepared to prove Theorem~\ref{t:LIG} and Theorem~\ref{main1}.

\begin{proof}[Proof of Theorem~\ref{t:LIG}]
By the previous lemma, the set $\Sigma$ consists of finitely many intervals. By Proposition~\ref{p:G} the maps $G_{x}:\h\mapsto\h$ have continuous extensions to $\h\cup\Sigma\to\h$ for $x\in\V$. By the vague convergence of the spectral measures, Lemma~\ref{l:mu}, this yields  $\si(T)\supseteq\clos (\Sigma)$.

For regular tree operators,  $\si(T)\subseteq \clos (\Sigma)$  follows directly from Proposition~\ref{p:RegTreeOp}, the extension from $\Gm$ to $G$, Proposition~\ref{p:G2}~(2.), and the vague convergence of the spectral measures, Lemma~\ref{l:mu}.

For non regular tree operators, we get by Lemma~\ref{l:Gmlimits} that $\Im\Gm_{j}(E)=0$ for  $E\in\R\setminus\Sigma_1$ and $j\in\A$. Moreover,   Lemma~\ref{l:Gmbounds} gives the uniform boundedness of $\Gm_j(E)$ in $E$. Thus, by  the extension from $\Gm$ to $G$, Proposition~\ref{p:G2}~(2.), we conclude that $\Im G_{x}(E)=0$ for all $E\in\R\setminus\Sigma_1$ and $x\in\V$. Hence, $\si(T)\subseteq\clos(\Sigma_1)$.
As $\Sigma_0$ is finite by Lemma~\ref{l:Sigma0}, the sets $\clos(\Sigma_1\setminus\Sigma_0)$ and $\clos(\Sigma_1)$ can differ by only finitely many isolated points. The only spectrum which can occur on isolated points are eigenvalues. However, they can be excluded as the uniform boundedness of $\Gm_{x}(z,T)$, guaranteed by Lemma~\ref{l:Gmbounds}, extends to $ G_{x}(z,T)$  by Proposition~\ref{p:G2}~(1.). Hence, there is no spectrum at these points and we get $\si(T)=\clos(\Sigma_1\setminus\Sigma_0)$. By the previous lemma we obtain $\si(T)=\clos(\Sigma_1\setminus\Sigma_0)\subseteq\clos(\Sigma)$ which finishes the proof.
\end{proof}

\begin{proof}[Proof of Theorem~\ref{main1}]
The statement now follows from Theorem~\ref{t:LIG} and the criterion for the absence of singular spectrum, Theorem~\ref{t:Klein}.
\end{proof}

\subsection{Stability of absolutely continuous spectrum}

By the invariance properties of $T$ and the potentials  $v\in\W_{\mathrm{sym}}(\T)$, the truncated Green function $\Gm_{x}(z,T+\lm v)$, $x\in\V$ of the operators $T+\lm v$, $\lm\geq0$ can be reduced to a vector $\Gm_{s,j}(z)=\Gm_{s,j}(z,T+\lm v)$ on the strip
$\N_0\times\A$ via
\begin{align*}
\Gm_{s}(z)=(\Gm_{s,j}(z))_{j\in\A}
\qand\Gm_{s,j}(z)=\Gm_{|x|,a(x)}(z,T+\lm v).
\end{align*}
We identify functions $v:\N_0\times \A\to[-1,1]$ with potentials $v\in\W_{\mathrm{sym}}(\T)$ via $v(x)=v_{|x|,a(x)}$ and vice versa. (See Subsection~\ref{ss:application} for discussion how to define the values of $\Gm_{s,j}$ and $v_{s,j}$ for $s\leq n(M)$.)

Moreover, the recursion map $\Psi_{z}^{(T+\lm v)}$ can be reduced to a map on $\h^{\N_0\times\A}$ into itself. For the restrictions  $\Psi^{(T+\lm v)}_{z,S^s}$ of $\Psi_{z}^{(T+\lm v)}$ to the spheres $S^{s}$, $s\in\N_0$ we define the reduced maps $\Psi_{z,\lm v,s}:\h^{\A}\to\h^{\A}$  via the components in $j\in\A$ by $\Psi_{z,\lm v,s,j}(g)=\Phi_{(z-\lm_{s,j}),j}$, i.e.,
\begin{align*}
\Psi_{z,\lm v,s,j}(g)=-\frac{1}{z-m_{j}-\lm v_{s,j}+\sum_{k\in\A}m_{j,k}g_k}.
\end{align*}
We have $\Gm_{s}(z)=\Psi_{z,\lm v,s}(\Gm_{s+1}(z))$ for all $s\in\N_0$.

The proof of the following perturbation result is a variant of the proof of Theorem~\ref{t:continuity}. However, as it only deals with existence of the limits, we only need the first part of the stability of limit points, Lemma~\ref{l:fixpoint2}.
\medskip

\begin{prop}\label{p:LabRadPotGm}
For all $E\in \Sigma_1\setminus \Sigma_0$ and all $R>0$ there exist $\lm_0>0$ and $r>0$ such that for all $z\in U_r(E)$, $\lm\in[0,\lm_0]$, $v\in\W_{\mathrm{sym}}(\T)$ and $x\in\V$
\begin{align*}
\Gm_{x}(z,T+\lm v)\in B_{R}\ap{\Gm_{a(x)}(E,T)}.
\end{align*}
\begin{proof}
Let $E\in\Sigma_1\setminus\Sigma_0$ and assume without loss of generality that $R>0$ is smaller than the $R$ from Theorem~\ref{t:Contraction}.

We want to use the stability of limit points, Lemma~\ref{l:fixpoint2},  with $(X_j,d_j)=(\h^{\A},\dist_{\h^{\A}})$, $B_j(R)=B_R(\Gm(E,T))$ and $\ph_j=\Phi^{n+1}_E$ for all $j\in\N_0$. By Theorem~\ref{t:continuity} the map $\Phi^{n+1}_E$ has a unique fixed point which is $\Gm(E,T)\in\h^{\A}$ and it is a contraction on $B_{R}\ap{\Gm(E,T)}$ according to Theorem~\ref{t:Contraction}. Let $\de=\de(R)>0$ be given by Lemma~\ref{l:fixpoint2} for $\eps=R$.

Fix $z\in\h\cup\R$, $\lm\in[0,\infty)$ and let $\Psi_{z,\lm v,s}^{(n+1)}:=\Psi_{z,\lm v,s}\circ\ldots\circ\Psi_{z,\lm v, s+n}$ for $v\in \W_{\mathrm{sym}}(\T)$ and $s\in\N_0$.
In order to control the potentials, we need to introduce some notation:
Let $K:=[-1,1]^{\{0,\ldots,n\}\times\A}$ and $\pi:\W_{\mathrm{sym}}(\T)\times \N_0\to K$ be given such that $\oh v_{k,j}:=\pi(v,s)(k,j)=v_{s+k,j}$ for $k=0,\ldots,n$, $j\in\A$. Note that for each $\oh v\in K$ we have  $\Psi_{z,\lm v,s}^{(n+1)}=\Psi_{z,\lm v',s'}^{(n+1)}$  for all $(v,s),(v',s')\in\pi^{{-1}}(\{\oh v\})$. Therefore,  we can define $\Psi_{z,\lm \oh v}^{(n+1)}:=\Psi_{z,\lm v,s}^{(n+1)}$ for $\oh v\in K$ and arbitrary $(v,s)\in\pi^{-1}(\{\oh v\})$.

The map $\h\cup\R\times [0,\infty)\to \Lip(B_R(h),\h^{\A})$, $(z,\lm)\mapsto \Psi_{z,\lm \oh v}^{(n+1)}$  is continuous for arbitrary $\oh v\in K$. Hence, there are $r>0$ and $\lm_0>0$ such that for all $z\in U_r(E)$ and $\lm\in [0,\lm_0]$ we have by the definition of $\Psi_{z,\lm \oh v}^{(n+1)}$
\begin{align*}
\sup_{s\in \N_0}\sup_{v\in \W_{\mathrm{sym}}(\T)} d_{B_R(h),\h^{\A}}\ap{\Psi_{z,\lm v,s}^{(n+1)},\Phi_{z}^{n+1}}=\max_{\oh v\in K}d_{B_R(h),\h^{\A}}\ap{\Psi_{z,\lm \oh v}^{(n+1)},\Phi_{z}^{n+1}}\leq \de(R).
\end{align*}

The stability of limit points, Lemma~\ref{l:fixpoint2}~(1.), now yields the existence of limit points of $\Psi_{z,\lm v, s}^{n+1}$ in $B_{R}(\Gm(E,T))$ for all $z\in U_r(E)$, $\lm\in[0,\lm_0]$, $v\in \W_{\mathrm{sym}}(\T)$ and $s\in\N_0$. By Theorem~\ref{t:uniqueness},  this limit point is unique and is $\Gm_{s}(z,T+\lm v)$ for $z\in \h$.
\end{proof}
\end{prop}

\begin{proof}[Proof of Theorem~\ref{main2}]
We let the finite set in the statement of the theorem be defined as $\Sigma_0'=\Sigma_0\cup(\si(T)\setminus \Sigma_1)$. The set $\Sigma_1\setminus\Sigma_0$ is a finite union of open intervals (see Lemma~\ref{l:milnor} and Theorem~\ref{t:continuity}) and therefore $\Sigma_0'$ is finite. Moreover, we have $\si(T)\setminus\Sigma_0'=\Sigma_1\setminus\Sigma_0\subseteq \Sigma$. Let $I,I'\subset\si(T)\setminus \Sigma_0'$ be compact and  $I\subset\inn I'$. As $I'$ is compact, we can pick finitely many $E_1, \ldots,E_k$ to cover $I$ by the intervals $[E_1-r(E_1),E_1+r(E_1)],\ldots,[E_k-r(E_k),E_k+r(E_k)]$, where $r_1,\ldots,r_k$ are chosen according to Proposition~\ref{p:LabRadPotGm} for $E_1,\ldots,E_k$. By Proposition~\ref{p:LabRadPotGm} there are $c,C>0$ and $\lm_0>0$ such that for $x\in\V$, $E\in I'$, $v\in \W_{\mathrm{sym}}(\T)$, $\lm\leq\lm_0$ and  sufficiently small $\eta>0$
\begin{align*}
\Im \Gm_{x}(E+i\eta,T+\lm v)>c\qand   \mo{\Gm_{x}(E+i\eta,T+\lm v)}\leq C.
\end{align*}
By the extension from $\Gm$ to $G$, Proposition~\ref{p:G2}, we get for each $x\in\V$ the existence of $c,C>0$ and $\lm_0>0$ such that
\begin{align*}
\Im G_{x}(E+i\eta,T+\lm v)>c\qand   \mo{G_{x}(E+i\eta,T+\lm v)}\leq C,
\end{align*}
for all $E\in I'$, $v\in \W_{\mathrm{sym}}(\T)$, $\lm\leq\lm_0$ and  sufficiently small $\eta>0$. Now, the criterion for the absence of singular spectrum, Theorem~\ref{t:Klein}, and the vague convergence of the spectral measures, Lemma~\ref{l:mu}, yield the result for $I$ as $I\subset\inn I'$.
\end{proof}

Indeed, we can prove by our methods a stronger statement. We will only sketch the proof, as it is not much different from what we have done so far.\medskip

\begin{thm}\label{t:LabRadPotG} Let $T$ be a non regular tree operator and $v\in\W_{\mathrm{sym}}(\T)$. Then, for every compact interval $I\subset\Sigma_1\setminus\Sigma_0$,  there exists $\lm_0>0$ such that the functions
\begin{align*}
\h\times[0,\lm_0]\to\h,&\;\quad (z,\lm)\mapsto \Gm_{x}(z,T+\lm v), \\ \h\times[0,\lm_0]\to\h,&\;\quad (z,\lm)\mapsto G_{x}(z,T+\lm v),
\end{align*}
have unique continuous extension to $(\h\cup I)\times[0,\lm_0] \to\h$, which are uniformly bounded in $z$ for every $x\in\V$.
\begin{proof}[Sketch of the proof] We choose $I\subset\Sigma_1\setminus\Sigma_0$ to be compact. Therefore,  for all $E\in I$ there are $j,k\in\A$ and $\de>0$ such that $\arg\ap{\Gm_{j}(E,T)\ov{\Gm_{k}(E,T)}}\not\in[-\de,\de]$ by the continuity of $\Gm(E,T)$ in $E\in\Sigma_1\setminus\Sigma_0$. Choose $R>0$ according to Theorem~\ref{t:Contraction}, (i.e., Lemma~\ref{l:al}).

By Proposition~\ref{p:LabRadPotGm}, the truncated Green functions $\Gm_{s,j}(z,T+\lm v)$ are in $B_R(\Gm_{j}(E,T))$ for small $\lm\ge0$ and all $s\in\N_0$. By the stability of limit points, Lemma~\ref{l:fixpoint2}, the limit point  $\Gm_{s,j}(E,T+\lm v)$ is unique in $B_R(\Gm(E,T))$. Moreover, by continuity of $\Psi_{z,\lm v,s}$ in $z$, we have $\Psi_{E,\lm v,s}(\Gm_{s+1}(E,T+\lm v))=\Gm_s(E,T+\lm v)$. Since $\Gm_{s,j}(z,T+\lm v)$ belongs to $B_R\ap{\Gm_{j}(E,T)}$, there is $\de>0$ such that
\begin{align*}
\arg\ap{\Gm_{s,j}(z,T+\lm v) \ov{\Gm_{s,k}(z,T+\lm v)}} \not\in[-\de,\de],
\end{align*}
for all $z\in\ov U_r(E)$ and $s\in\N_0$. (Indeed, the distance $R$ was chosen in Lemma~\ref{l:al} such that this holds.) Therefore,  by the same arguments as in Lemma~\ref{l:al}, there is $\de>0$ such that for $j,k\in\A$ from above
\begin{align*}
\max_{j,k\in\A}\na{\al_{j,k}(g,\Gm_{s})}\geq \de \quad \mbox{ or }\quad \max_{j,k\in\A}\na{\al_{j,k}(\Psi_{E,\lm v,s}(g),\Psi_{E,\lm v,s}(\Gm_{s+1}))} \geq \de,
\end{align*}
for $g$ close to $\Gm_s(E)=(\Gm_{s,j}(E,T+\lm v))_{j\in\A}$.
Now, following the arguments in the proofs of Lemma~\ref{l:tau} and Theorem~\ref{t:Contraction} we infer that there is $c<1$ such that for $\Psi_{E,\lm v,s}^{n+1}:=\Psi_{E,\lm v,s}\circ\ldots\circ\Psi_{E,\lm v,s+n}$
\begin{align*}
\dist_{\h^\A}\ap{\Psi_{E,\lm v,s}^{(n+1)}(g), \Gm_{s}(E)}&=\dist_{\h^\A}\ap{\Psi_{E,\lm v,s}^{(n+1)}(g), \Psi_{E,\lm v,s}^{(n+1)}(\Gm_{s+n+1}(E))}  \\
&\leq c\,\dist_{\h^\A}\ap{g,\Gm_{s+n+1}(E)},
\end{align*}
for all $g $ close to $\Gm_{s+n+1}(E)$. By similar arguments as in the proof of Theorem~\ref{t:continuity}, namely the stability of limit points, Lemma~\ref{l:fixpoint2}, we conclude that  $\Gm_{s}(E,T+\lm v)$  has a unique continuous extension in a neighborhood of $E$.
\end{proof}
\end{thm}


\section{Open problems and remarks}\label{s:problemsLI}

In this chapter we considered label invariant operators and certain deterministic perturbations. We have proven that the spectrum of  label invariant operators consists of finitely many intervals of pure absolutely continuous spectrum. Moreover, we have shown stability of the absolutely continuous spectrum under sufficiently small perturbations by radial label symmetric potentials.

The first ingredient are the uniform bounds for the truncated Green function which are obtained from the recursion formula.
Another vital ingredient is that $\Phi_{E+i\eta}$ stays a contraction for energies $E$ in certain intervals  as $\eta\downarrow0$.  This allowed us to show that the limits of the Green function have positive imaginary part and are continuous on these intervals. However,  we had to exclude a finite set $\Sigma_0\subset\R$.

This gives rise to the following question:
\medskip

\begin{question}
Are $\Gm_{x}(E,T)$ and $G_{x}(E,T)$ continuous in $E$ on $\R$?
\end{question}

Note that the Green function is analytic for energies outside the spectrum of the operator $T$. So the question is essentially about continuity of the Green functions on the set $\Sigma_0\cup(\si(T)\setminus \Sigma_1)$. 

We found that the spectrum of $T$ consists of finitely many intervals using a result of Milnor \cite{Mil} from algebraic geometry. However,  as already noted in \cite{Mil}, the estimate on the particular number of connected components is very rough. Therefore,  one might ask:\medskip

\begin{question}
How many connected components does $\Sigma_{1}$ have?
\end{question}

A closely related question is whether the distance between the intervals of $\Sigma_1$ is positive. An equivalent formulation is whether it can happen that $\Im \Gm_{x}(E,T)=0$ for some $E\in\R$  but $\Im \Gm_{x}(E',T)>0$ for $E'$ close to $E$. This leads to the questions:
\medskip

\begin{question}
How many connected components does $\si(T)$ have?
\end{question}
and\medskip

\begin{question}
Can the number of connected components of $\Sigma_1$ and $\si(T)$ differ for certain operators $T$?
\end{question}

We take a look at these questions for particular examples which we can compute numerically.\medskip

\begin{Ex} Let $\A=\{1,2\}$. We want to study the adjacency matrix, see Example~\ref{ex:operators},
\begin{align*}
    (A\ph)(x)=\sum_{y\sim x}\ph(y)
\end{align*}
on $\ell^2(\V)$ of the trees $\T(M,1)$ and $\T(M,2)$ generated by a substitution matrix $M\in\N^{2\times 2}$.

(1.) Let us first consider the substitution matrix $M_1$ given by
\begin{align*}
    M_1=\left(
        \begin{array}{cc}
          1 & 2 \\
          1 & 1 \\
        \end{array}
      \right).
\end{align*}
The corresponding polynomial equations \eqref{e:Polynom}
\begin{align*}
\xi_1^{2}+2\xi_1\xi_2+E\xi_1+1=0,\quad
\xi_1\xi_2+\xi_2^{2}+E\xi_2+1=0,
\end{align*}
can be reduced by Gröbner bases  to the system
\begin{align*}
\xi_1^4+2E\xi_1^3+(E^{2}-4)\xi_1^2-1=0,\quad
\xi_1^3+2E\xi_1^2+(E^{2}-3)\xi_1+2\xi_2+E=0.
\end{align*}
Solving these equations numerically, one finds that for each $E\in\R$ there is at most one solution $(E,\xi)\in\R\times\h^{2}$ to the equations above. Recall that $\Sigma_1$ is the set of energies where the polynomial equation \eqref{e:Polynom} has a solution $(E,\xi)$ in $\R\times\h^{2}$.
Hence, there is a function $\Sigma_1\to\h^{2}$, $E\mapsto\xi(E)=(\xi_1(E),\xi_2(E))$.
The functions $E\mapsto\Im \xi_1(E)$, $E\mapsto \Im\xi_2(E)$  are  plotted in Figure~\ref{f:example2}.
\begin{figure}[!h]
\centering
\scalebox{.6}{\includegraphics{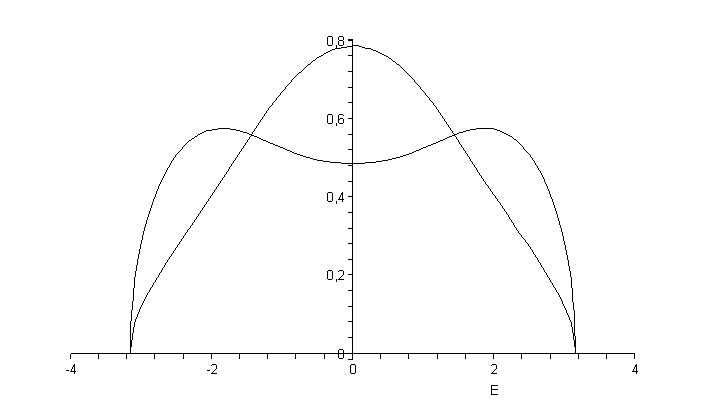}}
\caption{The imaginary parts of the truncated Green functions to $M_1$.}\label{f:example2}
\end{figure}
\\Now, one can read the set $\Sigma_1$ from Figure~\ref{f:example2} as the subset of the horizontal axis where the plotted functions do not vanish.

Recall that the spectrum of a label invariant operator does not depend on the choice of the label of the root (see Lemma~\ref{l:Gmlimits}). Moreover, by the statement and the proof of Theorem~\ref{t:LIG} we have
$\si(A)=\clos(\Sigma)=\clos(\Sigma_1\setminus\Sigma_0).$
Since $\Sigma_0$ is finite (see Lemma~\ref{l:Sigma0}) and $\Sigma_1$ does not have isolated points in our example (see Figure~\ref{f:example2}) the closure of $\Sigma_1$ is the spectrum of $A$ on $\T(M_1,1)$ and $\T(M_1,2)$. Therefore,  the spectrum consists of one single interval in this example. \\
Indeed, we can say more about $\Sigma_0$. One can calculate numerically that $\Re\xi_1(0)=\Re\xi_2(0)$ which implies $\arg \xi_1(0,A)\ov{\xi_2(0)}=0$. Therefore,  $\Sigma_0=\{0\}$, (as $\Sigma_0\subseteq \{0\}$ by Lemma~\ref{l:Sigma0}). By plotting the real parts one finds that $\xi_{1}$ and $\xi_2$  are continuous in $E=0$. This suggests that $\Gm_x(\cdot,A)$ and thus $G_x(\cdot,A)$ are continuous on $\R$.

Let us turn to another example.

(2.) Let  another
substitution matrix be given by
\begin{align*}
    M_2=\left(
        \begin{array}{cc}
          1 & 42 \\
          1 & 1 \\
        \end{array}
      \right).
\end{align*}
The polynomial equations \eqref{e:Polynom} for  the adjacency matrix $A$ on $\T(M_2,1)$ or $\T(M_2,2)$  can be reduced  with the help of Gröbner bases to
\begin{align*}
41\xi_1^4+82E\xi_1^3+(41E^{2}-1724)\xi_1^2+40E\xi_1-1&=0\\
41\xi_1^3+82E\xi_1^2+(41E^{2}-1724)\xi_1+42\xi_2+41E&=0.
\end{align*}
Again the solutions $(E,\xi)\in\R\times\h^{2}$ are unique for $E$. We  plot the imaginary parts of $E\mapsto\xi_1(E)$ and $E\mapsto\xi_2(E)$ in
 Figure~\ref{f:example}.
\begin{figure}[!h]
\centering
\scalebox{.6}{\includegraphics{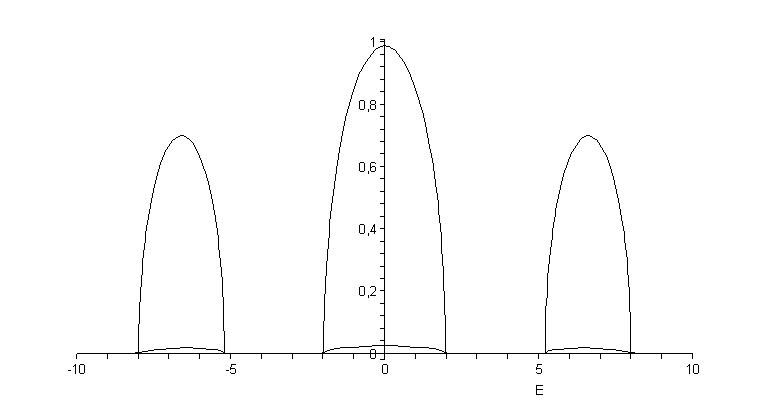}}
\caption{The imaginary parts of the truncated Green functions to $M_2$.}\label{f:example}
\end{figure}
In particular,  $\Im\xi_1\geq\Im\xi_2$ on the intervals on the left and the right and $\Im\xi_1\leq\Im\xi_2$ on the interval in the middle.\\
The spectrum consists of exactly the three intervals where $\Im \xi_1$ and $\Im \xi_2$ are positive. We infer that the spectrum of $A$ can  indeed be more than one interval.
\end{Ex}

Let us turn to another topic. The assumptions about positive diagonal (M1) and primitivity (M2) of the substitution matrix can be relaxed if one is only interested in partial results.

It would be interesting to study the theory of trees of finite cone type discussed in Example~\ref{ex:M2}. On the other hand, in \cite{Ao}, Aomoto gave a criterion for absence of eigenvalues for operators on trees which are invariant under a certain group action on the tree. So, one might ask the following question about the absolutely continuous spectrum.\medskip

\begin{question}
Consider label invariant operators on trees of finite cone type or trees invariant under a certain group action. Do they have pure absolutely continuous spectrum consisting of finitely many intervals under suitable assumptions similar to $\mathrm{(M1)}$ and $\mathrm{(M2)}$?
\end{question}

We also studied stability of the absolutely continuous spectrum under perturbations by radial label invariant potentials. There we imposed the assumption that $T$ is a non regular tree operator. This assumption is indeed necessary as the operator can be reduced to the one dimensional case otherwise.
In  \cite[Appendix A]{ASW1} it is shown that the absolutely continuous spectrum of a regular tree can be completely destroyed by arbitrary small radial symmetric  random  potentials. There,  two arguments are given. The first one uses a decomposition of the Hilbert space as it was later done in more detail by \cite{Br}. The  second argument concerns the reduction of the recursion relation and it is only sketched. We present it here  in more detail for the sake of completeness.\medskip

\begin{Ex}\label{ex:radsymPot} Let $\T=(\V,\E)$ be a $k$-regular tree, $A_\T$ the adjacency matrix on $\ell^2(\V)$, (see Example~\ref{ex:operators}~(3)), $v:\V\to[-1,1]$ given by $v(x)=f(|x|-1)$ where $f:\N\to[-1,1]$, $\lm>0$ and $H_\T=A_\T+\lm v$. For the truncated Green functions $\Gm_x=\Gm_x(z,H_\T)$ we have $\Gm_x=\Gm_y$ if $|x|=|y|$. Therefore, we can define the vector $(\Gm_n)_{n\in\N}$ by $\Gm_{|x|-1}=\Gm_x$. The recursion formula \eqref{e:Gm} thus reads
\begin{equation*}
-\frac{1}{\Gm_n(z,H_\T)}=z-\lm f(n)+k\Gm_{n+1}(z,H_\T).
\end{equation*}
Moreover, for the adjacency matrix $A_\N$ on $\ell^2(\N)$ the operator $H_\N=A_\N+\frac{\lm f}{\sqrt{k}}$ is given by
\begin{equation*}
(H_\N\ph)(n)=\sum_{j\in\N,|j-n|=1}\ph(j) +\frac{\lm{f(n)}}{\sqrt{k}}\ph(n).
\end{equation*}
The recursion formulas \eqref{e:Gm} for the truncated Green functions $\Gm_{n}(\cdot,H_{\N})$ of $H_{\N}$ and the energy $z/\sqrt{k}$  read  as
\begin{equation*}
-\frac{1}{\Gm_n(z/\sqrt{k},H_\N)}=\frac{z}{\sqrt{k}}- \frac{\lm f(n)}{\sqrt{k}}+\Gm_{n+1}(z/\sqrt{k},H_\N).
\end{equation*}
Since the recursion formula has a unique solution, (see Corollary~\ref{c:uniqueness}), for all $n\geq0$, we get
\begin{equation*}
\sqrt{k}\Gm_n(z,H_\T)=\Gm_n(z/\sqrt{k},H_\N).
\end{equation*}
This yields by the extension from $\Gm$ to $G$, Proposition~\ref{p:G2},
\begin{equation*}
\si_{\mathrm{ac}}(H_\T)=\sqrt{k}\si_{\mathrm{ac}}(H_\N).
\end{equation*}
Of course, there are many potentials  known to destroy the absolutely continuous spectrum of a one-dimensional operator completely. By Kotani theory this is the case for any non deterministic ergodic potential, see \cite{CFKS}. In fact, as shown recently in \cite[Theorem 1.1]{Re}, a potential $f:\N\to\R$ which takes only finitely many values and which is not eventually periodic yields $\si_{\mathrm{ac}}(A_\N+f)=\emptyset$. Hence, there are plenty of examples at hand for which the absolutely continuous spectrum of $H_\T$ can be completely destroyed.
\end{Ex}

The argument in this example fails whenever $T$ is a non regular tree operator. So, one might ask if there is a class of potentials which destroy the absolutely continuous spectrum of a non regular tree operator for arbitrary small $\lm>0$.
\medskip

\begin{question}
Let $T$ be a non regular tree operator. Is there a potential $v:\V\to[-1,1]$ such that $\si_{\mathrm{ac}}(T+\lm v)=\emptyset$ for all $\lm>0$?
\end{question}


\chapter{Random potentials}\label{c:main3}

\begin{quote}
\begin{flushright}
\scriptsize{A blessing on those whose robes are washed, so that they may have a right to the \emph{tree} of life, and may go in by the doors into the town.} Revelations 22:14   \end{flushright}
\end{quote}

We turn now to perturbations of label invariant operators $T$ by random potentials taken from $\W_{\mathrm{rand}}(\Om,\T)$. These potentials are independently distributed on non intersecting forward trees and identically distributed on isomorphic forward trees.
The techniques introduced in this chapter also allow for the treatment of regular trees. In any case,  as we deal with random quantities, the stability of absolutely continuous spectrum is only proven almost surely.

There are some similarities to the analysis in the previous chapter. However,  by the randomness of the potential $v\in\W_{\mathrm{rand}}(\Om,\T)$ we loose all symmetry of the operators $H^{\lm,\om}=T+\lm v^{\om}$. Nevertheless, by the assumptions on the random potential we still have some symmetry in distribution. This will be used to prove contraction and consequently stability of absolutely continuous spectrum. In contrast to the analysis for radial label symmetric potentials we do not apply a fixed point principle for iterates of $\Psi_{z}$.
Instead, we define a function $\ka$  via twofold application of $\Psi_{z}$ and interpret it as an averaged contraction coefficient. We show that $\ka\leq1-\de$ for some $\de>0$ which we refer to as uniform contraction. This leads us to an inequality which gives bounds for the mean value of the Green function.

An inequality of the same type was proven in  \cite{FHS2} for the binary tree (and later in \cite{Hal1} for regular trees). There the authors compactify the domain of the averaged contraction coefficient by blow ups. Then, they show contraction on the boundary to deduce contraction close to the boundary by a compactness argument.

Our situation is essentially more complex. Therefore,  we develop a new scheme.
Instead of one inequality (as in \cite{FHS2,Hal1}) we need one for each label. This gives a vector inequality indexed by the label set $\A$, see Proposition~\ref{p:vector}. Moreover, we do not compactify the domain of the averaged contraction coefficient.  We prove uniform contraction in the whole domain  excluding only an arbitrary small compact set, see Proposition~\ref{p:ka}. This has two additional benefits. Firstly, all estimates are explicit and would,  for instance,  allow for estimates on  the coupling parameter $\lm$. Secondly, we  prove continuity of the mean value of the Green function as $\lm\downarrow0$. 

The chapter is organized as follows. In the next section we discuss the strategy of the proof and state the theorems which  imply Theorem~\ref{main3}. In Section~\ref{s:ingredients} these theorems are proven except for the uniform contraction. This is the hard part of the analysis and it will be done in Section~\ref{s:kaproof}. At the end of this chapter, in Section~\ref{s:offdiagonal}, we discuss how our methods can be applied to small off diagonal perturbations.


\section{Mean value bounds for the Green function}\label{s:meanvalue}

We consider the family of random operators $H^{\lm,\om}$ given by $T+\lm v^{\om}$, $\om\in\Om$. Here, $T$ is a label invariant operator on a tree  $\T(M,j)=(\V,\E)$ given by a substitution matrix $M$ over a finite label set $\A$ and $j\in\A$. We assume that the matrix $M$ does  not yield the one dimensional case (M0), has positive diagonal (M1) and is primitive (M2). The random potentials $v:\Om \times\V\to[-1,1] $ in $\W_{\mathrm{rand}}(\Om,\T)$ are defined on some probability space $(\Om,\PP)$ and satisfy the independence assumption (P1) and the identical distribution assumption (P2).
More precisely, they are given by
\begin{itemize}
\item [(P1)] For all $x,y\in\V$ the random variables $v_x$ and $v_y$ are independently distributed if $ \V_x\cap\V_y=\emptyset$.
\item [(P2)] For all $x,y\in\V$ with $a(x)=a(y)$ the random variables $v\vert_{\V_{x}}$ and $v\vert_{\V_{y}}$ are identically distributed.
\end{itemize}
Recall that we call two random variables $X$ and $Y$ on two isomorphic trees $\T_X$ and $\T_Y$ {identically distributed} if for every graph isomorphism $\psi:\T_X\to\T_Y$ the random variables $X$ and $Y\circ\psi$ are identically distributed.

For a measurable function $f:\Om\to [0,\infty)$ we denote the \emph{mean value} of $f$ by
\begin{align*}
    \EE(f):=\int_{\Om} f(\om)d\PP(\om).
\end{align*}

In Theorem~\ref{t:LIG} we have shown that the spectrum $\si(T)$ of $T$ is given by the closure of the set
$\Sigma$ defined as the largest open subset of $\R$, where the maps
\begin{align*}
\h\to\h,\quad z\mapsto \Gm_{x}(z,T),
\end{align*}
have unique continuous extensions to $\h\cup\Sigma \to\h$ for all $x\in\V$. Moreover, also by Theorem~\ref{t:LIG}, the functions $\Gm_{x}(\cdot,T)$ stay uniformly bounded on $\h\cup\Sigma$.

Our goal is to prove Theorem~\ref{main3}. The following theorem is a much stronger statement and Theorem~\ref{main3} is a consequence of the second and third part. The  forth part tells us that perturbed Green functions converge in $L^p(\Om,P)$ towards the unperturbed one as the perturbation parameter $\lm$ tends to zero.
\medskip

\begin{thm}\label{t:KleinG}
Let $I\subset\Sigma$ be compact, $p>1$ and $x\in \V$. There exists $\lm_0>0$ such that for $\lm\in[0,\lm_0)$
\begin{itemize}
\item [(1.)] $
\sup\limits_{E\in I}\sup\limits_{\eta\in(0,1]} \int_\Om\mo{G_x(E+i\eta,H^{\lm,\om})}^{p}d\PP(\om)<\infty
$,
\item [(2.)]
$\liminf\limits_{\eta\to0} \int_I\mo{G_x(E+i\eta,H^{\lm,\om})}^{p}dE<\infty$ for almost every $\om\in\Om$,
\item [(3.)] $\liminf\limits_{\eta\downarrow0}
\Im G_x(E+i\eta,H^{\lm,\om})>0$ for
almost every $(\om,E)\in\Om\times I$,
\item [(4.)]
$\lim\limits_{\lm\downarrow0}\sup\limits_{\eta\in(0,1]} \int_{\Om}{\mo{G_x(E+i\eta,H^{\lm})-G_{x}(E,L)}^{p}} d\PP(\om)=0 $ for all $E\in \Sigma$
and the convergence is uniform on $I$.
\end{itemize}
\end{thm}
The previous theorem is stronger than the results about the Green functions which are obtained in \cite{ASW1} and \cite{FHS2} for regular trees. In particular,  we have uniform bounds on the mean values of the Green functions (compare to \cite{ASW1}) and continuity for $\lm\downarrow0$ (compare to \cite{FHS2}). For regular trees, there is a stronger statement  than ours  found in \cite[Theorem~1.4]{Kl1}. In particular,  there it  is proven that the mean value of the Green function is continuous for all energies in a compact interval in the interior of the spectrum and all $\lm$ sufficiently small.

We will derive Theorem~\ref{t:KleinG} from the following statement. It basically bounds the mean deviation of the perturbed Green function from the unperturbed one.
\medskip

\begin{thm}\label{t:EGm}Let $I\subset\Sigma$ be compact and $p> 1$. Then there exist $\lm_0=\lm_0(I,p)>0$ and a decreasing $c:[0,\lm_0)\to[0,\infty)$ with $c(\lm)\to0$ for $\lm\to0$ such that for $\lm\in[0,\lm_0]$
\begin{equation*}
\sup_{x\in\V}\sup_{E\in I}\sup_{\eta\in(0,1]}\int_{\Om}{\gm\ap{\Gm_x(E+i\eta, H^{\lm,\om}),\Gm_x(E+i\eta, T)}^p}d\PP(\om)\leq c(\lm).
\end{equation*}
\end{thm}

Let us sketch the idea of the proof of Theorem~\ref{t:EGm}. We first note that by (P2) the random variables $\om\mapsto \Gm_x(z, H^{\lm,\om})$ are identically distributed for all $x\in\V$ with label $a(x)$.
For $\lm\geq0$ we define the vector $\EE{\gm}:=\ap{\EE\gm_j}_{j\in\A}$ by
\begin{align*}
\ap{\EE\gm_{a(x)}}:=\EE\ap{\gm\ap{\Gm_x(z, H^{\lm}),\Gm_x(z, T)}^{p}}=\int_{\Om}\gm{\ap{\Gm_x(z, H^{\om,\lm}),\Gm_x(z, T)}}^{p}d\PP(\om).
\end{align*}
We  show that there exist a non-random stochastic matrix $P:{\A\times\A}\to[0,\infty)$ and $\de>0$ such that for $\lm$ sufficiently small we have the vector inequality
\begin{align*}
\EE{\gm}\leq (1-\de) P\,\EE{\gm}+c(\lm)\mathds{1},
\end{align*}
where $P$ depends continuously on $z$, the constant $\de$ is independent of $z\in I+i[0,1]$ and $\mathds{1}$ is the vector in $\R^{\A}$ which has ones in all components. The inequality is of course understood componentwise. By the Perron-Frobenius theorem
there is a positive normalized left eigenvector $u\in\R^{\A}$ such that $P^{\top}u=u$. We deduce
\begin{equation*}
\as{u,\EE\gm}\leq (1-\de) \as{u,\EE\gm}+c(\lm),
\end{equation*}
where $\as{\cdot,\cdot}$ denotes the standard scalar product in $\R^{\A}$. Since $P$ depends continuously on $z\in I+i[0,1]$, so does $u$ and we infer Theorem~\ref{t:EGm}.

\section{The ingredients of the proof}\label{s:ingredients}

There are two propositions which are vital for the proof of the vector inequality. The first one,  Proposition~\ref{p:expansion}, which is proven in the next subsection, is a two step expansion formula. The second one, Proposition~\ref{p:ka}, which is presented in Subsection~\ref{s:ka} and proven in Section~\ref{s:kaproof}, states that an averaged contraction coefficient is uniformly smaller than one. With these two propositions at hand the proof of the vector inequality is mainly an algebraic manipulation.

The graph in the figure below illustrates how this section and also the proof of Theorem~\ref{main3} is structured.
\begin{figure}[!ht]
\centering \scalebox{.8}{\includegraphics{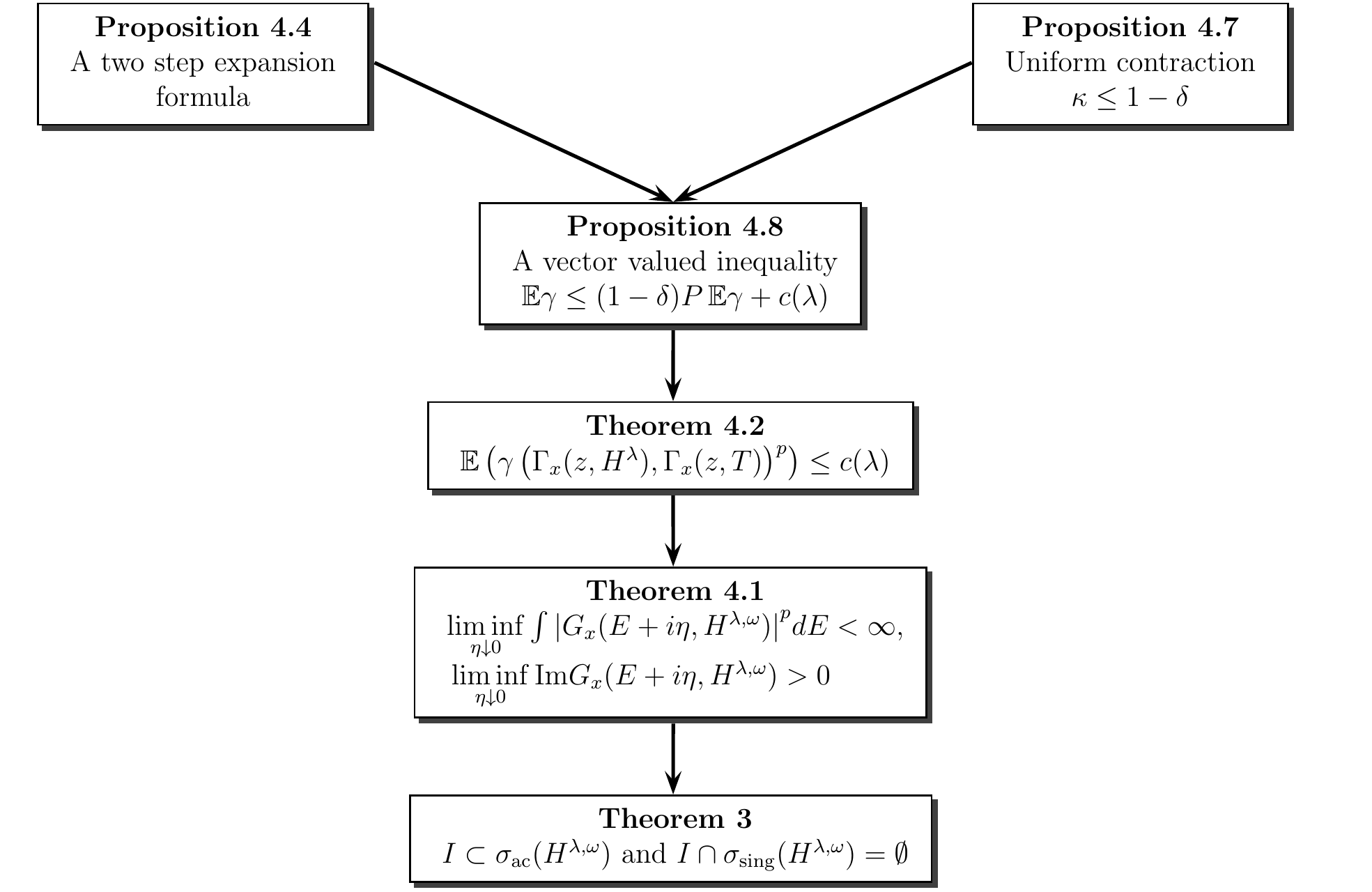}}
\caption{The logical structure of the proof of Theorem~\ref{main3}. }\label{f:structure}
\end{figure}

\subsection{A two step expansion formula}\label{s:expansion}
Before we present the two step expansion formula we introduce some notation.
Let $o\in\V$ be fixed. It can be thought of as  the root of the tree (but it does not have to be the root). Let $o'\in S_o$ with $a(o')=a(o)$ be fixed and define
\begin{align*}
S_{o,o'}&:=S_{o'}\cup S_{o}\setminus\{o'\}.
\end{align*}
In Figure~\ref{f:tree_Soo} this definition is illustrated for the tree $\T(M,1)$ of Example~\ref{ex:M}~(2.).
\begin{figure}[!ht]
\centering \scalebox{.3}{\includegraphics{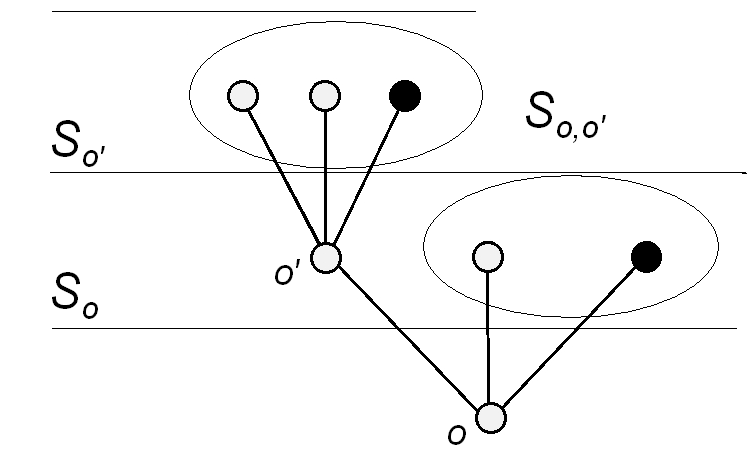}}
\caption{The upper and the lower sphere of $S_{o,o'}$.} \label{f:tree_Soo}
\end{figure}
We speak about $S_{o}\setminus\{o'\}$ as the \emph{lower sphere} and about $S_{o'}$ as the \emph{upper sphere}.
Note that $o'$ with $a(o')=a(o)$ always exists by (M1). As discussed in Section~\ref{s:LIOp}, the degree of each vertex is always larger or equal than two and therefore
\begin{align*}
S_{o,o'}\setminus S_{o'}\neq \emptyset.
\end{align*}
Moreover, since $a(o)=a(o')$ we also have
\begin{align*}
a(S_{o}\setminus \{o'\})\subseteq a(S_{o'}).
\end{align*}
This means that every label which is found in the lower sphere is also found in the upper sphere. (The opposite inclusion is true if and only if there are two forward neighbors of $o$ with label $a(o)$, i.e.,  $M_{a(o),a(o)}\geq2$.) We will often use these facts without explicitly mentioning them.

We recall the definition of the
quantities $\gm_x$, $p_{x}$, $c_{x}$, $q_x$, $Q_{x,y}$ and $\al_{x,y}$ which are originally defined in Section~\ref{s:Contraction}.
Let $z\in\h$ and $h=(\Gm_x(z,T))_{x\in S_{o}\cup S_{o'} }$,
$i\in\{o,o'\}$, $x\in S_{i}$ and $g\in\h^{ S_{o}\cup S_{o'} }$ be given.
We have
\begin{align*}
\gm_x&=\gm(g_x,h_x)
\end{align*}
and
\begin{align*}
p_{x}&=p_{x}(h)= \frac{|t(i,x)|^{2}\Im h_x}{\sum_{y\in S_i}|t(i,y)|^{2}\Im h_y}.
\end{align*}
Moreover,
\begin{align*}
c_{x}&=c_{x}(g,h)={\sum_{y\in S_i}q_{y}(g)Q_{x,y}(g,h)\cos\al_{x,y}(g,h)},
\end{align*}
where for $y\in S_{i}$
\begin{align*}
q_{y}&=q_{y}(g)=\frac{|t(i,y)|^{2}\Im g_y}{\sum_{v\in S_i}|t(i,v)|^{2}\Im g_v} \\
\end{align*}
and for arbitrary $x,y\in S_{o}\cup S_{o'}$ with $g_{x}\neq h_{x}$, $g_{y}\neq h_{y}$
\begin{align*}
Q_{x,y}&=Q_{x,y}(g,h)=\frac{\ab{\Im g_x\Im g_y\Im h_x\Im h_y\gm_x\gm_y}^\frac{1}{2}}{\frac{1}{2}\ab{\Im g_x\Im h_y\gm_y+\Im g_y\Im h_x\gm_x}},\\
\al_{x,y}&=\al_{x,y}(g,h)=\arg{(g_x-h_x)\ov{(g_y-h_y)}}.
\end{align*}
Moreover, we defined $Q_{x,y}=0$ for $g_{x}=h_{x}$ or $g_{y}= h_{y}$.
Note that, in contrast to Section~\ref{s:Contraction}, the quantities $p$ and $c$ carry only one index. This is possible as $a(o')=a(o)$. Recall that $\sum_{x\in S_i}  p_x=\sum_{y\in S_i}  q_y=1$ and $c_x\in[-1,1]$ as $Q_{x,y}\in[0,1]$ and $\cos\al_{x,y}\in[-1,1]$.

For given $W\subseteq W' \subseteq\V$ and $g\in \h^{W'}$, we denote by $g_{W}\in \h^{W}$ the restriction of $g$ to $W$. Moreover, we sometimes write a vector $g\in\h^{\{x\}\times W}$ as $g=(g_x,g_{W})$ for $x\not\in W$.

The  recursion map is defined in \eqref{e:Psi}  for $g\in \h^{S_x}$, $x\in\V$, $v:\V\to \R$ and $z\in\h$ and is given by
\begin{align*}
\Psi_{z,x}^{(T+v)}(g) =-\frac{1}{z-v(x)-m_{a(x)} +\sum\limits_{y\in\pi(S_{x})}|t(x,y)|^{2}g_{y}} =\Psi_{z-v(x),x}^{(T)}(g).
\end{align*}
For  $g\in\h^{S_{o,o'}}$ and $v,v'\in\R$ we define
\begin{align*}
g_{o'}:=g_{o'}(z,v')&:=\Psi_{z-v',o'}^{(T)}\ap{g_{S_{o'}}},\\
g_{o}:=g_{o}(z,v,v')&:= \Psi_{z-v,o}^{(T)}\ap{\ap{g_{o'}(z,v'),g_{S_{o}\setminus\{o'\}}}}.
\end{align*}

Next, we define the quantities which come up in the two step expansion formula in Proposition~\ref{p:expansion} below. For $p\geq1$, $z\in \h$, $v\in\R$ , $h=(\Gm_{x}(z,T))_{x\in S_{o,o'}}$ and $g\in\h^{S_{o,o'}}$ let
\begin{align*}
Z_0:=Z_0(z,v,g)&:=
{\sum_{y\in S_{o'}}p_{o'} p_{y}c_{o'}c_{y}\gm_{y} + \sum_{x\in S_{o}\setminus\{o'\}}p_{x}c_{x}\gm_{x}},\\
Z_1^{(p)}:=Z_1^{(p)}(z,v,g)&:=
{\sum_{y\in S_{o'}}p_{o'} p_{y}\gm_{y}^{p} + \sum_{x\in S_{o}\setminus\{o'\}}p_{x}\gm_{x}^{p}},
\end{align*}
where $g_{o'}=g_{o'}(z,v)=\Psi^{(T)}_{z-v,o'}(g_{S_{o'}})$ and $h_{o'}=h_{o'}(z,0)=\Psi^{(T)}_{z,o'}(h_{S_{o'}})$ in the definition of the quantities $c_{y}=c_{y}(g,h)$, $y\in S_{o}$.

Let us mention some simple, but important, facts about $Z_0$ and $Z_1$.\medskip

\begin{lemma}\label{l:Z} Let $p\geq1$, $z\in\h$, $v\in\R$ and $g\in\h^{S_{o,o'}}$.
Then
\begin{align*}
(Z_0(z,v,g))^{p}\leq Z_1^{(p)}(z,v,g)\leq \gm_{S_{o,o'}}(g,h)^p,
\end{align*}
where
$\gm_{S_{o,o'}}(g,h)=\max_{x\in S_{o,o'}}\gm(g_{x},h_{x})$.
Moreover, the following are equivalent
\begin{itemize}
\item [(i.)]   $Z_0(z,v,g)=Z_1^{(1)}(z,v,g)$.
\item [(ii.)]  $c_{x}=1$ for all $x\in S_{i}$ and $i\in\{o,o'\}$.
\item [(iii.)]  $Q_{x,y}=1$ and $\al_{x,y}=0$ for all $x,y\in S_{i}$ and $i\in\{o,o'\}$.
\end{itemize}
\begin{proof}
The statements follow from the Jensen inequality and the basic properties of the quantities such as the fact that the $p_{x}$'s sum up to one and the $c_{y}$'s are bounded by one.
\end{proof}
\end{lemma}

The following two step expansion is essentially a consequence of the recursion relation applied twice.\medskip

\begin{prop}\label{p:expansion}(Two step expansion.)
Let $I\subset \Sigma$ be compact. Then there exist $c,C:[0,\infty)\to[0,\infty)$ with $c(\lm), C(\lm)\to 0$ for $\lm\to0$ such that for all $z\in I+i[0,1]$, $\lm\in[0,\infty)$, $v_o, v_{o'}\in [-\lm,\lm]$, $g\in \h^{S_{o,o'}}$ and $h=\Gm_{S_{o,o'}}(z,T)$ we have
\begin{align*}
{\gm\ap{g_{o}(z,v_{o},v_{o'}),h_{o}(z,0,0)}}
&\leq(1+c(\lm))Z_{0}(z,v_{o'},g)+C(\lm),
\end{align*}
where, of course, $h_{o}(z,0,0)=\Gm_o(z,T)$.
\begin{proof}
Recall the decomposition of the recursion map $\Psi_{z-v}^{(T)}=\rho\circ\si_{z-v}\circ\tau$,  $z\in\h$, $v\in \R$, from Section~\ref{s:Contraction}. By Lemma~\ref{l:Psicontr} the maps $\rho_x:\xi\mapsto-1/\xi$ and $\si_{z-v,x}:\xi\mapsto z-v+\xi$ are quasi contractions on the semi metric space $(\h,\gm)$ into itself.

We use this to estimate after employing the definition of $g_o=g_{o}(z,v_{o},v_{o'})$, $g_{o'}=g_{o'}(z,v_{o'})$, $h_o=h_{o}(z,0,0)$, $h_{o'}=h_{o'}(z,0)$ to obtain for $i\in\{o,o'\}$
\begin{align*}
{\gm\ap{g_{i},h_{i}}}
&=\gm\ap{\Psi^{(T)}_{z-v_{i},i}\ap{g_{S_{i}}},\Psi^{(T)}_{z,i} \ap{h_{S_{i}}}}\leq \gm\ap{-v_i+\tau_i\ap{g_{S_{i}}},\tau_i \ap{h_{S_{i}}}}.
\end{align*}
To get rid of the $v_i$ we apply the substitute of the triangle inequality, the first statement of Lemma~\ref{l:ti}, and the formula for $\tau_{x}$, Lemma~\ref{l:Psicontr}~(3.),
\begin{align*}
...&\leq (1+c_0(\lm))\gm\ap{\tau_i\ap{g_{S_{i}}},\tau_i \ap{h_{S_{i}}}}+C_0(\lm)= (1+c_0(\lm))\sum_{x\in S_i}p_{x}\gm_{x}c_{x}+C_0(\lm),
\end{align*}
where we introduce $c_0(\lm):=(1+2\lm/\eps_0(I))^{2}-1$ and $C_0(\lm):=2\lm(\lm+1)/\eps_0(I)^2$ with $\eps_0(I):=\min_{z\in I+i[0,1]}\inf_{x\in \V}\Im \Gm_x(z,T)$. Note that $\eps_{0}(I)$ is positive by the definition of $\Sigma$.
Using this estimate we calculate directly
\begin{align*}
\lefteqn{\gm\ap{g_{o}(z,v_{o},v_{o'}),h_{o}(z,0,0)}}\\
&=\gm\ap{\Psi^{(T)}_{z-v_{o},o}\ap{\ap{\Psi^{(T)}_{z-v_{o'},o'}(g_{S_{o'}}), g_{S_{o}\setminus\{o'\}}}},\Psi^{(T)}_{z,o} \ap{\ap{\Psi^{(T)}_{z,o'}(h_{S_{o'}}),h_{S_{o}\setminus\{o'\}}}}}\\
&\leq(1+c_0) \sum_{x\in S_{o}}p_{x}c_{x}\gm_{x}
+C_0\\
&=(1+c_0)\ap{p_{o'} c_{o'} \gm\ap{\Psi^{(T)}_{z-v_{o'},o'}(g_{S_{o'}}),\Psi^{(T)}_{z,o'}(h_{S_{o'}})}+ \sum_{x\in S_{o}\setminus \{o'\}}p_{x}c_{x}\gm_{x}}
+C_0\\
&\leq (1+c_0)^{2} \ap{\sum_{y\in S_{o'}}p_{o'} p_{y}c_{o'}c_{y}\gm_{y} + \sum_{x\in S_{o}\setminus \{o'\}}p_{x}c_{x}\gm_{x}}+(2+c_0)C_0.
\end{align*}
\end{proof}
\end{prop}

\subsection{Permutations and an averaged contraction coefficient}\label{s:ka}

The considerations of this section make essential use of the following two facts implied by the assumptions (P1) and (P2) on the random potential. As the random variables $\om\mapsto\Gm_x(z,H^{\lm,\om})$ depend only on the values of the potential on the forward tree $\T_x$, we conclude that they
\begin{itemize}
  \item [(1.)] are independent for all $x$ with non intersecting forward trees,
  \item [(2.)] are identically distributed for all $x$ which carry the same label $a(x)$.
\end{itemize}
Let again $o\in\V$ and $o'\in S_o$ with $a(o)=a(o')$ be fixed.
We introduce permutations of $S_{o,o'}$, $o\in\V$, which leave the label invariant.
\medskip

\begin{Def}\label{d:Pi}(Label invariant permutations.) For $o\in \V$ we define the set of \emph{label invariant permutations } $\Pi:=\Pi(S_{o,o'})$ of $S_{o,o'}$ by
\begin{align*}
\Pi:=\{\pi:S_{o,o'}\to S_{o,o'}\mid \mbox{bijective and $a(\pi(x))=a(x)$ for all $x\in S_{o,o'}$ }\}.
\end{align*}
A permutation $\pi\in\Pi$ can be extended to $o'$ via $\pi(o')=o'$.
\end{Def}

For a given $g\in\h^{S_{o,o'}}$  the composition of $g$ and $\pi$ is of course given as $g\circ\pi=(g_{\pi(x)})_{x\in S_{o,o'}}$.
This gives rise to the definition of the averaged contraction coefficient using the quantities $Z_{0}$ and $Z_{1}$ defined in the previous section.
\medskip

\begin{Def}(Averaged contraction coefficient.) Let $p\geq1$, $z\in\h$, $v\in\R$ and $h=(\Gm_x(z,T))_{x\in S_{o,o'}}$.
We define the function $\ka_o^{(p)}:=\ka_{o}^{(p)}(z,v,\cdot):\h^{S_{o,o'}}\to[0,1]$, called the \emph{averaged contraction coefficient} by
\begin{align*}
\ka_o^{(p)}(g)&:=\frac{\sum_{\pi\in\Pi} Z_{0}(z,v,g\circ\pi)^p}{\sum_{\pi\in\Pi} Z_{1}^{(p)}(z,v,g\circ \pi)}.
\end{align*}
\end{Def}

By Lemma~\ref{l:Z} we have $\ka_o^{(p)}\leq 1$. In Section~\ref{s:kaproof} we  prove that
$\ka_o^{(p)}\leq 1-\de$ under suitable conditions. This is stated in the proposition below.\medskip

\begin{prop}\label{p:ka}(Uniform contraction.)
Let $I\subset \Sigma$ be compact, $p>1$ and $h=(\Gm_x(z,T))_{x\in S_{o,o'}}$ for $z\in I+i[0,1]$. There exist $\de_o=\de_o(I,p)>0$, $\lm_o=\lm_o(I)>0$ and $R_{o}:[0,\lm_o]\to[0,\infty)$ with $R_{o}(\lm)\to0$ for $\lm\to 0$ such that for all $z\in I+i[0,1]$ and $\lm\in[0,\lm_o]$
\begin{align*}
\sup_{v\in[-\lm,\lm]}\sup_{g\in \h^{S_{o,o'}}\setminus B_{R_o(\lm)}(h)}\ka_{o}^{(p)}(z,v,g)\leq 1-\de_o.
\end{align*}
\end{prop}
The proof of the proposition involves an analysis of the quantities which enter in the definition of $\ka_o^{(p)}$. We postpone it to Section~\ref{s:kaproof}. This statement is also true for $p=1$, but the proof becomes essentially harder in this case. Therefore, we do not give a proof in this context.

\subsection{A vector inequality}\label{s:vector}
For $j\in\A$ denote the tree $\T(M,j)$ by $\T_{j}=(\V_j,\E_j)$ and its by $o(j)$.
Let  $\T=(\V,\E)$ be the union of the disjoint trees $\T_j=\T(M,j)$, $j\in\A$, i.e.,
\begin{align*}
\V={\bigcup}_{j\in\A}V_j,\quad \E={\bigcup}_{j\in\A}\E_{j}.
\end{align*}
Moreover, let the labeling function $a:\V\to\A$ be the labeling function on the union of trees.
Let $\nu$ be a label invariant measure on $\V$, i.e., for $\nu(x)=\nu(y)$ for all $x,y\in\V$ with $a(x)=a(y)$. Let $T:\ell^2(\V,\nu)\to\ell^2(\V,\nu)$ be a labeling invariant operator on the union of trees $\T$, i.e., it satisfies the axioms (M0), (M1) and (M2).

For a vertex $x\in \V$ with $x\in \V_j$, $j\in\A$
we denote by $\T_x=(\V_x,\E_x)$ the forward tree of $x$ in $\T_j=(\V_j,\E_j)$. Given a probability space $(\Om,\PP)$ we denote by $\W_{\mathrm{rand}}(\Om,\T)$ the set of random potentials $v^{\om}:\V\to[-1,1]$ such that the restrictions $v\vert_{\V_j}$, $j\in\A$,  belong to $\W_{\mathrm{rand}}(\Om,\T_{j})$. This means that they fulfill the following two properties:
\begin{itemize}
\item[(P1')] For  $x,y\in\V$ with $\V_x\cup\V_{y}=\emptyset$ the random variables $v_{x}$ and $v_{y}$ are independent.
\item[(P2')] If $a(x)=a(y)$ the random variables $v\vert_{\V_x}$ and $v\vert_{\V_y}$ are identically distributed.
\end{itemize}
The assumptions (P1'), (P2') differ from the assumptions (P1), (P2) only by the fact that $v$ is defined on a disjoint union of trees instead of only on one tree.

To $v\in\W_{\mathrm{rand}}(\Om,\T)$ and $\lm\geq0$ we associate a family of random operators $H^{\lm,\om}:\ell^{2}(\V,\nu)\to\ell^{2}(\V,\nu)$ on the union of trees $\T$ via
\begin{align*}
H^{\lm,\om}:=T+ \lm v^{\om}.
\end{align*}
Moreover, for $z\in\h$ and $p> 1$ we define
\begin{align*}
\EE\gm&:=\ap{\EE\ap{\gm_j^{p}}}_{j\in\A}=\ap{\int_{\Om} \gm\ap{\Gm_{o(j)}(z,H^{\lm,\om}), \Gm_{o(j)}(z,T)}^{p}d\PP(\om)}_{j\in\A}.
\end{align*}

For $z\in \h\cup\Sigma$ let
the stochastic matrix $P=P(z):{\A\times\A}\to[0,\infty)$  be given by
\begin{align*}
P_{j,k}=p_{o(j),o(j)'}(h)\sum_{{y\in S_{o(j)'},}\atop{a(y)=k}}p_{o(j)',y}(h)+\sum_{{x\in S_{o(j)}\setminus\{o(j)'\},}\atop{a(x)=k}}p_{o(j),x}(h),\qquad j,k\in\A,
\end{align*}
where $h_x=\Gm_{x}(z,T)$,${x\in\V}$. By this choice the matrix $P$ depends continuously on $z$.

The following proposition is the vector inequality discussed in Subsection~\ref{s:meanvalue}.
\medskip

\begin{prop}\label{p:vector}(Vector inequality.)
Let $I\subset \Sigma$ be compact, $p>1$, $\T=(\V,\E)={\bigcup}_{j}(\V_j,\E_j)$ and $(\Om,\PP)$ as introduced above. Then there are constants $\de=\de(I,p)>0$, $\lm_0=\lm_0(I,p)>0$ and a function $C:[0,\lm_0]\to[0,\infty)$ with $C(\lm)\to0$ for $\lm\to0$ such that for all $z\in I+i[0,1]$, $v\in\W_{\mathrm{rand}}(\Om,\T)$ and $\lm\in[0,\lm_0)$
\begin{align*}
\EE\gm\leq (1-\de)P\;\EE\gm+C(\lm).
\end{align*}
\end{prop}

Before we prove the vector inequality above, we need an auxiliary lemma. It is an easy  consequence of Jensen's inequality.\medskip

\begin{lemma}\label{l:Jensen} For $p\geq 1$ and $x,y\geq0$ we have
\begin{align*}
(x+y)^{p}\leq(1+y)^{p-1} x^p+\ap{1+y}^{p-1}y.
\end{align*}
\begin{proof}
Writing
\begin{align*}
(x+y)^{p}=\ap{\frac{1}{1+y}\ap{1+y} x+\ap{\frac{y}{1+y}}\ap{\frac{1+y}{y}}y}^{p},
\end{align*}
we obtain  the statement directly by Jensen's inequality.
\end{proof}
\end{lemma}

\begin{proof}[Proof of Proposition~\ref{p:vector}.]
Let $j\in\A$, $o=o(j)$,
$h=\Gm_{S_{o,o'}}(z,T)$ and $g=\Gm_{S_{o,o'}}(z,H^{\lm,\om})$. Note that, by definition, $g_{o}(z,v_{o}^{\om},v_{o'}^{\om})=\Gm_{o}(z,H^{\lm,\om})$ and $g_{o'}(z,v_{o'}^{\om})=\Gm_{o'}(z,H^{\lm,\om})$.

Moreover, $g_x$ and $g_{y}$ are identically distributed for $a(x)=a(y)$ and independent for all $x,y\in S_{o,o'}$. This gives in particular $\EE(Z_0(z,v_{o'},g)^p)=\EE(Z_0(z,v_{o'},g\circ\pi)^p)$ for all $\pi\in\Pi$. We use this to compute
\begin{align*}
\EE\ap{Z_0(z,v_{o'},g)^p} &=\frac{1}{|\Pi|}\EE\ap{\sum_{\pi\in\Pi}Z_0(z,v_{o'},g\circ\pi)^p} \\
& =\frac{1}{|\Pi|}\EE\ap{\ka_{o}^{(p)}(z,v_{o'},g) \sum_{\pi\in\Pi}Z_1^{(p)}(z,v_{o'},g\circ\pi)} .
\end{align*}
In order to apply the uniform contraction, Proposition~\ref{p:ka}, we split up the expectation value. Let $R_j:=R_{o(j)}:[0,\lm_j]\to[0,\infty)$ and $\lm_j:=\lm_{o(j)}(I)>0$ be given by Proposition~\ref{p:ka}.
For $\lm\in[0,\lm_j]$ let $\mathds{1}_{R_j}$ be the characteristic function of the set $B_{R_j(\lm)}(h)=\{g\in \h^{ S_{o,o'}}\mid \gm_{S_{o,o'}}(g,h)\geq R_j(\lm)\}$ and $\mathds{1}_{R_j}^{c}$ be the characteristic function of its complement $B_{R_j(\lm)}(h)^{c}=\h^{S_{o,o'}}\setminus B_{R_j(\lm)}(h)$. We proceed by  using the definition of $\ka_o^{(p)}=\ka_{o}^{(p)}(z,v_{o'},g)$ and estimating $\ka_{o}^{(p)}\leq 1$ in the second term
\begin{align*}
\ldots&\leq\frac{1}{|\Pi|}\EE\ap{\ka_{o}^{(p)} \sum_{\pi\in\Pi}Z_1^{(p)}(z,v_{o'},g\circ\pi)\mathds{1}_{R_j}^{c}(g)} +\frac{1}{|\Pi|} \EE\ap{\sum_{\pi\in\Pi}Z_1^{(p)}(z,v_{o'},g\circ\pi) \mathds{1}_{R_j}(g)},\\
&=\frac{1}{|\Pi|}\EE\ap{\ka_{o}^{(p)} \sum_{\pi\in\Pi}Z_1^{(p)}(z,v_{o'},g\circ\pi)\mathds{1}_{R_j}^{c}(g)} +\EE\ap{Z_1^{(p)}(z,v_{o'},g)\mathds{1}_{R_j}(g)},
\end{align*}
where we used $\EE\ap{\sum_\pi Z_1^{(p)}(z,v_{o'},g\circ\pi)\mathds{1}_{R_j}(g)} =|\Pi|\EE\ap{Z_1^{(p)}(z,v_{o'},g)\mathds{1}_{R_j}(g)}$ for the second term, as $B_{R_j(\lm)}(h)$ is $\pi$ invariant. We now apply the uniform contraction, Proposition~\ref{p:ka}, to the first term with $\de_j:=\de_{o(j)}(I,p)$ taken from the proposition. For the second term, note that $Z_1(z,v_{o'},g)\leq \gm_{S_{o,o'}}(g,h)\leq R_j(\lm)$ by Lemma~\ref{l:Z}. This gives
\begin{align*}
\ldots&\leq (1-\de_j)\frac{1}{|\Pi|}\EE\ap{ \sum_{\pi\in\Pi}Z_1^{(p)}(z,v_{o'},g\circ\pi) \mathds{1}_{R_j}^{c}(g)}+R_j(\lm)^p\\
&=(1-\de_j)\EE\ap{Z_1^{(p)}(z,v_{o'},g) \mathds{1}_{R_j}^{c}(g)}+R_j(\lm)^p,
\end{align*}
as also $B_{R_j(\lm)}(h)^{c}$ is $\pi$ invariant.
Since $g_x$ is equal to $\Gm_{o(k)}(z,H^{\lm,\om})$ in distribution whenever $a(x)=k$, we get
 by the definition of $Z_{1}^{(p)}$ and $P$ and the estimate $\EE({\gm_{o(k)}^{p}\mathds{1}_{R_j}^{c}})\leq \EE({\gm_{o(k)}^{p}})$
\begin{align*}
\ldots\leq(1-\de_{j}) \sum_{k\in\A}P_{j,k}\EE\ap{\gm_{o(k)}^{p}}+R_{j}(\lm)^p
\end{align*}
with $\gm_{o(k)}=\gm\ap{\Gm_{o(k)}(z,H^{\lm,\om}),\Gm_{o(k)}(z,T)}$.
In summary, this yields
\begin{align*}
\EE\ap{Z_0(z,v_{o'},g)^p} =  (1-\de_{j}) \sum_{k\in\A}P_{j,k}\EE\ap{\gm_{o(k)}^{p}}+R_{j}(\lm)^p.
\end{align*}
We want to combine this with the two step expansion, Proposition~\ref{p:expansion}. For each $j\in\A$, we apply Proposition~\ref{p:expansion}, Lemma~\ref{l:Jensen} and the inequality above to obtain
\begin{align*}
\EE\ap{\gm_{o(j)}^p}&\leq\EE\ap{\ap{ (1+c(\lm))Z_0(z,v_{o'},g) +C(\lm)}^p}\\
&\leq(1+C(\lm))^{p-1}(1+c(\lm))^p\EE\ap{Z_0(z,v_{o'},g)^p} +(1+C(\lm))^{p-1}C(\lm)\\
&\leq (1+\ow c(\lm))(1-\de_j)\sum_{k\in\A}P_{j,k}\EE\ap{\gm_{o(k)}^{p}} +\ow C_{j}(\lm),
\end{align*}
where we set
\begin{align*}
\ow c(\lm)&:=(1+C(\lm))^{p-1}(1+c(\lm))^p-1,\\ \ow C_j(\lm)&:=(1+C(\lm))^{p-1}C(\lm)+(1+\ow c(\lm))^{p}R_j(\lm)^p,\quad j\in \A.
\end{align*}
Let us define the  constants $\lm_0>0$, $\de>0$ and the new function $C:[0,\lm_0)\to\R$ to finish the proof.
As $c(\lm), C(\lm)\to0$ for $\lm\to0$, by the two step expansion, Proposition~\ref{p:expansion}, we can let
\begin{align*}
\lm_0&:=\sup\set{t\in(0,\max_{j\in\A}\lm_j]\mid \ow c(t)<\de_j\mbox{ for all } j\in\A}
\end{align*}
and for an arbitrary $t\in[0,\lm_0)$ we let
\begin{align*}
\de&:=1-\max_{j\in\A}(1+\ow c(t))^{p-1}(1-\de_j)^{p-1}.
\end{align*}
Note that by this choice $\de>0$ as $\ow c(t)<\de_{j}$.
Moreover, let the new function $C$ be given by
$\lm\mapsto \max_{j\in\A}\ow C_j(\lm)$.
It tends to zero as $\lm\to 0$ since
 $\ow C_j(\lm),R_j(\lm)\to0$ for $\lm\to 0$.
\end{proof}

\subsection{Proof of the mean value bounds}
Having the vector inequality, the proof of Theorem~\ref{t:EGm} is essentially an application of the Perron-Frobenius theorem.

\begin{proof}[Proof of Theorem~\ref{t:EGm}] Let $p>1$, $I\subset\Sigma$ be compact and $z\in I+i(0,1]$.
Let $P=P(z)$ be the stochastic matrix defined right before Proposition~\ref{p:vector}.
Note that the entries $P_{j,k}$ are positive whenever the corresponding entries $M_{j,k}$ of the substitution matrix are positive.  Hence, by the Perron-Frobenius theorem, there is a normalized left eigenvector $u=u(z)$ of $P$ to the eigenvalue one, i.e.,  $P^{\top}u=u$, which is positive and depends continuously on $z$ in $I+i[0,1]$.
Let $\de>0$, $\lm_0>0$ and $C:[0,\lm_{0}]\to[0,\infty)$ be taken from the vector inequality, i.e., Proposition~\ref{p:vector}. Then, this proposition gives for all $\lm\in[0,\lm_0)$
\begin{align*}
\as{u,\EE\gm}\leq (1-\de)\as{u,P\EE\gm}+ C(\lm)=(1-\de)\as{u,\EE\gm}+ C(\lm),
\end{align*}
where $\as{\cdot,\cdot}$ is the standard scalar product in $\R^{\A}$. The previous inequality yields
\begin{align*}
\as{u,\EE\gm}\leq \frac{C(\lm)}{\de}.
\end{align*}
Since $\Gm_{x}(z,H^{\lm})$, $x\in\V$ is identically distributed for all $x$ carrying the same label, we get
 for every $x\in\V$ and $\lm\in[0,\lm_0]$
\begin{align*}
\EE\ap{\gm\ap{\Gm_{x}(z,H^{\lm}),\Gm_{x}(z,T)}^p}\leq c(\lm),
\end{align*}
with the function $c:[0,\lm_0)\to[0,\infty)$ given by $c(\lm):=C(\lm)/(\de\eps)$ and $\eps:=\min_{j\in\A}\min_{z\in I+i[0,1]} u_j(z)$. As $C(\lm)\to0$ for $\lm\to0$ by the vector inequality,  Proposition~\ref{p:vector}, so does $c$.
\end{proof}

Let us now turn to the proof of Theorem~\ref{t:KleinG}.

\begin{proof}[Proof of Theorem~\ref{t:KleinG}]
In contrast to the beginning of the section, we let $\T=(\V,\E)$ be one single tree $\T(M,j)$, $j\in\A$,  and let $o=o(j)\in\V$ be the root of $\T$. Let $I\subset\Sigma$ be compact. We recall formula \eqref{e:G} from Proposition~\ref{p:G}, which reads for $x\in\V$, $y\in S_{x}$ and $z\in\h$ as
$$G_{y}(z,H^{\lm,\om})=\Gm_y(z,H^{\lm,\om})+ |t(x,y)|^{2}\Gm_{y}(z,H^{\lm,\om})^{2}G_{x}(z,H^{\lm,\om}).$$
We set $g=\Gm_{x}(E+i\eta,H^{\om,\lm})$, $h=\Gm_{x}(E+i\eta,T)$ for $E\in I$, $\eta\in (0,1]$ and arbitrary $x\in\V$.
The strategy is to check the statements first for $x=o$ and then derive the general case using  formula \eqref{e:G} above.

(1.) We employ the inequality
\begin{align*}
\mo{\xi}\leq 4\gm(\xi,\zeta)\Im \zeta+2|\zeta|,\qquad \xi,\zeta\in\h,
\end{align*}
from \cite{FHS2}. (This inequality is obvious for $|\xi|\leq 2| \zeta|$ and follows from
$|\xi|\Im \xi\leq|\xi|^{2}\leq 2|\xi-\zeta|^2+2|\zeta|^2\leq 4|\xi-\zeta|^2$ for $|\xi|\geq2|\zeta|$.)\\
 Let $\lm_0>0$ be strictly less than the constant $\lm_0(I,p)$ from Theorem~\ref{t:EGm}. Then, by Theorem~\ref{t:EGm} and the inequality above, it follows for all $p>1$
\begin{align*}
\sup_{E\in{I}}\sup_{\eta\in(0,1]} \sup_{\lm\in[0,\lm_0]}\EE\ap{|g|^{p}}<\infty.
\end{align*}
This implies (1.) for $x=o$. For general $x\in\V$ we get the statement inductively using \eqref{e:G}  and  Hölder's inequality. (Note that $f\cdot g$ belongs to all $L^{p}$, $p>1$, if $f,g$  belong to all $L^{p}$, $p>1$.)

(2.) We apply Fatou's lemma and  Fubini's theorem to get by (1.)
\begin{align*}
\EE\ap{\liminf_{\eta\downarrow0}\int_{I}|g|^{p}dE}&\leq \liminf_{\eta\downarrow0}\int_{I}\EE\ap{|g|^{p}}dE\\
&\leq\sup_{E\in{I}}\sup_{\eta\in(0,1]} \sup_{\lm\in[0,\lm_0]}\EE\ap{|g|^{p}} \Leb(I)\\
& <\infty.
\end{align*}
We conclude
\begin{align*}
\liminf_{\eta\downarrow0}\int_{I}|g|^{p}dE<\infty
\end{align*}
for almost all $\om\in\Om$ which is (2.) for $x=o$.  In order to prove the statement for general $x$ we employ the same  calculation with $g=G_x$. The statement then follows by using (1.).

(3.) Note that $\Im g\to0$ as $\eta\downarrow 0$ implies $\gm(g,h) \to\infty$. Now, for every $E\in\Sigma$ and $\lm\in[0,\lm_0(E)]$ the case $\gm(g,h) \to\infty$ can only occur on a set of $\PP$-measure zero by  Theorem~\ref{t:EGm}.
Moreover, by the pointwise converges of the Green function, Lemma~\ref{l:Gpointwise} we have that the limit $\lim_{\eta\downarrow0}g=\lim_{\eta\downarrow0}\Gm_{x}(E+i\eta, H^{\lm,\om})$ exists for almost every $E\in\R$. Thus, the statement follows from the extension from $\Gm$ to $G$, Proposition~\ref{p:G2}~(3.).

(4.) By the Cauchy-Schwarz inequality and Theorem~\ref{t:EGm} we get
\begin{align*}
\EE\ap{|g-h|^{p}}^2\leq\EE\ap{\gm(g,h)^{p}} \EE\ap{\ap{\Im{g}\Im{h}}^{p}}\leq c(\lm)\EE\ap{\mo{g}^{p}}|h|^{p}.
\end{align*}
Note that $h=\Gm_{o}(E+i\eta,T)$ is uniformly bounded in $\eta $ by Lemma~\ref{l:Gmbounds}. Hence, the right hand side is bounded by (1.) and tends to zero as $c(\lm)\to0$ for $\lm\to0$. These quantities can be chosen uniformly for $E$ in a compact interval $I\subset\Sigma$ which gives (4.) for $x=o$. In order to prove the convergence for general $x$ we again use \eqref{e:G}  inductively. More precisely, we obtain the statement as by the Hölder inequality $f_n+f_n^2g_n\to f+f^2g$ in $L^p(I,dE)$ for all $p\in(1,\infty)$ whenever $f_n\to f$ and $g_n\to g$ in $L^p(I,dE)$ for all $p\in(1,\infty)$.
\end{proof}

Finally,  we prove  Theorem~\ref{main3}.\medskip

\begin{proof}[Proof of Theorem~\ref{main3}] Let $I,I'\subset\Sigma$ be compact such that $I\subset \inn I'$. Theorem~\ref{t:KleinG}~(2.) implies the existence of $\lm_0=\lm_{0}(I')>0$ such that for all $\lm\in[0,\lm_0]$, $v\in\W_{\mathrm{rand}}(\Om,\T)$ and almost all $\om\in\Om$ the assumptions of the criterion for the absence of singular spectrum in Theorem~\ref{t:Klein} hold.  Therefore,  we get $\si_{\mathrm{sing}}(H^{\lm,\om})\cap I=\emptyset$ as $I\subset\inn I'$. Moreover, by Theorem~\ref{t:KleinG}~(3.) and the vague convergence of the spectral measures, Lemma~\ref{l:mu}, we have $I\subset I'\subset\si(H^{\lm,\om})$ for a set of full $\PP$-measure. This finishes the proof.
\end{proof}

\section{Proof of uniform contraction}\label{s:kaproof}
In this section we prove the uniform contraction claimed in Proposition~\ref{p:ka}, i.e.,
\begin{align*}
 \ka_o^{(p)}(g)\leq 1-\de,
\end{align*}
for $g$ in a sufficiently large subset of $\h^{S_{o,o'}}$ and $p>1$. In the following subsection we derive a formula for $\ka_{o}^{(p)}$ and discuss the strategy of the proof.

For the rest of the section we fix the following quantities. Let $I\subset\Sigma$ be compact and $p>1$. Moreover, let $o\in\V$ and $o'\in S_o$ with $a(o)=a(o')$ be  fixed. For $z\in I+i[0,1]$ we denote
\begin{align*}
h=\Gm_{S_{o,o'}}(z,T)\qand h_{o'}=\Gm_{o'}(z,T).
\end{align*}
We keep the dependence on $z$ suppressed since we are dealing with a continuous function on a compact set.

\subsection{A formula for the averaged contraction coefficient}\label{ss:kaformula}
We want to represent $\ka_o^{(p)}$ as a contraction sum. We discuss how uniform contraction is proven and what kind of problems we have to deal with.

Recall that $\Pi=\Pi(S_{o,o'})$ is the set of permutations of $S_{o,o'}$ which leave the labels invariant. We extend $\pi\in\Pi$ to the vertex $o'$ by $\pi(o')=o'$. For given
$g\in \h^{S_{o,o'}}$ the recursion map determines the value  $g_{o'}$.
Let $\pi\in\Pi$, $v\in\R$, $z\in\h$ and $g\in\h^{S_{o,o'}}$. We define $g^{(\pi)}:=g^{(\pi)}(z,v) \in\h^{S_{o}\cup S_{o'}}$ by
\begin{align*}
g_{o'}^{(\pi)}:=\Psi^{(T)}_{z-v,o'}\ap{(g_{\pi(x)})_{x\in S_{o'}}}\qand g_{x}^{(\pi)}:=g_x,\quad x\in S_{o,o'}.
\end{align*}
We can think of $g^{(\pi)}$ as taking $g\circ\pi$ to define $g_{o'}^{(\pi)}$ and permuting back by $\pi^{-1}$ afterwards.
The reason for this definition is that we want to consider the values of $g$ at the indices $x$ in the permuted spheres $\pi(S_{i})$, $i\in\{o,o'\}$. This is more direct than choosing $y\in S_{i}$ first in order to consider $x=\pi(y)$.

We adapt the quantities $p_{x}$, $c_{x}$, $q_x$, $Q_{x,y}$, $\al_{x,y}$, $\gm_{x}$, $i\in\{o,o'\}$, $x,y\in S_{i}$ from Section~\ref{s:Contraction} and Section~\ref{s:expansion} to the application of a permutation from $\Pi$.  We define for
$\pi\in\Pi$  and $x\in S_{o,o'}$
\begin{align*}
p_{x}^{(\pi)}
:=\left\{\begin{array}{ll}
p_{x}(h) &: x\in\pi(S_{o})\setminus\{o'\},  \\
p_{o'}(h)p_{x}(h) &: x\in\pi(S_{o'}), \\
\end{array}\right.\;
\mbox{ where }p_{x}(h)=\frac{|t(i,x)|^2\Im h_{x}}{\sum_{y\in S_{i}}|t(i,y)|^2\Im h_{y}},
\end{align*}
with $i\in\{o,o'\}$ such that $x\in\pi(S_{i})$. For $g\in \h^{S_{o,o'}}$, $x\in\pi(S_{i})$, $i\in\{o,o'\}$ denote
\begin{align*}
q_{x}^{(\pi)}:=q_{x}(g^{(\pi)}):=\frac{|t(i,x)|^2\Im g_{x}^{(\pi)}}{\sum_{y\in \pi(S_{i})}|t(i,y)|^2\Im g_{y}^{(\pi)}}
\end{align*}
and
\begin{align*}
\gm^{(\pi)}_x&:=\gm(g_{x}^{(\pi)},h_{x}).
\end{align*}
For $x,y\in \pi(S_{i})$ with $g^{(\pi)}_x\neq h_{x}$  and $g^{(\pi)}_y\neq h_{y}$ denote
\begin{align*} Q_{x,y}^{(\pi)}&:=Q_{x,y}(g^{(\pi)},h)=\frac{\ab{\Im g_x^{(\pi)}\Im g_y^{(\pi)}\Im h_x\Im h_y\gm_x^{(\pi)}\gm_y^{(\pi)}}^\frac{1}{2}}{\frac{1}{2}\ab{\Im g_x^{(\pi)}\Im h_y\gm_y^{(\pi)}+\Im g_y^{(\pi)}\Im h_x\gm_x^{(\pi)}}},\\ \al_{x,y}^{(\pi)}& :=\al_{x,y}(g^{(\pi)},h) =\arg\ab{(g_x^{(\pi)}-h_x)\ov{(g_y^{(\pi)}-h_y)}}
\end{align*}
and $Q_{x,y}^{(\pi)}=0$ for $g^{(\pi)}_x= h_{x}$ or $g^{(\pi)}_y= h_{y}$.
Note that in the case where the index $x$ is not $o'$ we have by definition $g_x=g_x^{(\pi)}$. So, we omit the superscript `${(\pi)}$' for $x,y\neq o'$, i.e., we write $q_{x}=q^{(\pi)}_{x}$, $\gm_{x}=\gm_{x}^{(\pi)}$, $Q_{x,y}=Q_{x,y}^{(\pi)}$ and $Q_{x,y}=\al_{x,y}^{(\pi)}$ if $x,y\neq o'$.
Furthermore, we denote $c_{x}^{(\pi)}:=c_{\pi^{-1}(x)}(g^{(\pi)}\circ\pi,h)$ with $x\in\pi(S_{i})$, $i\in\{o,o'\}$ which gives
\begin{align*}
c_{x}^{(\pi)}={\sum_{v\in \pi(S_i)}q_{v}^{(\pi)}(g) Q_{x,v}^{(\pi)}(g,h)\cos\al_{x,v}^{(\pi)}(g,h)}.
\end{align*}
A direct calculation gives the following formula for $\ka_o^{(p)}$
\begin{align}\label{e:ka}
\ka_o^{(p)}=\frac{\sum\limits_{\pi\in\Pi} \ap{\sum\limits_{y\in\pi(S_{o'})} p_{y}^{(\pi)}c_{o'}^{(\pi)}c_{y}^{(\pi)}\gm_{y}+ \sum\limits_{x\in\pi(S_o\setminus\{o'\})} p_{x}^{(\pi)}c_{x}^{(\pi)}\gm_{x}}^p}{\sum\limits_{\pi\in\Pi} \ap{\sum\limits_{x\in S_{o,o'}} p_{x}^{(\pi)}\gm_{x}^p}}.
\end{align}

Our aim is to show that there is $\de>0$ such that $\ka_o^{(p)}\leq 1-\de$ outside of a compact set.
By the basic properties of $Z_0$ and $Z_1^{(p)}$ in Lemma~\ref{l:Z} two ways to prove contraction are indicated.
The first one is to estimate
$$Z_0(z,v,g)^p\leq Z^{(1)}_1(z,v,g)^p\leq (1-\de) Z^{(p)}_1(z,v,g),$$
where the $(1-\de)$ is squeezed out of the error term of the Jensen inequality. Secondly, one can try to find suitable $x$, $y$ and  $\pi$ such that the contraction quantities $Q_{x,y}^{(\pi)}$ or $\cos\al_{x,y}^{(\pi)}$ are less than one. Then, the corresponding $c_{x}^{(\pi)}$ is smaller than one which implies $\ka_o^{(p)}<1$.
However,  for uniform contraction we need more. Even if we manage to show uniform bounds for the contraction quantities $\cos\al_{x,y}^{(\pi)}$ or $Q_{x,y}^{(\pi)}$, (i.e., that one of them is less or equal $c$ for some $c<1$), two problems can occur which make this contraction `invisible'. The first problem is that the weight $q_{y}^{(\pi)}$ might  not be bounded from below and thus the contraction quantity $c_{x}^{(\pi)}$ can become arbitrary close to one. The second problem is that even if one has $c_{x}^{(\pi)}\leq c$ for some $c<1$ the quantity $p^{(\pi)}_x\gm_{x}/\sum_{\pi}Z_1^{(p)}(z,v,g\circ\pi)$ might become so small such that $\ka_{o}^{(p)}$ becomes arbitrary close to one. Our strategy to quantify these problems is to introduce the notion of visibility in Subsection~\ref{s:Vis}.

We then prove Proposition~\ref{p:ka} by distinguishing three cases for $g\in\h^{S_{o,o'}}$.

In Case~1, Subsection~\ref{ss:Case1}, we look at the case where there is $x\in S_{o,o'}$ such that $\gm_{x}/\max_{y\in S_{o,o'}}\gm_y$ is very small. This implies that $p^{(\pi)}_x\gm_{x}/\sum_{\pi}Z_1^{(p)}(z,v,g\circ\pi)$ is very small. In this case we get uniform contraction from the error term of Jensen's inequality.

In Case~2, Subsection~\ref{ss:Case2}, we assume that we are not in Case~1, but that there is $\pi\in\Pi$ and $x\in\pi(S_{o'})$ such that $\Im g_{x}/\max_{y\in\pi(S_{o'})}\Im g_{y}$ is very small. Hence, $q_{x}^{(\pi)}$ is very small. Then, we exploit that $Q_{x,y}$ is a quotient of a geometric and an arithmetic mean.

For Case~3, Subsection~\ref{ss:Case3}, we prove that there always are $\pi$, $x,y$ such that $\cos\al_{x,y}^{(\pi)}$ is uniformly smaller than one. Assuming we are not in the Case~1 or~2 none of the problems mentioned above occurs and we conclude uniform contraction.

Finally,  in Subsection~\ref{ss:kaproof}, we put the pieces together.


\subsection{Visibility}\label{s:Vis}

As discussed above, it makes only sense to look for $\al_{x,y}^{(\pi)}$ and $Q_{x,y}^{(\pi)}$ where the corresponding weights are uniformly bounded from below. Otherwise the contraction might become `invisible'. This is quantified below by the sets of visible indices.
\medskip

\begin{Def}\label{d:Vis}
For $g\in \h^{S_{o,o'}}$ and $\de>0$ we define the set of indices in $S_{o,o'}$ \emph{visible with respect to $\gm$ }by
\begin{align*}
\Vis_{\gm}(g,\de)&:=\set{x\in S_{o,o'}\mid \min_{y\in S_{o,o'}}\frac{\gm_x}{\gm_y}>\de}
\end{align*}
and the  set of indices in $\pi(S_{i})$, $\pi\in\Pi$,  $i\in\{o,o'\}$, \emph{visible with respect to the imaginary parts} by
\begin{align*}
\Vis_{\Im}^{i}(g,\pi,\de)&:=\set{x\in \pi(S_{i})\mid \min_{y\in \pi(S_{i})}\frac{\Im g^{(\pi)}_x}{\Im g^{(\pi)}_y}>\de}.
\end{align*}
\end{Def}

Let us remark two simple facts. Firstly,  for $0<\de'\leq\de$ we have
$\Vis_{\gm}(g,\de)\subseteq\Vis_{\gm}(g,\de')$ and $\Vis_{\Im}^{i}(g,\pi,\de)\subseteq\Vis_{\Im}^{i}(g,\pi,\de')$, for $i\in\{o,o'\}$.
Secondly, we have $\Vis_{\gm}(g,\de)=\emptyset$ or $\Vis_{\Im}^{i}(g,\pi,\de)=\emptyset$, $i\in\{o,o'\}$ if and only if $\de\geq1$.

We now prove two lemmas which demonstrate how to put the concept of visibility into action.
We start by visibility with respect to the imaginary parts.

 Let
$z\in I+i[0,1]$, $v\in\R$ and $g\in \h^{S_{o,o'}}$ be fixed for this subsection.
\medskip

\begin{lemma}\label{l:using_alQ<1}
Let $\eps>0$, $c\in[0,1)$. There is $\de:=\de(\eps,c)>0$ such that for all $\pi\in\Pi$, $i\in\{o,o'\}$, $\ow x\in\pi(S_{i})$ and  $\ow y\in \Vis_{\Im}^{i}(\pi,g,\eps)$ with
\begin{align*}
Q_{\ow x,\ow y}^{(\pi)}\leq c\qquad\mbox{ or } \qquad\cos\al_{\ow x,\ow y}^{(\pi)}\leq c,
\end{align*}
we have
\begin{align*}
c_{\ow x}^{(\pi)}\leq 1-\de.
\end{align*}
\begin{proof}
Note that, for all $\pi\in\Pi$, $i\in\{o,o'\}$ and $x\in\Vis_{\Im}^{i}(\pi,g,\de_2)$, we have
\begin{align*}
q_{x}^{(\pi)}
=\frac{|t(i,x)|^2}{\sum_{y\in \pi(S_{i})}|t(i,y)|^2\Im g_{y}^{(\pi)}/\Im g_{x}^{(\pi)}}
\geq{\eps}\frac{\min_{y\in S_{o}}|t(o,y)|^2}{\sum_{w\in S_{o}}|t(o,w)|^2}=:\eps'
\end{align*}
and $\eps'>0$.
Let $\ow x,\ow y$ be chosen according to the assumption.
We get by estimating $Q_{\ow x,y}^{(\pi)}\cos\al_{\ow x,y}^{(\pi)} \leq 1$ for all $y\neq \ow y$ that
\begin{align*}
c_{\ow x}^{(\pi)}\leq \sum_{y\in\pi(S_{i})\setminus\{\ow y\}}q_{y}^{(\pi)}+q_{\ow y}^{(\pi)}Q^{(\pi)}_{\ow x, \ow y}\cos\al^{(\pi)}_{\ow x, \ow y}=1-q_{\ow y}^{(\pi)}(1-Q^{(\pi)}_{\ow x, \ow y}\cos\al^{(\pi)}_{\ow x,\ow y})\leq1-\de,
\end{align*}
where we set $\de:=\eps'(1-c)$.
\end{proof}
\end{lemma}

The next lemma shows how to use visibility with respect to $\gm$.\medskip

\begin{lemma}\label{l:using_c<1}
Let  $\eps>0$, $c\in[0,1)$. There is $\de:=\de(\eps,c)>0$ such that if
\begin{align*}
    \Vis_{\gm}(g,\eps)=S_{o,o'}
\qqand
c_{x}^{(\pi)}\leq c,
\end{align*}
for some $\pi\in\Pi$, $i\in\{o,o'\}$, $x\in\pi(S_{i})$, we have
\begin{align*}
\ka_o^{(p)}(z,v,g)\leq 1-\de.
\end{align*}
\begin{proof}
As the $p_{x}=p_{x}(h)$, $x\in S_{o,o'}$, are uniformly larger than zero for all $z\in I+i[0,1]$ we conclude the existence of $\eps'>0$ such that for all $\pi\in\Pi$ and $x\in\Vis_{\gm}(g,\eps)$
\begin{align*}
\frac{p^{(\pi)}_{x}\gm_{x}^p}{\sum_{\pi} Z_{1}^{(p)}(z,v,g\circ\pi)} =\frac{1}{\sum_{\pi} \sum_{y\in S_{o,o'}}(p_{y}^{(\pi)}/p^{(\pi)}_{x})(\gm_{y}/\gm_{x})^p}
\geq\frac{\eps^p}{\sum_{\pi} \sum_{y\in S_{o,o'}}p_{y}^{(\pi)}/p^{(\pi)}_{x} }\geq \eps'.
\end{align*}
Let $\pi\in\Pi$, $i\in\{o,o'\}$, $x\in\pi(S_{i})$ be chosen according to the assumption. Choose
$\oh x\in S_{o,o'}$ such that  $\oh x=x$ if $x\neq o'$ and arbitrary in $\pi(S_{o'})$ otherwise.
We get by estimating $c_{y}^{(\pi)}\leq1$ for $y\neq  x$ and $c_{x}^{(\pi)}\leq c$, Jensen's inequality and the previous estimate
\begin{align*}
\ap{Z_{0}(z,v,g\circ\pi)}^p&\leq \ap{\sum_{y\in S_{o,o'}\setminus\{\oh{x}\}} p_{y}^{(\pi)}\gm_{y}+p_{\oh{x}}^{(\pi)}\gm_{\oh{x}}c_{x}^{(\pi)}}^{p}\\ &\leq \sum_{y\in S_{o,o'}\setminus\{\oh{x}\}}p_{y}^{(\pi)}\gm_{y}^p +p_{\oh{x}}^{(\pi)}\gm_{\oh{x}}^{p}c^p\\
&=Z_{1}^{(p)}(z,v,g\circ\pi)-p_{\oh{x}}^{(\pi)}\gm_{\oh{x}}^{p}(1-c^p)\\
&\leq Z_{1}^{(p)}(z,v,g\circ\pi)
-(1-c^p)\eps'\sum_{\pi'\in\Pi}Z_{1}^{(p)}(z,v,g\circ\pi').
\end{align*}
We set $\de:=\eps'(1-c^p)$. Applying the basic  estimate  $Z_0$ with $Z_1^{(p)}$, Lemma~\ref{l:Z}, to all $\ow \pi\in\Pi$, $\ow \pi\neq\pi$ and using the inequality above, we get
\begin{align*}
\ka_{o}^{(p)}(g) \leq\frac{1}{\sum_{\pi'\in\Pi}Z_{1}^{(p)}(z,v,g\circ\pi')} \ap{\sum_{\ow\pi\neq\pi}Z_{1}^{(p)}(z,v,g\circ\ow\pi) +\ap{Z_{0}(z,v,g\circ\pi)}^p} \leq 1-\de.
\end{align*}
\end{proof}
\end{lemma}


\subsection{Case~1: A Jensen type inequality}\label{ss:Case1}
In this section we will prove the statement which was discussed as Case~1 in Subsection~\ref{ss:kaformula}.
We will need the following auxiliary constant
\begin{equation*}\label{e:c1}
c_1:=c_1(I):=\ap{\max_{z\in I+i[0,1]}\max_{\pi\in\Pi}\max_{x,y\in S_{o,o'}}\frac{(1-p^{(\pi)}_{x})}{p^{(\pi)}_{y}}}^{-1}. \end{equation*}
As $h_x=\Gm_{x}(z,T)$, $z\in I$ and $I\subset \Sigma$ is chosen compact we have
\begin{align*}
0<\min_{z\in I+i[0,1]}\min_{x\in S_{o,o'}}\Im h_{x}\leq \max_{z\in I+i[0,1]}\max_{x\in S_{o,o'}}\Im h_{x}<\infty.
\end{align*}
Therefore,  $p_{x}^{(\pi)}\in(0,1)$, $x\in S_{o,o'}$. Let
$z\in I+i[0,1]$ and $v\in\R$ be fixed for this subsection.

\medskip

\begin{prop}\label{p:invisgm}
Let $p>1$, $\eps\in(0,c_1)$.  There is $\de=\de(\eps)>0$  such that for all  $g\in \h^{S_{o,o'}}$ with
\begin{align*}
\Vis_{\gm}(g,\eps)\neq S_{o,o'},
\end{align*}
we have
\begin{align*}
\ka_o^{(p)}(z,v,g)\leq 1-\de.
\end{align*}
\end{prop}

For the proof we employ a refinement of Jensen's inequality for monomials. To this end, we take a closer look at the error term in Jensen's inequality.
\medskip

\begin{lemma}Let $f:\R\to\R$ be twice continuously differentiable, $\lm\in[0,1]$ and $x,y\in\R$. Then
\begin{align*}
f(\lm x+(1-\lm)y)=\lm f(x)+(1-\lm)f(y)-E_f(\lm,x,y),
\end{align*}
where
\begin{align*}
E_f(\lm,x,y)=(x-y)^2\!\int_0^1\!\! \ap{\lm t \mathds{1}_{[0,1-\lm]}(t)+(1-\lm)(1-t)\mathds{1}_{[1-\lm,1]}(t)}f''((1-t)x+ty)dt    \end{align*}
and $\mathds{1}_A$ is the characteristic function of  a set $A$.
\begin{proof} We take the Taylor expansion of $f$ in $x_0=\lm x+(1-\lm)y$  at the points $x$ and $y$ with integral error term. Inserting this into $\lm f(x)+(1-\lm)f(y)$ yields the statement.
\end{proof}
\end{lemma}

Jensen's inequality for twice continuously differentiable functions is a direct corollary. However,  if we know more about the function $f$ we can use the explicit error term to get finer estimates.\medskip

\begin{lemma}\label{l:Jensen2}Let $p\geq1$, $\lm\in[0,1]$ and $a,b\in[0,\infty)$, $a\geq b$. Then
\begin{align*}
(\lm a+(1-\lm)b)^p\leq(1-\de_p)\ap{\lm a^p+(1-\lm)b^p},
\end{align*}
where
\begin{align*}
\de_p:=\de_p(\lm,b/a):=\ap{1-{b}/{a}}^2\left\{\begin{array}{ll}
p(p-1)\lm(1-\lm)/2&: p\in[1,2), \\
\lm(1-\lm^{p-1})&: p\geq2
\end{array}\right.
\end{align*}
\begin{proof}
The statement is trivial for $p=1$, $a=b$ or $\lm\in\{0,1\}$. Therefore,  we only treat the case where $p>1$, $a\neq b$ and $\lm\not\in\{0,1\}$.

Assume first that $a=1$ and $b\in[0,1)$. Then, the error term $E_p:=E_{(\cdot)^p}$ from the previous lemma can be estimated for $p\geq2$, using $b\geq0$,
\begin{align*}
\frac{E_p(\lm,1,b)}{(1-b)^2}&=p(p-1)\int_0^1 \ap{\lm t \mathds{1}_{[0,1-\lm]} +(1-\lm)(1-t)\mathds{1}_{[1-\lm,1]}}((1-t)+tb)^{p-2}dt\\
&\geq p(p-1)\ap{\lm\int_{0}^{1-\lm}t(1-t)^{p-2}dt+ (1-\lm) \int_{1-\lm}^{1}(1-t)^{p-1}dt}\\
&=\lm\ap{1-\lm^{p-1}}.
\end{align*}
On the other hand, for $p\in(1,2)$ we get by estimating $((1-t)+tb)^{(p-2)}\geq1$,
\begin{align*}
\frac{E_p(\lm,1,b)}{(1-b)^2}&\geq p(p-1)\int_0^1 \ap{\lm t \mathds{1}_{[0,1-\lm]}+(1-\lm)(1-t)\mathds{1}_{[1-\lm,1]}}dt =p(p-1)\frac{\lm(1-\lm)}{2}.
\end{align*}
Employing the previous lemma  for $a=1$, $b\in[0,1]$, we get since $\lm +(1-\lm) b^p\leq 1$,
\begin{align*}
(\lm +(1-\lm)b)^p&\leq{\lm +(1-\lm)b^p}-\de_p\leq(1-\de_p)\ap{\lm +(1-\lm)b^p}.
\end{align*}
Now, let $a\geq b\geq0$ be arbitrary. We get, employing the previous inequality, that
\begin{align*}
(\lm a+(1-\lm)b)^p&=a^p\ap{\lm +(1-\lm)\frac{b}{a}}^p\leq (1-\de_p)\ap{\lm a^p+(1-\lm)b^p}.
\end{align*}
\end{proof}
\end{lemma}

With these preparations we are ready to prove Proposition~\ref{p:invisgm}.\medskip

\begin{proof}[Proof of Proposition~\ref{p:invisgm}]
Let  first $x_0\in S_{o,o'}$ be chosen arbitrary. By the previous lemma we have
\begin{align*}
\ap{\sum_{x\in S_{o,o'}}p^{(\pi)}_{x}\gm_{x}}^p &=\ap{\ap{1-p^{(\pi)}_{x_0}}\ap{\sum_{x\in S_{o,o'}\setminus\{x_0\}}\frac{p^{(\pi)}_{x}}{(1-p^{(\pi)}_{x_0})}\gm_{x}} +p^{(\pi)}_{x_0}\gm_{x_0}}^p\\
&\leq (1-\de_{p})\ap{(1-p^{(\pi)}_{x_0}) \ap{\sum_{x\in S_{o,o'}\setminus\{x_0\}}\frac{p_{x}^{(\pi)}} {(1-p^{(\pi)}_{x_0})}\gm_{x}}^p+p^{(\pi)}_{x_0}\gm_{x_0}^p}\\
&\leq(1-\de_p)\sum_{x\in S_{o,o'}}p^{(\pi)}_{x}\gm_{x}^p =(1-\de_p) Z_{1}^{(p)}(z,v,g\circ \pi),
\end{align*}
where we also applied the Jensen inequality in the third line. Note that we do not have control over $\de_p$ so far. This will be done later by suitable choice of $x_0$. We get by estimating $c_{x}^{(\pi)}\leq1$ in \eqref{e:ka} for all $x\in \pi(S_{o,o'})$
\begin{align*}
\ka_o^{(p)}(g) \leq\frac{\sum_{\pi\in\Pi}\ap{\sum_{x\in S_{o,o'}} p_{x}^{(\pi)}\gm_{x}}^p}{\sum_{\pi\in\Pi} Z_{1}^{(p)}(z,v,g\circ\pi)}\leq 1-\de_p.
\end{align*}
Let us give an estimate for $\de_p$.
For the application of Lemma~\ref{l:Jensen2} we chose $\lm=1-p^{(\pi)}_{x_0}$.
We estimate the term in $\de_p$ involving $\lm$ and $p$ by the quantity $c_2:=c_2(I)$
\begin{align*}
c_{2}&:=\min_{z\in I+i[0,1]}\min_{\pi\in\Pi}\min_{x\in S_{o,o'}}\min\set{\frac{p(p-1)}{2}p^{(\pi)}_{x}(1-p^{(\pi)}_{x}), (1-p^{(\pi)}_{x})\ap{1-(1-p^{(\pi)}_{x})^{p-1}}},
\end{align*}
As the $p_x$'s are uniformly larger than zero, we have $c_2>0$.
Let now $x_0,\oh x\in S_{o,o'}$ be such that $\gm_{x_0 }=\min_{x\in S_{o,o'}}\gm_{x}$, $\gm_{\oh x }=\max_{x\in S_{o,o'}}\gm_{x}$. This is the part where the assumption $ \Vis_{\gm}(g,\eps)\neq S_{o,o'}$ comes into play. Namely, by assumption  $x_0$ must be in $ S_{o,o'}\setminus \Vis_{\gm}(g,\eps)$ and, hence, $\gm_{x_0}/\gm_{\oh x}\leq \eps$.
For the application of Lemma~\ref{l:Jensen2} we chose $a=\sum_{x}\frac{p_{x}}{(1-p^{(\pi)}_{x_0})}\gm_{x}$ and $b=\gm_{x_0}$.
We obtain by the definition of $c_1$
\begin{align*}
\frac{a}{b}=\frac{(1- p^{(\pi)}_{x_0})\gm_{x_0}} {\sum_{x\in S_{o,o'}} p^{(\pi)}_{x}\gm_{x}}
= \frac{\gm_{x_0}}{\gm_{\oh x}}\frac{(1- p^{(\pi)}_{x_0})} {\sum_{x\in S_{o,o'}} p^{(\pi)}_{x}\gm_{x}/\gm_{\oh x} }
\leq \eps\frac{(1-p^{(\pi)}_{x_{0}})}{p^{(\pi)}_{\oh x}}\leq \frac{\eps}{c_1}.
\end{align*}
Hence, $
\de_{p}\geq c_2\ap{1-{\eps}/c_1}^2>0.$
\end{proof}

\subsection{Case 2: Geometric and arithmetic means}\label{ss:Case2}

This subsection deals with Case~2 as discussed in Subsection~\ref{ss:kaformula}.  Define
\begin{equation*}\label{e:eps0}
\eps_0:=\min_{z\in I+i[0,1]}\min_{x,y\in S_{o,o'}}\frac{\Im h_x}{\Im h_{y}}.
\end{equation*}
By the definition of $\Sigma$ and $I\subset\Sigma$ compact we have $\eps_0>0$.  Let again
$z\in I+i[0,1]$ and $v\in\R$ be fixed for this subsection.

\medskip

\begin{prop}\label{p:invisIm}
Let $\eps>0$, $\eps'\in(0,\eps \eps_0)$. There is $\de=\de(\eps,\eps')>0$ such that for all $\pi\in\Pi$ and $g\in \h^{S_{o,o'}}$ with
\begin{align*}
    \Vis_{\gm}(g,\eps)=S_{o,o'}
\qqand
\Vis_{\Im}^{o'}(\pi,g,\eps')\neq \pi(S_{o'}),
\end{align*}
we have
\begin{align*}
c_{x}^{(\pi)}\leq (1-\de)\quad\mbox{for some $x\in\pi(S_{o'})$}.
\end{align*}
\begin{proof}
Since $Q_{x,y}$ is the ratio of a geometric and an arithmetic mean it is one if and only if the averaged quantities are equal. For a finer analysis we introduce,
for  $\pi\in\Pi$, $x,y\in\pi(S_{o'})$,
\begin{align*}
\te_{x,y} :=\frac{\Im g_x\Im h_{y}\gm_y}{\Im g_y\Im h_{x}\gm_x}
\end{align*}
and note that $\te_{x,y}=1/\te_{y,x}$. Then
\begin{align*}
Q_{x,y}= \frac{2}{ \ap{\sqrt{\te_{x,y}}+\sqrt{\te_{y,x}}}}.
\end{align*}
We have $Q_{x,y}<1$ if and only if $\te_{x,y}\neq 1$.\\
Let $\pi\in\Pi$ be such that there is $\Vis_{\Im}^{o'}(\pi,g,\eps')\neq \pi(S_{o'})$, which implies the existence of $x,y\in\pi(S_{o'})$ such that ${\Im g_{y}}/{\Im g_{x}}\leq\eps'$.
Since, by assumption, $x, y\in \Vis_{\gm}(g,\eps)$ we have $\gm_{y}/\gm_{x}>\eps$.
Hence,
\begin{align*}
\te_{x,y}>\frac{\eps_{0}\eps}{\eps'}.
\end{align*}
By the choice of $\eps'\in(0,\eps_0\eps)$, we have $\te_{x,y}>1$. Moreover, employing the inequality $a/(1+a^2)\leq b/(1+b^2)$ for $a\geq b\geq1$ we get
\begin{align*}
Q_{x,y}=2\frac{\sqrt{\te_{x,y}}}{1+\te_{x,y}}\leq \frac{2\sqrt{\eps_0\eps \eps'}}{\eps_0 \eps+\eps'}<1.
\end{align*}
The statement now follows from Lemma~\ref{l:using_alQ<1}.
\end{proof} \end{prop}

\subsection{Case 3: A general bound on the relative arguments}\label{ss:Case3}

In this subsection we show that there are always $\pi\in\Pi$ and $x,y\in S_{o,o'}$ such that $\cos\al_{x,y}^{(\pi)}$ is uniformly smaller than one. We define a quantity, related to the `minimal angle with the real axis', by
\begin{align*}
\de_0&:=\frac{1}{4}\min_{z\in I+i[0,1]}\min_{x\in S_{o,o'}}\min\set{\ma{\arg h_x},\ma{\pi-\arg h_x}}.
\end{align*}
By the definition of $\Sigma$ and $I\subset\Sigma$ compact we have $\de_0>0$.  Let again
$z\in I+i[0,1]$ be fixed for this subsection.
\medskip

\begin{prop}\label{p:al>0}
There is $c=c(\de_0)<1$, $\lm_0=\lm_0(\de_0,\eps_0)>0$ and $R:[0,\lm_0)\to[0,\infty)$ with $R(\lm)\to0$ as $\lm\to0$ such that for all $\lm\in[0,\lm_0)$, $v\in[-\lm,\lm]$, $g\in \h^{S_{o,o'}}\setminus B_{R(\lm)}(h)$ there is $\pi\in \Pi$ with
\begin{align*}
Q_{x,y}^{(\pi)}\leq c\qquad\mbox{or}\qquad
\cos    \al_{x,y}^{(\pi)}\leq c,
\end{align*}
either for  some  $x,y\in \pi(S_{o'})$ or for $x=o'$ and all $y\in \pi(S_{o})$.
\end{prop}

Note that the $Q_{x,y}\leq c$ is in the statement only to exclude the trivial case which is $g_{x}=h_{x}$ or $g_{y}=h_{y}$.

The proof carries some similarities to the analysis in Subsection~\ref{ss:relative_arguments}. We need some basic geometric observations. The first one is about perturbations of arguments. Recall that the argument $\arg:\C\setminus\{0\}\to\Sp^{1}$ is a continuous group homomorphism  and we have the modulus $\ma{\cdot}=d_{\Sp^{1}}(\cdot,1)$ in $\Sp^{1}$. \medskip

\begin{lemma}\label{l:arg} Let $\xi,\zeta\in\C$, $\xi\neq0$ such that $\arg\xi\in[-\pi/2,\pi/2]$ and $\mo{\zeta}<1$.
Then
\begin{align*}
\ma{\arg\ap{1+ \xi+\zeta}}\leq \ma{\arg \xi}+\frac{\mo{\zeta}}{1-\mo{\zeta}}.
\end{align*}
In particular,  for $\xi=0$,  we get
\begin{align*}
\ma{\arg\ap{1+ \zeta}}&\leq \frac{\mo{\zeta}}{1-\mo{\zeta}}.
\end{align*}
\begin{proof}
We write $\xi$, $\zeta$ in polar coordinates $\xi=r e^{i\de}$ and $\zeta=\eps e^{i\te}$. We denote the left hand side of the inequality by $\be$, i.e., $\be=\ma{\arg\ap{1+ r e^{i\de}+\eps e^{i\te}}}$. Assume without loss of generality that $\de\ge0$. Since we have for $z\in \C$ with $\Re z,\Im z\geq0$ and $u\geq0$ that $\ma{\arg{z+iu}}\geq\ma{\arg(z-iu )}$ we conclude $\be\leq\ma{\arg{(1+ r e^{i\de}+\eps e^{i|\te|})}}$. Hence, we may assume without loss of generality that $\te\ge0$. We  use $\arg z=\arctan (\Im z/\Re z)$ and subadditivity, monotonicity  on $[0,\infty)$  and  $0\leq\arctan'\leq1$ to calculate
\begin{align*}
\be&=\arctan\ap{\frac{r\sin\de+\eps\sin\te}{1+r\cos\de+\eps\cos\te}}\\
&\leq\arctan\ap{\frac{r\sin\de}{1+r\cos\de+\eps\cos\te}}+
\arctan\ap{\frac{\eps\sin\te}{1+r\cos\de+\eps\cos\te}}\\
&\leq \arctan\ap{\frac{\sin\de}{\cos\de}}+ \arctan\ap{\frac{\eps}{1-\eps}}\\
&\leq\de+ \frac{\eps}{1-\eps}.
\end{align*}
The statement about $\xi=0$ is a direct consequence from the first statement.
\end{proof}
\end{lemma}

The second auxiliary lemma deals with sums of complex numbers.\medskip

\begin{lemma}\label{l:sum} Let $\de\in[0,\pi/2]$, $\xi\in\C^{S_{o,o'}}$ with $\xi_x\neq0$ for all $x\in S_{o,o'}$  and $\arg(\xi_x\ov\xi_y)\in[-\de,\de]$ for all $\pi\in\Pi$, $x,y\in\pi(S_{o'})$. Then,
\begin{itemize}
\item [(1.)] $\arg (\xi_x\ov\xi_{y})\in[-2\de,2\de]$ for all $x,y\in S_{o,o'}$.
\item [(2.)]  $\mo{\sum_{y\in\pi(S_{o'})}\xi_{y}}\geq\mo{\xi_{x}}$ for all $\pi\in\Pi$ and $x\in\pi(S_{o'})$.
\item [(3.)] $\arg\ap{\sum_{y\in\pi(S_{o'})}\xi_{y}\ov{\xi_{x}}}\in[-2\de,2\de]$ for all $\pi\in\Pi$ and $x\in S_{o,o'}$.
\end{itemize}
\begin{proof}
The numbers $\xi_{x}$, $x\in S_{o,o'}$ shall be thought as non zero vectors in $\C$. Then, the assumption about the arguments means that they  point approximately in the same direction. The first statement follows as we can compare two elements $\xi_{x}$, $\xi_{y}$ always over a third one. The second and the third statement can be easily seen by drawing a picture. However,  let us give a precise argument.

(1.) Let $x,y\in S_{o,o'}$ and $\pi\in\Pi$ such that $x\in\pi(S_{o'})$.  Suppose $y\in\pi(S_{o})$ otherwise the statement follows directly by the assumption. Let $\pi_y$ be a permutation that permutes $y\in\pi(S_{o})$ with a vertex  $\ow y\in\pi(S_{o'})$ with $a(y)=a(\ow y)$ and leaves every other vertex invariant.
In particular,  one has $y\in\pi_{y}\circ\pi(S_{o'})$.
Let  $w\in \pi(S_{o'})$ with $w\neq\ow y$. We have by assumption that $\aM{\arg(\xi_{x}\ov\xi_{w})}\le\de$ as  $w\in \pi(S_{o'})$ and $\aM{\arg(\xi_{y}\ov\xi_{w})}\le\de$ as $w\neq \ow y$ and therefore $w\in\pi_{y}\circ\pi(S_{o'})$. Statement (1.) now follows from the triangle inequality of $\am{\cdot}$.

(2.) One calculates $|\sum_{y\in\pi(S_{o'})}\xi_{y}| =|1+\sum_{y\in\pi(S_{o'})}\xi_{y}/\xi_{x}|\mo{\xi_x}$. Now, since $\arg(\xi_{y}/\xi_{x})=\arg(\xi_x\ov\xi_y)\in[-\pi/2,\pi/2]$ the first factor is larger or equal to one.

(3.) Let first $x\in\pi(S_{o'})$ and calculate
$
\arg({\sum_{y\in\pi(S_{o'})}\xi_{y}\ov{\xi_{x}}}) =\arg({1+\sum_{y\in\pi(S_{o'})}\xi_{y}/\xi_{x}}).$
Note that the terms in the sum on the right hand side have arguments in $[-\de,\de]$ and therefore the sum has argument in $[-\de,\de]$. By Lemma~\ref{l:arg} applied with $\xi$ equal to the sum and $\zeta=0$ yields the statement.
If on the other hand $x\in\pi(S_{o})$, we find by a similar argument as in (2.), a vertex $w\in\pi(S_{o'})$ with $\aM{\arg(\xi_{x}\ov \xi_{w})}\leq\de$. By the triangle inequality for $\am{\cdot}$ and the considerations of the first case, we get,
$\am{\arg({\sum_{y\in\pi(S_{o'})}\xi_{y}\ov{\xi_{x}}}) }\leq \am{\arg({\sum_{y\in\pi(S_{o'})}\xi_{y}\ov{\xi_{w}}}) }+ \am{\arg(\xi_{w}\ov{\xi_{x}})} \leq2\de.
$
\end{proof}
\end{lemma}

We are now in the position to prove Proposition~\ref{p:al>0}.

\begin{proof}[Proof of Proposition~\ref{p:al>0}]
Let $R:[0,\lm_0)\to[0,\infty)$ with $\lm_{0}:={\eps_0\de_0}/\ap{1+\de_0}$ be defined as
\begin{align*}\label{e:R}
R(\lm)=\frac{\de(\lm)^2}{(\eps_0-\de(\lm))\eps_0},\quad\mbox{with } \de(\lm)=\frac{(1+\de_0)}{\de_0}\frac{\lm}{\min_{y\in S_{o'}}|t(o',y)|^2}.
\end{align*}
Take $g\in \h^{S_{o,o'}}\setminus B_{R(\lm)}(h)$. If there is $x\in S_{o,o'}$ such that $g_{x}=h_{x}$ then $Q_{x,y}=0$ by definition for all $y\in S_{o,o'}$. In this case we are done. Therefore,  assume that $g_{x}\neq h_{x}$ for all $x\in S_{o,o'}$.
Moreover, assume $\aM{\al_{x,y}}\leq\de_0$ for all $\pi\in\Pi$ and $x,y\in\pi(S_{o'})$ since otherwise we are also done.

Our aim is to show $\ma{{\al_{o',y}^{(\pi)}}}>\de_0$ for some $\pi\in\Pi$ and all $y\in\pi(S_{o})$.
We start with a claim.

\emph{Claim 1: There is  $\pi\in\Pi$ such that
\begin{align*}
\mo{\sum_{y\in\pi(S_{o'})}\mo{t(o',y)}^2\ap{g_{y}-h_{y}}} \geq \frac{(1+\de_0)}{\de_0}\lm.
\end{align*}}
Proof of Claim~1.
By Lemma~\ref{l:sum}~(2.) we get for all $\pi\in \Pi$ and $x\in \pi(S_{o'})$
\begin{align*}
\mo{\sum_{y\in\pi(S_{o'})}\mo{t(o',y)}^2\ap{ g_{y}-h_{y}}}\geq \mo{t(o',x)}^2\mo{ g_{x}-h_{x}}.
\end{align*}
Moreover, we assumed $g\not\in B_{R(\lm)}(h)$, so there is $\oh x\in S_{o,o'}$ such that $\gm_{\oh x}\geq R(\lm)$. By the definition  $\eps_{1}(r)=\inf _{g\in \h^{S_{o,o'}}\setminus B_r(h)}\min_{x\in S_{o,o'}}\mo{g_{x}-h_{x}}$ for $r\geq0$ from Subsection~\ref{ss:eps1},  the definition of $R$ above and Lemma~\ref{l:eta}~(3.) we get
\begin{align*}
\mo{g_{\oh x}-h_{\oh x}}\geq \eps_{1}(R(\lm))=\de(\lm)=\frac{(1+\de_0)}{\de_0}\frac{\lm}{\min_{y\in S_{o'}}|t(o',y)|^2}.
\end{align*}
By choosing $\pi\in\Pi$ such that $\oh x\in\pi(S_{o'})$ we conclude the claim.

\emph{Claim 2: There is  $\pi\in\Pi$ such that for all $x\in S_{o,o'}$}
\begin{align*}
\aM{\arg{\ap{\sum_{y\in\pi(S_{o'})}\mo{t(o',y)}^2\ap{g_{y}-h_{y}}-v} \ov{(g_{x}-h_{x})}} }\leq 3\de_0.
\end{align*}
Proof of Claim~2. Set $\tau:=\sum_{y\in\pi(S_{o'})}\mo{t(o',y)}^2\ap{g_{y}-h_{y}}$. Note that $|v/\tau|\leq \de_0/(1+\de_0)$  by the assumption $v\in[-\lm,\lm]$ and Claim~1. We get, by the triangle inequality, Lemma~\ref{l:arg} (applied with $\xi=0$ and $\zeta=v/\tau$)  and Lemma~\ref{l:sum}~(3.)
\begin{align*}
\aM{\arg{\ap{\tau-v}\ov{(g_{x}-h_{x})}} }=\aM{\arg{\ap{1-v/\tau}} }+\aM{\arg{\ap{\tau\ov{(g_{x}-h_{x})}}}}\leq3\de_0.
\end{align*}

By the definitions $g_{o'}^{(\pi)}=\Psi_{z-v,o'}^{(T)}\ap{(g_{\pi(x)})_{x\in S_{o'}}}$ and $h_{o'}=\Psi_{z,o'}^{(T)}\ap{h_{S_{o'}}}$, we get  for $\pi\in \Pi$
\begin{align*}
\arg\ap{g_{o'}^{(\pi)}-h_{o'}} &=\arg\ap{\frac{-1}{z-v+\sum_{x\in\pi(S_{o'})}|t(o',x)|^2 g_{x}}-\frac{-1}{z+\sum_{x\in\pi(S_{o'})}|t(o',x)|^2 h_{x}}}\\
&=\arg\ap{g_{o'}^{(\pi)}h_{o'}
\ap{\sum_{y\in\pi(S_{o'})}\mo{t(o',y)}^2\ap{g_{y}-h_{y}} -v}}.
\end{align*}
This calculation is similar to the proof of Lemma~\ref{l:taual}.
By definition of $\de_0$ we have that
$\arg h_{o'}\in[4\de_0,\pi-4\de_0]$.
Moreover, since $g_{o}^{(\pi)}\in\h$ we have $\arg g_{o}^{(\pi)}\in(0,\pi)$. Hence, $\am{\arg({g_{o'}^{(\pi)}h_{o'}})}>4\de_0$.
Combining this with  the equality above, Claim~2 and the triangle inequality we get for $\pi$ taken from Claim~2
\begin{align*}
\aM{\al_{o',x}^{(\pi)}}&\geq \aM{\arg({g_{o'}^{(\pi)}h_{o'}})} -\aM{\arg{\ap{\sum_{y\in\pi(S_{o'})}\mo{t(o',y)}^2\ap{g_{y}-h_{y}}-v} \ov{(g_{x}-h_{x})}}}\\
&\geq\de_0
\end{align*}
 for all $x\in\pi(S_{o})$.
The assertion follows by letting $c:=\cos\de_0$.
\end{proof}

\subsection{Proof}\label{ss:kaproof}
In this final subsection we put the pieces together in order to prove Proposition~\ref{p:ka}.

\begin{proof}[Proof of Proposition~\ref{p:ka}]
Let $I\subset\Sigma$ be compact, $h=\Gm_{S_{o,o'}}(z,T)$ for $z\in I+i[0,1]$ and $\eps_0=\eps_0(I)$, $\de_0=\de_0(I)$ defined in Subsection~\ref{ss:Case2} and~\ref{ss:Case3}. Moreover, let $R:[0,\lm_0)\to[0,\infty)$ and $\lm_0>0$ be given by Proposition~\ref{p:al>0}. Choose $\de_1\in(0,c_1)$, where $c_1$ is defined in Subsection~\ref{ss:Case1} and $\de_2\in(0,\eps_0 \de_1)$.

Let $\lm\in[0,\lm_0)$, $g\in \h^{S_{o,o'}}\setminus B_{R(\lm)}(h)$, $v\in[-\lm,\lm]$  and $z\in I+i[0,1]$. We now treat the three cases which we already distinguished above:

\emph{Case~1: $\Vis_{\gm}(g,\de_1)\neq S_{o,o'}$.}
The statement  follows directly from Proposition~\ref{p:invisgm}.

\emph{Case~2: $\Vis_{\gm}(g,\de_1)=S_{o,o'}$ but $\Vis_{\Im}^{o'}(\pi,g,\de_2)\neq\pi(S_{o'})$ for some $\pi\in\Pi$.}
The statement  follows by the combination of Proposition~\ref{p:invisIm} and Lemma~\ref{l:using_c<1}.

\emph{Case~3: $\Vis_{\gm}(g,\de_1)=S_{o,o'}$ and $\Vis_{\Im}^{o'}(\pi,g,\de_2)=\pi(S_{o'})$ for all $\pi\in\Pi$.} We can now find  $\pi$, $i$, $\ow x$, $\ow y$ which satisfy the assumptions of Lemma~\ref{l:using_alQ<1} as follows: By Proposition~\ref{p:al>0}  there is $\pi\in\Pi$ such that  either for some $x,y\in\pi(S_{o'})$ or $x=o'$ and all $y\in \pi(S_{o})$ we have $Q_{x,y}^{(\pi)}\leq c$ or $\cos\al_{x,y}^{(\pi)}\leq c$.\\
In the first case we have $y\in\Vis_{\Im}^{o'}(\pi,g,\de_2)$ by assumption. We choose consequently $i=o'$, $\ow x=x$ and $\ow y=y$. \\
For the second case let $w\in \Vis_{\Im}^{o}(\pi,g,\de_2)$. We set $i=o$, $\ow y=w$ and $\ow x=o'$ if $w\neq o'$ or pick $\ow x\in\pi(S_{o})\setminus\{o'\}$ arbitrary if $w=o'$.\\
With these choices we see by Lemma~\ref{l:using_alQ<1} that the assumptions of Lemma~\ref{l:using_c<1} are satisfied and the statement follows.
\end{proof}

\section{Off~diagonal perturbations}\label{s:offdiagonal}

In this section we take a look at operators whose off diagonal elements are perturbed by a small random quantity. Indeed, a similar result as Theorem~\ref{main3} can be proven.

Such a model was first introduced by \cite{Ham} (see also \cite{Kes} for a survey) under the name \emph{first passage percolation}. There, a   graph is studied  whose edge weights are given by independent random variables. The edge weights are interpreted as passage times and then transit times are studied.  Here, we want to look at the spectral properties of the corresponding operators rather than transit times.

Let us be more precise about our model. For  a vertex
$x\in\V$  which is not the root $o\in\V$, let $\dot{x}\in\V$ be the unique vertex which precedes $x$ with respect to the root. Let $\lm\geq0$ and $v\in\W_{\mathrm{rand}}(\Om,\T)$. We define by $\te(\lm)=(\te_{e}(\lm))_{e\in\E}$ the random variables given by
\begin{align*}
\te_{\{\dot{x},x\}}(\lm,{\om})=(1+\lm v_{x}^{\om}),\quad\om\in\Om.
\end{align*}
For a given label invariant operator $T$ we define the operator $H^{\te(\lm,\om)}:\ell^{2}(\V,\nu)\to\ell^{2}(\V,\nu)$ by
\begin{align*}
(H^{\te(\lm,\om)}\ph)(x)=\sum_{y\sim x}\te_{\{x,y\}}(\lm,{\om})t(x,y)\ph(y)
+w(x)\ph(x).
\end{align*}

The following result for these operators with random off diagonal perturbations is similar to Theorem~\ref{main3}  for random diagonal perturbations, i.e., random potentials. We will sketch a proof below.\medskip

\begin{thm}\label{t:offdiag} Let $T$ be a label invariant operator.  There exists a finite subset $\Sigma_0\subset\si(T)$ such that for every compact set $I\subseteq \si(T) \setminus \Sigma_0$ there exists $\lm_0>0$ such that for all $\lm\in[0,\lm_0]$ and almost every $\om\in\Om$
\begin{align*}
I\subseteq \si_{\mathrm{ac}}(H^{\te(\lm,\om)}) \qqand I\cap \si_{\mathrm{sing}}(H^{\te(\lm,\om)})=\emptyset.
\end{align*}
\end{thm}

For the truncated Green function of the operators $H^{\te(\lm,\om)}$, the recursion formula \eqref{e:Gm} reads as
\begin{align*}
-\frac{1}{\Gm_{x}\ap{z,H^{\te(\lm,\om)}}}=z-m_{a(x)}+\sum_{y\in S_{x}}|\te_{\{x,y\}}(\lm,{\om})|^{2}|t(x,y)|^{2} \Gm_{y}\ap{z,H^{\te(\lm,\om)}}.
\end{align*}
Define for $z\in\h$, $\lm\geq0$  and  $\om\in\Om$
\begin{align*}
&h_x:=\Gm_{x}(z,T),&& x\in S_{o}\cup S_{o'},    \\
&g_{x}(\lm,\om):=|\te_{{i,x}}(\lm,\om)|^{2} \Gm_{x}(z,H^{\te(\lm,\om)}),\quad &&x\in S_{i},\,i\in\{o,o'\}
\end{align*}
and denote $h=(h_{x})_{x\in S_{o,o'}}$, $g(\lm,\om)=(g_{x}(\lm,\om))_{x\in S_{o,o'}}$.
With these definitions we have
\begin{align*}
\Gm_{o'}\ap{z,H^{\te(\lm,\om)}}=\Psi^{(T)}_{z,o'}\ap{g_{S_{o'}(\lm,\om)}}.
\end{align*}
We get the following two step expansion formula similar to Proposition~\ref{p:expansion}.
\medskip

\begin{prop}\label{p:expansion3}(Two step expansion - off~diagonal version.)
Let $I\subset \Sigma$ be compact and $o\in\V$. Then there exist $c,C:[0,\infty)\to[0,\infty)$ with $c(\lm), C(\lm)\to 0$ for $\lm\to0$ such that for all $z\in I+i[0,1]$, $\lm\in[0,\infty)$, $\om\in\Om$ we have
\begin{align*}
{\gm\ap{\Gm_{o}\ap{z,H^{\te(\lm,\om)}},\Gm_{o}(z,T)}} &\leq(1+c(\lm))Z_{0}(z,0,g(\lm,\om))+C(\lm).
\end{align*}
\begin{proof}
We follow the proof of Proposition~\ref{p:expansion} with only one difference. Whenever the  first inequality of Lemma~\ref{l:ti} is used, we employ the second  inequality of Lemma~\ref{l:ti} instead.
\end{proof}
\end{prop}

The following estimate for the contraction coefficient is a corollary of Proposition~\ref{p:ka}.\medskip

\begin{prop}\label{p:kaoffdiag}
Let $I\in \Sigma$ be compact, $p>1$ and $o\in \V$. There exists $\de=\de_o(I)>0$ such that for all $z\in I+i(0,1]$ and $\lm\in[0,1)$ and for all $\om\in\Om$
\begin{align*}
\ka_{o}^{(p)}(z,0,g(\lm,\om))\leq 1-\de.
\end{align*}
\begin{proof}
The statement directly follows by Proposition~\ref{p:ka}. Note that since there is no potential and $R(0)=0$ for the function $R$ of Proposition~\ref{p:ka}, we even do not have to exclude a ball.
\end{proof}
\end{prop}

Putting the statements of these two propositions into the calculations of the proof of the vector  inequality, Proposition~\ref{p:vector}, we obtain a similar result. In particular,  for all  $I\subset\Sigma$ compact and $p>1$ there are $\lm_0>0$ and $\de>0$ such that for all $\lm\in[0,\lm_0)$
\begin{align*}
\EE\gm\leq (1-\de)P\;\EE\gm+C(\lm),
\end{align*}
where $C(\lm)\to0$ as $\lm\to0$,
\begin{align*}
\EE\gm&:=\int_{\Om}\ap{ \gm\ap{\Gm_{o(j)}(z,H^{\te(\lm,\om)}), \Gm_{o(j)}(z,T)}^{p}}_{j\in\A}d\PP(\om).
\end{align*}
and $P:\A\times\A\to[0,\infty)$ as it is defined above Proposition~\ref{p:vector}.
Now, by the Perron-Frobenius argument, we get a theorem similar to Theorem~\ref{t:EGm}.\medskip

\begin{thm}\label{t:EGm2}
Let $I\subset\Sigma$ be compact and $p> 1$. Then there exist $\lm_0=\lm_0(I,p)>0$ and $c:[0,\lm_0)\to[0,\infty)$ monotone decreasing with $c(\lm)\to0$ for $\lm\to0$ such that for all $\lm\in[0,\lm_0)$
\begin{equation*}
\sup_{x\in\V}\sup_{E\in I}\sup_{\eta\in(0,1]}\int_{\Om}{\gm\ap{\Gm_x(E+i\eta, H^{\te(\lm,\om)}),\Gm_x(E+i\eta, T)}^p}d\PP(\om)\leq c(\lm).
\end{equation*}
\end{thm}

Application of Proposition~\ref{p:G}, Fatou's lemma and Fubini's theorem yields the inequality
\begin{align*}
\liminf_{\eta\downarrow 0}\int_I \mo{G_x(E+i\eta, H^{\te(\lm,\om)})}^pdE <\infty,
\end{align*}
for $p>1$, $I\subset\Sigma$ compact, $\lm>0$ sufficiently small and almost all $\om\in\Om$. Moreover, we can derive $\liminf_{\eta\downarrow 0}\Im G_{x}(E+i\eta,H^{\te(\lm,\om)})>0$ for $(\om,E)$ on a set of full $\PP\times\Leb$-measure.
By the vague convergence of the spectral measures, Lemma~\ref{l:mu}, and the criterion for the absence of singular spectrum, Theorem~\ref{t:Klein}, we conclude the statement of Theorem~\ref{t:offdiag}.

\section{Open problems and remarks}
We have proven in this chapter that the absolutely continuous spectrum of a label invariant operator remains stable on certain  subsets under sufficiently small random perturbations. We gave the proof in full detail for random potentials. Moreover, in the previous section we showed that our method also applies to off diagonal perturbations. Therefore,  it is natural to ask to what other models our method might apply as well.

One of these models are Galton-Watson trees. In a multi-type Galton-Watson tree the number of forward vertices for a vertex  of a certain type is given by a random variable. These random variables are independent in every vertex. Moreover, they are identically distributed in each vertex with the same type.
To each realization we associate a nearest neighbor operator, for instance the Laplacian $\Delta$. This way we get a family of random operators. Note that the stationary case, where one realization occurs with probability one, can be considered as a tree generated by a substitution matrix.
It is clear that if one allows  with positive probability for dead ends, i.e, vertices with no forward neighbors, then, one has to expect plenty of point spectrum. However,  if one excludes this case, one should be able to answer the following question by our methods.\medskip

\begin{question}
Does a nearest neighbor operator of a multi-type Galton-Watson tree with no dead ends have pure absolutely continuous spectrum if it is close to the stationary case in distribution, i.e., to a tree which is generated by a substitution matrix?
\end{question}

As mentioned above, the case where one allows for vertices to have no forward neighbors is much more involved. We can describe a special case as a percolation model. Let $\T=(\V,\E)$ be a tree generated by a substitution matrix and a random variable on $\E$ which deletes an edge with probability $p\in[0,1]$ and keeps it with probability $1-p$. As mentioned above, we cannot expect pure absolutely continuous spectrum in this case. However,  one might ask the following question:\medskip

\begin{question}
Does a nearest neighbor operator of a percolation tree have absolutely continuous spectrum, if $p$ is close to zero?
\end{question}

Let us turn to some questions regarding large disorder.
By the fractional moments method of Aizenman/Molchanov \cite{AM} it is not too hard to prove that for large perturbations by a random potential one has pure point spectrum  almost surely. Asking the same question for large off diagonal perturbations seems to be more involved since the perturbations are not monotone anymore. We want to ask three questions.

The first concerns the model of first passage percolation discussed in the previous section. In contrast to random trees the number of forward neighbors is fixed  and we only perturb the weights randomly.
\medskip

\begin{question}
Does a nearest neighbor operator in a first passage percolation model have almost surely pure singular or even pure point spectrum, if the perturbation is large enough?
\end{question}

We may also ask the question for Galton-Watson trees with no dead ends.\medskip

\begin{question}
Does a nearest neighbor operator of a Galton-Watson tree with no dead ends have almost surely pure singular or even pure point spectrum if it is far away from the stationary case in distribution?
\end{question}

Finally,  we want to ask the question about the percolation model. As already mentioned, there will be plenty of point spectrum in this case. Moreover, it is clear that in the subcritical case there are only compactly supported eigenfunctions since all components are finite with probability one.
\medskip

\begin{question}
Can one exclude continuous spectrum almost surely for a nearest neighbor operator of a percolation tree if $p$ is not close to one?
\end{question}

%% file: cv.tex
\pagestyle{empty}\begin{center}\section*{Curriculum Vitae - Matthias Keller}
\end{center}

\subsection*{Personal and contact information}
\begin{quote}
 Department of Mathematics \\Friedrich Schiller University \\
 Ernst-Abbe-Platz 2 \\
 D-07743 Jena, Germany\\
{Office:} +49-(0)3641-9-46135\\
{E-mail:} \href{mailto:m.keller@uni-jena.de}{m.keller@uni-jena.de}

{Date of birth:}  December 31, 1980\\
{Country of Citizenship:} Germany 
\end{quote}

\subsection*{Present position}
\begin{quote}Research and teaching assistant, Friedrich Schiller University Jena
\end{quote}

\subsection*{Education}
\begin{quote}
\textbf{Ph.D. in Mathematics},\\
December 2010, Friedrich Schiller University Jena\\
Title: On the spectral theory of operators on trees\\
1st Advisor: Prof. Dr. Daniel Lenz\\
2nd Advisor: Prof. Dr. Simone Warzel\medskip

\textbf{Diploma in Mathematics},\\ June 2006,
Chemnitz University of Technology,
Germany\\
Title: Convergence Properties of Products of Random Matrices 
\\
Advisors: PD Dr. Daniel Lenz and Prof. Dr. Peter Stollmann
\end{quote}

\subsection*{Academic activity}
\begin{quote}
2007-2008 Visiting  Student Research Collaborator, Princeton University

2006-2007 Research assistant, Chemnitz University of Technology 

2004-2006 Student assistant, Chemnitz University of Technology

2001-2004 Student assistant, Fraunhofer Institute Chemnitz
\end{quote}

\subsection*{Awards}
\begin{quote}
2007-2010 Klaus Murmann Fellowship Programme (sdw)
\end{quote}